\documentclass[12pt,reqno]{amsart}
\usepackage{amsmath}
\usepackage{amssymb}
\usepackage{amstext}
\usepackage{mathrsfs}
\usepackage{a4wide}
\usepackage{graphicx}
\usepackage{bbm}
\usepackage{verbatim}
\usepackage{bm}
\usepackage{graphicx}
\usepackage{tikz}
\usepackage{pgfplots}
\usepackage{titletoc}
\usepackage[T1]{fontenc}

\allowdisplaybreaks \numberwithin{equation}{section}
\usepackage{color}
\usepackage{cases}
\usepackage{hyperref}

\numberwithin{equation}{section}

\newtheorem{theorem}{Theorem}[section]
\newtheorem{proposition}[theorem]{Proposition}

\newtheorem{lemma}[theorem]{Lemma}

\theoremstyle{definition}

\theoremstyle{remark}
\newtheorem{remark}[theorem]{Remark}

\newcommand{\ep}{\varepsilon}
\newcommand{\Om}{\Omega}

\begin{document}
		
	\title[3D Euler equations with helical symmetry]{On concentrated vortices of 3D incompressible Euler equations under  helical symmetry: with swirl}
	
	\author{Guolin Qin, Jie Wan}

	\address{School of Mathematical Science, Peking University, Beijing 100871, People's Republic
		of China}
	\email{qinguolin18@mails.ucas.ac.cn}

	\address{School of Mathematics and Statistics, Beijing Institute of Technology, Beijing 100081,  P.R. China}
	\email{wanjie@bit.edu.cn}

	\begin{abstract}
In this paper, we consider the existence of concentrated  helical vortices of 3D incompressible Euler equations with swirl. First, without the assumption of the orthogonality condition, we derive a  2D vorticity-stream formulation of  3D incompressible Euler equations under helical symmetry. Then based on this system, we  deduce a non-autonomous second order semilinear elliptic equations in divergence form, whose solutions correspond to traveling-rotating invariant helical vortices  with non-zero helical swirl. Finally,   by using Arnold's variational method, that is, finding maximizers of a properly defined energy functional over a certain function space and proving the asymptotic behavior of maximizers, we construct   families of concentrated traveling-rotating helical vortices of 3D incompressible Euler equations with non-zero  helical swirl in infinite cylinders. As parameter $ \varepsilon\to0 $, the associated vorticity fields   tend asymptotically to a singular helical vortex filament evolved by the binormal curvature flow. 
	\end{abstract}
	
	 \maketitle{\small{\bf Keywords:} Incompressible Euler equations; Helical symmetry; With swirl; Variational methods; Semilinear elliptic equations.   \\	
	 	
		 }

	\section{Introduction and Main results}\label{sec1}
In this paper, we study the concentrated helical vortices of an incompressible ideal fluid described by the Euler equations. We are interested in the existence of  vorticity field with helical symmetry which is concentrated in a helical tube of radius $ \varepsilon\ll 1 $ around a helical curve in $ \mathbb{R}^3 $, when the orthogonality condition fails.
The motion of an inviscid incompressible flow  is   governed by the following Euler equations
	 \begin{equation}\label{1-1}
	 	\begin{cases}
	 		\partial_t\mathbf{v}+(\mathbf{v}\cdot \nabla)\mathbf{v}=-\nabla P,\ \ &\text{in}\ \  \Om\times (0,T),\\
	 		\nabla\cdot \mathbf{v}=0,\ \ &\text{in}\ \ \Om\times (0,T),\\ 
	 	\mathbf{v}\cdot \mathbf{n}=0,\ \ &\text{on}\ \ \partial \Om\times (0,T),\\
	 	\mathbf{v}(\cdot, 0)=\mathbf{v}_0(\cdot), \ \ &\text{in}\ \ \Om,
	 	\end{cases}
	 \end{equation}
	 where the domain $\Om\subseteq \mathbb{R}^3 $, $ \mathbf{v}=(v_1,v_2,v_3) $ is the velocity field, $ P$ is the scalar pressure and  $ \mathbf{n} $ is the outward unit normal to $ \partial \Om$. For a solution  $ \mathbf{v} $ of \eqref{1-1}, define the corresponding vorticity field of the fluid $$\mathbf{w} :=\nabla\times\mathbf{v}.$$   
	
Taking curl of the first equation in \eqref{1-1},  we obtain the following equations for vorticity:
	 \begin{align}\label{1-2}
	 	\begin{cases}
	 		\partial_t \mathbf{w} +(\mathbf{v}\cdot\nabla)\mathbf{w} =(\mathbf{w} \cdot\nabla)\mathbf{v},
	 		\ \ &\text{in}\ \   \Om\times (0,T),\\
	 		\mathbf{v}\cdot \mathbf{n}=0,\ \ &\text{on}\ \  \partial \Om\times (0,T)\\
	 		\mathbf{w} (\cdot, 0)=\nabla\times \mathbf{v}_0(\cdot), \ \ &\text{in}\ \  \Om.
	 	\end{cases}
	 \end{align}

 \subsection{Backgrounds}

The study of the evolution of vorticity fields of \eqref{1-2} concentrated in a tube with small cross-section  around a smooth curve can be traced back to Helmholtz's classical result in 1858 \cite{Hel}. 
 Generally speaking, there are two main concerns: 
\begin{enumerate}
\item[$\bullet$] How does the curve evolved?\\
\item[$\bullet$] How to give a rigorous justification of the relation between vorticity field of \eqref{1-2} and the curve?
\end{enumerate}

The first question  seems to have been well studied.  Helmholtz \cite{Hel} first found that vorticity fields of \eqref{1-2} with axial symmetry (called the vortex ring)  supported in a toroidal region with small cross-section near a circular curve have an approximately steady form and travel with a large constant velocity along the axis of the ring. Then, Lord Kelvin and Hicks \cite{Lam,Tho} further gave a quantitative formula for the velocity of a vortex ring with a small cross-section, which confirmed Helmholtz's observation.  Let $ l $ be any oriented closed curve with tangent vector field $\mathbf{t}$ that encircles the aforementioned vorticity region   once and $ \sigma $ is any surface with $ l $ as its boundary. Define the circulation of a vortex  
   \begin{equation}\label{1001}
 	\kappa=\oint_l\mathbf{v}\cdot \mathbf{t}dl=\iint_{\sigma}\mathbf{w} \cdot\mathbf{n}d\sigma.
 \end{equation}
It is  shown in \cite{Lam,Tho} that vortex rings concentrated around a torus with fixed radius $r^*$, circulation $\kappa$ and a small, nearly singular cross-section of radius $0<\varepsilon\ll 1$  would approximately move at the velocity 
   \begin{equation}\label{1-3}
   	\frac{\kappa}{4\pi r^*}\left(\ln\frac{8r^*}{\varepsilon}-\frac{1}{4}\right).
   \end{equation}
For general curves, the first derivation of the motion law of curves dates back to the work of da Rios  in 1906 \cite{DR} and  Levi-Civita  in 1908 \cite{LC,LC2} with the help of potential theory. Consider a vortex tube with a small section of radius $\ep$ and a fixed circulation $ \kappa $, uniformly distributed around an evolving smooth curve $\Gamma (t) $. By using the Biot-Savart law and calculating the leading order of instantaneous velocity of the curve, \cite{DR} indicates that the  curve $\Gamma (t) $ evolves by its \textit{binormal curvature flow}: let $\Gamma(t)$ be parameterized by $\gamma(s,t)$ with $s$ the arc length parameter, then the motion of $\gamma$ would asymptotically  obey a law of the form
    \begin{equation}\label{1-4}
    	\partial_t \gamma=\frac{\kappa}{4\pi}|\ln\varepsilon|(\partial_s\gamma\times\partial_{ss}\gamma)=\frac{\kappa\bar{K}}{4\pi}|\ln\varepsilon|\mathbf{b}_{\gamma(t)},
    \end{equation}
where $ \mathbf{b}_{\gamma(t)}  $ is the binormal unit vector and $ \bar{K} $ is its local curvature. We  refer interested readers to \cite{FM} for a generalization of \eqref{1-4}. By applying the scale $ t = | \ln\varepsilon|^{-1}\tau $,   \eqref{1-4} is equivalent to
    \begin{equation}\label{1-5}
    	\partial_\tau \gamma=\frac{\kappa \bar{K}}{4\pi}\mathbf{b}_{\gamma(\tau)}.
    \end{equation}
\eqref{1-4} is known by many names, including the vortex filament equation, and the local induction approximation. In subsequent years, the binormal curvature flow was rediscovered several times, see for example Arms and Hama \cite{AH} and Moore and Saffman \cite{MS}. More discussions about \eqref{1-4} can be seen in  the texts \cite{MB,Ric,Ric2,S} and references
therein.

The second problem, that is, a rigorous mathematical justification of the relation between vorticity field of \eqref{1-2} and the curve \eqref{1-4}, is still a very concerned problem and has not been completely solved. Jerrard and Seis \cite{JS} first justified  da Rios' computation mathematically under some assumptions on a solution to \eqref{1-2} which remains suitably concentrated
around an evolved vortex filament. It is shown in \cite{JS} that under certain conditions, if there exists a family of solutions $ \mathbf{w}_\varepsilon $ to \eqref{1-2} such that in the sense of distribution,
\begin{equation}\label{1-6}
	\mathbf{w} _\varepsilon(\cdot,|\ln\varepsilon|^{-1}\tau)\to \kappa \delta_{\gamma(\tau)}\mathbf{t}_{\gamma(\tau)},\ \ \text{as}\ \varepsilon\to0,
\end{equation}
where  $ \mathbf{t}_{\gamma(\tau)} $ is the tangent unit vector of $ \gamma $ and $ \delta_{\gamma(\tau)} $ is the uniform Dirac measure on the curve, then $ \gamma(\tau) $ satisfies \eqref{1-5}. See also \cite{JS2}. 
On the other hand, for a general curve that solves \eqref{1-4},  finding a family of true solutions of \eqref{1-2} that satisfy \eqref{1-6} remains a long-standing open problem, which is commonly known as the \emph{vortex filament conjecture}. Until now, the only rigorous result in favor of the vortex filament conjecture for Euler flows is restricted to flows with several kinds of symmetry structures: translation invariance, axial symmetry and helical symmetry, which correspond to the cases where the curve $ \gamma $ is a straight line, a traveling circle  and a traveling-rotating helix, respectively. 
In the case of straight lines, one always reduces the problem to the vortex desingularization problem of  point vortices in planar Euler equations, which was investigated in much literature, see e.g. \cite{CF, CLW, CPY, DDMW, Gar2, Gac2, Has2, Has3,  HM, SV, T} and the references therein. For traveling circles, Fraenkel \cite{Fra1,Fra2} first provided a  constructive proof of the existence of a vortex ring concentrated around a torus with fixed radius $r^*$ and a small cross-section of radius $\varepsilon>0$, traveling with constant speed $\sim|\ln\ep|$, rigorously varifying the results of Helmholtz \cite{Hel}. We refer readers to \cite{ALW, CWWZ,VS, BF1, FT,Ni,Nor72,YJ}  for more results on the existence of vortex rings.

The third particular curve satisfying \eqref{1-4} is the traveling-rotating helix parameterized as
\begin{equation}\label{1007}
	\gamma(s,\tau)=\left( r_*\cos\left( \frac{-s-a_1\tau}{\sqrt{h^2+r_*^2}}\right),  r_*\sin\left( \frac{-s-a_1\tau}{\sqrt{h^2+r_*^2}}\right), \frac{hs-b_1\tau}{\sqrt{h^2+r_*^2}}\right)^t,
\end{equation}
where $ r_* > 0$ is the constant  characterizing  the distance between a point in $ \gamma(\tau) $ and the $ x_3 $-axis and $h\neq 0 $  stands for the pitch of the helix, and
\begin{equation*}
	a_1=\frac{\kappa h}{4\pi (h^2+r_*^2)}, \ b_1=\frac{\kappa r_*^2}{4\pi (h^2+r_*^2)}.
\end{equation*}
Many articles concern  the global well-posedness of helical solutions of \eqref{1-1} and \eqref{1-2} under suitable functional spaces, see, e.g., \cite{AS, BLN,Du,ET, GZ, JLN, Jiu1}. Recently, the vortex filament conjecture was verified to be true for  traveling-rotating helices by D$\acute{\text{a}}$vila et al. \cite{DDMW2}, where a family of smooth traveling-rotating helical vortices of \eqref{1-2} concentrated around a helical filament  parameterized by \eqref{1007} in $ \mathbb{R}^3 $ in the sense of \eqref{1-6} were constructed via  the  inner-outer gluing method.  
More results of traveling-rotating helical vortices of \eqref{1-2} with small cross-section concentrated around the helix \eqref{1007} in infinite cylinders were constructed in \cite{CW} through the Lyapunov-Schmidt reduction method and \cite{CL23, CW2} by a variational method. For more results, see, e.g., \cite{GM}. 

We remark that in all existing results \cite{CFL, CL23, CW, CW2, DDMW2, GM},  constructions of traveling-rotating invariant helical vortices were proved under the assumption of the \textit{``without swirl'' condition}: 
\begin{equation}\label{1009}
	 \mathbf{v} \cdot \vec{\xi} \equiv 0,
\end{equation}
where $	\vec{\xi}(x)= (x_2, -x_1, h)^t. $  
The assumption \eqref{1009} plays an essential role in aforementioned works, which makes it possible to reduce the vorticity equations \eqref{1-2} to a 2D evolved model for some scalar function $ w $ and then construct helical vortices by solving an autonomous semilinear elliptic equation in a planar domain, see also subsection 1.2. 
Note that under \eqref{1009},  the vorticity field and the velocity field are orthogonal at every point, so \eqref{1009} is also called the  \textit{orthogonality condition}, see \cite{ET}.
However,   the assumption \eqref{1009} appears to be artificial as many helical symmetric vorticity fields may not satisfy this condition, such as Beltrami flows, see \cite{EP,EP2}. So here comes a natural question: when \eqref{1009} fails, is it possible to construct a family of  helical vorticity fields  of \eqref{1-2} which concentrates in a helical tube with small cross-section of radius $ \ep $ around the  helix \eqref{1007}?

In fact, there is some evidence suggesting this possibility 
 when the orthogonality condition \eqref{1009} is not satisfied. From a physical perspective, it has been formally demonstrated in \cite{FM} that when the velocity field and the vorticity field are not orthogonal,   the instantaneous velocity of the vortex tube with cross-section of radius $ \ep $ still satisfies the vortex filament equation \eqref{1-4}, as referenced in (3.1), \cite{FM}. The key  lies in decomposing the vorticity into ''axial vorticity'' along the curve and ''transverse vorticity'' perpendicular to the curve, and proving formally that the transverse vorticity makes no contribution to the induced velocity. Mathematically, as mentioned earlier \cite{JS} provided sufficient conditions for the curve where the vorticity field concentrates to satisfy the vortex filament equation, as stated in Theorem 2, \cite{JS}. It is noted that the proof of Theorem 2 does not require the assumption of the orthogonality condition, thus making it possible to construct  vortex solutions of \eqref{1-2} concentrated around a curve satisfying \eqref{1-4} even when the vorticity and velocity fields are not orthogonal. The third argument is that in the study of   vortex rings, \cite{CZ,Tur89} have constructed vortex rings with small cross-section of radius  $ \ep $ concentrating around a traveling circles, and the corresponding velocity field is not orthogonal to the vorticity field. In \cite{Abe22}, using the critical point theory the author proved the existence of vortex rings satisfying the Beltrami flow condition. Here the Beltrami flow   condition  means that the vorticity field and the velocity field are parallel everywhere, so the corresponding velocity field must not be orthogonal to the vorticity field. To conclude, it is possible to construct helical vortices of \eqref{1-2} concentrating around the helix \eqref{1007}, when \eqref{1009} is not valid.

In this paper,  without assuming  \eqref{1009}, we will construct  helical vortices  of \eqref{1-2} concentrating in a helical tube with cross-section of radius $ \ep $ around the  helix \eqref{1007}. Our idea is as follows. First, using the helical symmetry structure we derive a  2D vorticity-stream evolved model \eqref{1-23} of  3D incompressible Euler equations \eqref{1-2}, when the orthogonality condition \eqref{1009} does not hold, see  Theorem \ref{thm1}. 
Second, using the rotational invariance of solutions, we deduce a non-autonomous second order semilinear elliptic equations in divergence form, whose solutions correspond to traveling-rotating invariant helical vortices  with swirl, see \eqref{3-9} and \eqref{3-10} in section 3. Finally,   by using Arnold's variational method,   we construct    concentrated traveling-rotating invariant helical vortices of 3D incompressible Euler equations with swirl in infinite cylinders. As parameter $ \varepsilon\to0 $, the associated vorticity fields $ \mathbf{w}_{\ep} $  tend asymptotically to the helix \eqref{1007}, see Theorem \ref{thm2}. 

We conclude this subsection with a brief overview on related
works. Indeed, helical solutions were also constructed for other models, see for example \cite{WY} for the {S}chr\"{o}dinger map equations,  \cite{chi,ddmr2} for the Gross-Pitaevskii equation  and \cite{ddmr} for the Ginzburg-Landau equation. Another related topic is the knotted solutions of Euler equations \eqref{1-1}. The possibility of non-trivial steady vortex configurations featuring knots and links was conjectured by Kelvin \cite{Kel}. Such solutions were found by Enciso and Peralta-Salas \cite{EP,EP2}. Explicit knotted solutions to the binormal
curvature flow were studied in \cite{Keen}.

\subsection{2D vorticity-stream formulation}

Now we begin to show our main results. To construct helical vorticity field of three-dimensional Euler equations \eqref{1-2} without  the assumption \eqref{1009}, the first key step   is to reduce  the system \eqref{1-2} to  a two-dimensional evolved model by the  helical symmetric structure. For any $ \theta\in\mathbb R$, let us define  the rotation by an angle $\theta$ around the $x_3$-axis
	\begin{equation*}
		Q_\theta=\begin{pmatrix}
			R_\theta & 0  \\
			0 & 1
		\end{pmatrix}  \  \ \text{with}\   \ R_\theta=\begin{pmatrix}
			\cos\theta & \sin\theta  \\
			-\sin\theta &\cos\theta
		\end{pmatrix}.
	\end{equation*}
Let $ h>0 $ be a fixed constant. We introduce a group of one-parameter helical transformations $ \mathcal{H}_h=\{H_{\theta}:\mathbb{R}^3\to\mathbb{R}^3\mid  \theta\in\mathbb{R}\} $, where
	\begin{equation*}
		H_{\theta}(x)=Q_\theta(x)+\begin{pmatrix}
			  0  \\
			0  \\ 
			h\theta
		\end{pmatrix}=\begin{pmatrix}
		x_1\cos\theta+x_2\sin\theta  \\
		-x_1\sin\theta +x_2\cos\theta  \\ 
		x_3+h\theta
	\end{pmatrix}.
	\end{equation*}

We can  define helical domains as follows. For $ \Omega\subseteq \mathbb{R}^3$, we say that $ \Omega $ is a helical domain, if $ \Omega $ is invariant under the action of $ \mathcal{H}_h $, that is, for any $ \theta\in\mathbb R$
\begin{equation}\label{hel dom} 
H_{\theta}(\Omega):=\{H_{\theta}(x)\mid x\in\Omega\}=\Omega.
\end{equation}	
Denote $ D=\{(x_1,x_2)\in \mathbb{R}^2\mid (x_1,x_2,0)\in \Omega\} $ the cross-section of $ \Omega $.

Following \cite{DDMW2, Du, ET}, we can define helical functions and vector fields. We say that
  a scalar function $f:\Om\rightarrow\mathbb R $ is  helical, if $$ f(H_{\theta}(x))=f(x),  \ \ \forall\, \theta\in\mathbb{R}, x\in \Om, $$
  and a vector field $ \mathbf{v}:\Om\rightarrow\mathbb R^3$ is   helical, if $$ \mathbf{v}(H_{\theta}(x))= Q_{\theta} \mathbf{v}(x),\ \ \forall\, \theta\in\mathbb{R}, x\in \Om. $$
 Naturally, a function pair ($\mathbf{v}, P$) is called a $ helical$ solution  of \eqref{1-1} in $ \Om $, if ($\mathbf{v}, P$) solves \eqref{1-1} and both vector field $ \mathbf{v} $ and scalar function $ P $ are helical in the above sense. 
 
 Denote the field of tangents of symmetry lines of $ \mathcal{H}_h $
 \begin{equation}\label{def of xi}
 	\vec{\xi}(x)= (x_2, -x_1, h)^t,
 \end{equation}
 where $ \mathbf{v}^t $ is the transposition of a vector $ \mathbf{v} $. For a velocity field $\mathbf{v} $,  we define its \emph{helical swirl} as 
  \begin{equation}\label{1-21}
 	v_{\vec{\xi}}(x):=\mathbf{v}(x)\cdot \vec{\xi}(x).
 \end{equation}
 
 As mentioned earlier, it is proved in \cite{DDMW2,Du, ET} that if a   helical solution ($\mathbf{v}, P$) of \eqref{1-1}   satisfies  the \textit{orthogonality condition}:
 \begin{equation}\label{1-7}
 	v_{\vec{\xi}}\equiv0,
 \end{equation}
 then its  vorticity field $ \mathbf{w}  $ satisfies
 \begin{equation}\label{1-8}
 	\mathbf{w} =\frac{w_3}{h}\overrightarrow{\xi},
 \end{equation}
 where $ w_3=\partial_{x_1}v_2-\partial_{x_2}v_1 $,   the third component of $ \mathbf{w}  $, is a helical function. Moreover, if we set $ w(x_1,x_2,t)=w_3(x_1,x_2,0,t) $, then  
 $ w $ satisfies the following 2D evolved vorticity-stream equations:
  \begin{equation}\label{1-9}
  	\begin{cases}
  		\partial_t w+\nabla^\perp\varphi\cdot \nabla w=0,\ \ &\text{in}\ \   D\times(0,T),\\
  		\mathcal{L}_{K}\varphi=w,\ \ &\text{in}\ \  D\times(0,T),\\
  		\varphi=0,\ \ &\text{on}\ \   \partial D\times(0,T),\\
  		w(\cdot,0)=w_0(\cdot),\ \ & \text{in}\ \   D,
  	\end{cases}
  \end{equation}
  where $ \varphi $ is called the stream function,  $ \mathcal{L}_{K}\varphi:=-\text{div}(K(x_1,x_2)\nabla\varphi) $ is a second order elliptic operator in divergence form with the coefficient matrix
  \begin{equation}\label{1-10}
  	K(x_1,x_2)=\frac{1}{h^2+x_1^2+x_2^2}
  	\begin{pmatrix}
  		h^2+x_2^2 & -x_1x_2 \\
  		-x_1x_2 &  h^2+x_1^2
  	\end{pmatrix},
  \end{equation}
  and $ \perp $ denotes the clockwise rotation through $ \pi/2 $, i.e., $ (a,b)^\perp=(b,-a) $.  

By studying the rotating invariant solutions of  \eqref{1-9}, several families of  traveling-rotating helical vortices with the assumption \eqref{1-7} were constructed in \cite{CFL, CL23, CW, CW2, DDMW2, GM}. 
Our first result is that even in the absence of the zero helical swirl condition \eqref{1-7}, the three-dimensional Euler equations \eqref{1-2} under helical symmetry can still be reduced  to a two-dimensional evolved system.

\begin{theorem}\label{thm1}
	Let $(\mathbf{v}, P)$ be a $ helical$ solution pair  of \eqref{1-1} in $ \Om $ and  $\mathbf{w} =(w_1,w_2,w_3)^t$  be the corresponding vorticity field of $ \mathbf{v} $. Let $ v_{\vec{\xi}} $ be the helical swirl of $\mathbf{v}$ defined by \eqref{1-21} and $ D $ be the cross-section of $ \Omega $. For $(x_1,x_2)\in D$, define $v(x_1,x_2,t)= v_{\vec{\xi}}(x_1,x_2,0,t)$ and $w(x_1,x_2,t)=w_3(x_1,x_2,0,t)$. Then $(v, w)$ solves the following evolved system:
	\begin{equation}\label{1-23}
		\begin{cases}
			\partial_t v+\nabla^\perp\varphi\cdot \nabla v=0,\ \ &\text{in}\ \   D\times(0,T),\\
			\partial_t w+\nabla^\perp\varphi\cdot \nabla w=\partial_{x_1} v\partial_{x_2}\left(\frac{x_1\partial_{x_1}\varphi+x_2\partial_{x_2}\varphi}{|{\vec{\xi}}|^2}\right)\\\qquad \qquad -\partial_{x_2} v\partial_{x_1}\left(\frac{x_1\partial_{x_1}\varphi+x_2\partial_{x_2}\varphi}{|{\vec{\xi}}|^2}\right)
		 +\frac{2 v\left( x_2\partial_{x_1} v-x_1\partial_{x_2} v\right)}{|{\vec{\xi}}|^4},\ \ &\text{in}\ \   D\times(0,T),\\
			\mathcal{L}_{K}\varphi=w+\frac{2 h^2 v}{|{\vec{\xi}}|^4}+\frac{ x_1\partial_{x_1} v+x_2\partial_{x_2} v}{|{\vec{\xi}}|^2},\ \ &\text{in}\ \  D\times(0,T),\\
			\varphi=0,\ \ &\text{on}\ \   \partial D\times(0,T),\\
			w(\cdot,0)=w_{3}((\cdot,0),0),\ \  v(\cdot,0)=v_{\vec{\xi}}((\cdot,0),0), & \text{in}\ \   D.
		\end{cases}
	\end{equation}

Conversely, suppose that $(v, w)$ is a solution pair of \eqref{1-23}. Let $  \tilde{\mathbf{u}}=(\tilde{u}_1,\tilde{u}_2,\tilde{u}_3)^t $ and $\tilde{\mathbf{v}}=(\tilde{v}_1,\tilde{v}_2,\tilde{v}_3)^t $  be vector fields  in $ D $ such that for $ (x_1,x_2) \in D$
\begin{equation*} 
\begin{pmatrix}
\tilde{u}_1(x_1,x_2) \\
\tilde{u}_2(x_1,x_2)
\end{pmatrix}=\frac{1}{|\vec{\xi}|^2}\begin{pmatrix}
- x_1x_2 &   h^2+x_1^2  \\
-(h^2+x_2^2 ) &  x_1x_2
\end{pmatrix}  \begin{pmatrix}
\partial_{x_1}\varphi(x_1,x_2) \\
\partial_{x_2}\varphi(x_1,x_2)
\end{pmatrix}, 
\end{equation*}
\begin{equation*} 
\tilde{u}_3(x_1,x_2)=\frac{-x_2\tilde{u}_1(x_1,x_2)+x_1\tilde{u}_2(x_1,x_2)}{h} ,
\end{equation*}
and $\tilde{\mathbf{v}}(x_1,x_2)=\tilde{\mathbf{u}}(x_1,x_2)+\frac{v(x_1,x_2)\vec{\xi}}{|\vec{\xi}|^2}$.
Define for  any $ (x_1,x_2,x_3)\in \Om $
$$\mathbf{v}(x_1,x_2,x_3)=Q_{x_3/h}\tilde{\mathbf{v}}\left( R_{-x_3/h} (x_1,x_2)\right),
$$ 
and
$$\mathbf{w} (x_1,x_2,x_3)=\frac{w\left(R_{-x_3/h}(x_1,x_2)\right)}{h} \vec{\xi}+\frac{(\partial_{x_2}v\left(R_{-x_3/h}(x_1,x_2)\right), -\partial_{x_1}v\left(R_{-x_3/h}(x_1,x_2)\right), 0)^t}{h}.$$
Then $ (\mathbf{v}, \mathbf{w} ) $ is a helical solution pair of
the vorticity equations \eqref{1-2}.
\end{theorem}
\begin{remark}
Theorem \ref{thm1} implies that to find helical solutions of \eqref{1-2}, it suffices to solve solution pairs $ (v, w) $ of the 2D evolved model \eqref{1-23}. Note that if the orthogonality condition \eqref{1-7} holds, then $ v(x_1,x_2,t)\equiv0 $ in \eqref{1-23} and we immediately recover \eqref{1-9} from \eqref{1-23}. Thus to some extend,   system \eqref{1-23} in Theorem \ref{thm1} generalizes previous 2D classical model \eqref{1-9} in \cite{CW,DDMW2,ET}. 

We note that when \eqref{1-7} holds, \cite{ET} obtained the existence and uniqueness of $ L^\infty $ weak solutions to \eqref{1-9}, which corresponds to Yudovich's global well-posedness result in the case of the 2-dimensional incompressible Euler equations, see \cite{AS,BLN,JS} for more results. However  when  \eqref{1-7} fails, to our knowledge  there is no result to obtain the global well-posedness of  solutions to the system \eqref{1-23}, see \cite{Jiu1, LMNNT} and reference therein.
\end{remark}
\begin{remark}
Due to the complexity of 3D incompressible Euler equations, many references always use specific symmetries to  simplify the original problem, among which  axial symmetry is most  widely used.  There are many results devoted to deducing 3D Euler equations under axial symmetry with and without swirl to   2D models, and we refer readers to Proposition 2.15 in \cite{MB} for the 2D formulation for these cases. The vorticity-stream formulation \eqref{1-23} plays a parallel role in the context of 3D Euler equations under helical symmetry, just as the  Proposition 2.15 in \cite{MB} (take $\nu=0$ there) does for 3D Euler equations under axial symmetry. We believe that \eqref{1-23}  will be useful in other related  research area of helical solutions to 3D incompressible Euler equations.
\end{remark}

\bigskip

\subsection{Traveling-rotating helical vortices with swirl}
Based on \eqref{1-23}, we now consider the existence of concentrated helical vortices in infinite cylinders, such that the corresponding  support sets of vorticity fields remain invariant in shape as time evolves. Let $\Om\in \mathbb{R}^3$ be an infinite cylinder $ B_R(0)\times\mathbb{R} $ for some $ R>0 $, which is a helical domain defined by \eqref{hel dom}. Let $ D $ be the cross-section of $ \Om $. 
 Let $\bar\alpha>0$ be a fixed constant.  We study rotating invariant solutions of \eqref{1-23} being of the form
\begin{equation}\label{1-12}
	w(x,t)=W(R_{\bar\alpha t}x),\ \ \  v(x,t)=V(R_{\bar\alpha t}x),\ \ \ \varphi(x,t)=\Phi(R_{\bar\alpha t}x).
\end{equation}

Substituting \eqref{1-12} into \eqref{1-23}, we deduce later in section \ref{sec3} that  it suffices to solve the following semilinear second order elliptic equations: 
\begin{equation}\label{1-24}
	\begin{cases}
		\mathcal{L}_{K}\Phi=\frac{2h^2F_1\left(\Phi-\frac{\bar\alpha}{2}|x|^2 \right)}{(h^2+|x|^2)^2}+\frac{F_1\left(\Phi-\frac{\bar\alpha}{2}|x|^2 \right)F_1'\left(\Phi-\frac{\bar\alpha}{2}|x|^2 \right)}{h^2+|x|^2}-\frac{\bar\alpha |x|^2 F_1'\left(\Phi-\frac{\bar\alpha}{2}|x|^2 \right)}{h^2+|x|^2}\\
		\qquad\qquad\qquad\qquad\qquad\qquad\qquad\qquad\qquad\qquad+F_2\left(\Phi-\frac{\bar\alpha}{2}|x|^2 \right),\ \ &\text{in}\ \  D,\\
		\Phi=0,\ \ &\text{on}\ \   \partial D,\\
	\end{cases}
\end{equation}
for some profile functions $F_1$ and $F_2$. Note that if we impose $F_1\equiv0$ , then \eqref{1-24} coincides with  the elliptic system considered in \cite{CL23, CW, CW2, DDMW2, GM}, which corresponds to traveling-rotating helical vortices of \eqref{1-2} with the orthogonality condition \eqref{1-7}. 

We observe that  if $F_1\not\equiv0$, then  solutions to \eqref{1-24} correspond  to  rotating invariant solutions to \eqref{1-23} with non-trivial $ v $, and thus to traveling-rotating helical vortices of  \eqref{1-2} with non-zero helical swirl $ v_{\xi} $. Our second result is the existence of a family of concentrated helical vortices without the assumption \eqref{1-7}  in infinite cylinders, whose vorticity field tends to a helical vortex filament  \eqref{1007} as $ \varepsilon\to0 $. 
\begin{theorem}\label{thm2}
	Let $ h\not=0 $, $\kappa>0 $ and $ r_*\in (0,R) $ be any given numbers. Let  $ \gamma(\tau) $ be the helix parameterized by equation \eqref{1007}. Then there exists a small number $\ep_0\in(0,1)$ such that for any $ \varepsilon\in (0, \ep_0)$, there exists	a helical symmetric solution pair $ (\mathbf{v}_\varepsilon, P_\varepsilon)(x,t)  $ of  \eqref{1-1} such that the support set of  $ \mathbf{w} _\varepsilon=\nabla\times \mathbf{v}_\varepsilon $ is a topological traveling-rotating helical tube that does not change form and concentrates near $ \gamma(\tau) $ in sense   that for all $ \tau $,
	\begin{equation*}
	\text{dist}\left (\partial(supp(\mathbf{w} _\varepsilon(\cdot, |\ln\varepsilon|^{-1}\tau))), \gamma(\tau)\right ) \to 0,\ \ \text{as}\ \varepsilon\to0.
	\end{equation*}
	
	Moreover, the following properties hold:
	\begin{enumerate}
		\item The vorticity field $ \mathbf{w} _\varepsilon $ tends to $ \gamma  $ in distributional sense, i.e., 
		\begin{equation*}
		\mathbf{w} _\varepsilon(\cdot,|\ln\varepsilon|^{-1}\tau)\to \kappa \delta_{\gamma(\tau)}\mathbf{t}_{\gamma(\tau)},\ \ \text{as}\ \varepsilon\to0.
		\end{equation*}
		\item The  helical swirl $v_{\vec{\xi},\ep}$   defined by \eqref{1-21} is not vanishing. That is, $v_{\vec{\xi},\ep}\not\equiv0$. 
		\item Define $ \bar{A}_\varepsilon=\{(x_1,x_2)\in D\  |\  (x_1,x_2,0)\in supp(\mathbf{w} _\varepsilon) \}$ the cross-section of the support of $ \mathbf{w} _\varepsilon $. Then there are constants $ r_1,r_2>0 $ independent of $ \varepsilon $ such that
		\begin{equation*}
		r_1\varepsilon\leq diam(\bar{A}_\varepsilon)\leq r_2\varepsilon.
		\end{equation*}
	\end{enumerate}
\end{theorem}

The idea of proof of Theorem \ref{thm2} is as follows. Let   $a>0$ and  $b\geq0$ be some fixed constants and $\varepsilon\in (0,1)$. Let $ \alpha=\frac{\kappa}{4\pi h\sqrt{h^2+r_*^2}} $ be a constant dependent of $ \kappa,h $ and $ r_* $. We choose $\bar\alpha=\alpha |\ln\ep|$ and
\begin{equation}\label{nonl}
F_1(s)=\ep^{-1}a (s-\mu_\ep)_+,\ \ \   F_2(s)=\ep^{-2}(b+\bar\alpha a \ep) \textbf{1}_{\{s-\mu_\ep>0\}},\ \ \  s\in \mathbb R,
\end{equation}
in \eqref{1-24}, where $s_+:=\max\{s,0\}$,  $ \textbf{1}_{A_0} $ the characteristic function of a set $ {A_0} $, namely $ \textbf{1}_{A_0}=1 $ in $ {A_0} $ and $ \textbf{1}_{A_0}=0 $ in $ {A_0}^c  $, and $ \mu_\ep $ is a proper constant. Indeed,  $ \mu_\ep $ is a Lagrange multiplier whose value is determined in Lemma \ref{lm4-2}.

We prove Theorem \ref{thm2} by constructing solutions to \eqref{1-24} via Arnold's variational principle, which is successfully used in the construction of concentrated helical vortices with  the assumption \eqref{1-7} in infinite cylinders, see, e.g., \cite{CW}. For $ \Lambda>0 $, we define constraint sets
\begin{equation*} 
	\mathcal{M}_{\kappa, \Lambda} =\left\{ \omega\in L^1\cap L^\infty(D) \ \Big| \ 0\leq \omega\leq \frac{\Lambda}{\ep^2},\,\,\, \int_{D} \omega dx=\kappa\right\}, 
\end{equation*}
and the energy functional  
\begin{equation*} 
	E_\ep(\omega)=\frac{1}{2}\int_{D} \omega\mathcal{G}_{K} \omega \mathrm{d} x-\frac{\bar{\alpha}}{2}\int_{D} |x|^2 \omega \mathrm{d} x-\ep^{-2}\int_{D}  J_\ep(|x|, \ep^2\omega(x)) \mathrm d x, 
\end{equation*}
where $ \mathcal{G}_{K} $ is the inverse operator of $ \mathcal{L}_K $ defined by \eqref{1-10}, $ |x|=\sqrt{x_1^2+x_2^2}$ and 
\begin{equation}\label{def of J} 
		J_\ep\left(r, s\right)=\frac{(h^2+r^2)^2\left(s-\frac{h^2\bar{\alpha} a\ep}{h^2+r^2}-b\right)_+^2}{4h^2a\ep+2a^2(h^2+r^2)}.
\end{equation}

The reason of the choice of $ J(r,s) $ in \eqref{def of J} is shown  in section 4. By proving the existence of maximizers of $E_\ep$ relative to $\mathcal{M}_{\kappa, \Lambda} $ and analyzing their asymptotic behavior, we will eventually find a family of concentrated helical vortex solutions  of \eqref{1-1} with non-zero helical swirl.  

\begin{remark}
The reason of taking $F_1$ and $F_2$ in \eqref{1-24} in the form of \eqref{nonl} is for simplicity of the proof. Indeed, by a slight modification of arguments in the present paper (see the proof in section 4), one may construct solutions of \eqref{1-24} with other  types of profile functions $ F_1$ and  $F_2 $. For example, for any  $p>1$,  we can take $F_1(s)=\ep^{-1}  (s-\mu_\ep)_+^{p+1},\    F_2(s)=\ep^{-2} (s-\mu_\ep)_+^{p}  $  and   construct traveling-rotating helical vortices with swirl, which correspond to $ C^1 $ classical helical vortex solutions of \eqref{1-1}.  Note also that constants $ a,b $ in \eqref{nonl} is arbitrary. By choosing different $ a $ and $ b $, we can get different families of concentrated helical vorticity field of \eqref{1-2} satisfying results in Theorem \ref{thm2}.
\end{remark}

\begin{remark}
	One of the key ingredient in the proof of Theorem \ref{thm2} in section 4 is the refined decomposition of Green's functions for  the operator $\mathcal L_K$ in the disk obtained in   \cite{CW2} (see Lemma \ref{lm4-1}). Recently, similar decomposition of Green's functions for   $\mathcal L_K$ in the whole space was proved in \cite{CLQW} and helical vortex patches without helical  swirl in $\mathbb R^3$ were constructed by a variational method. Therefore, by using decomposition of Green's functions  obtained in \cite{CLQW}  and combining ideas of the proof of Theorem \ref{thm2}, we are able to construct a family of concentrated helical vortices with \emph{non-zero helical swirl} in the whole space $\mathbb R^3$. Due to the absence of new difficulties, we will omit the statement and proof of this result in this paper.
\end{remark}

The rest of this paper is organized as follows. In section 2, we derive the 2D  vorticity-stream formulation of 3D Euler equations under helical symmetry with swirl and prove Theorem \ref{thm1}. Some basic properties of helical functions and vector fields are proved. In section 3, we show that to find traveling-rotating invariant helical vortices, it  is sufficient to solve the semilinear elliptic equations \eqref{1-24}. In section 4, we prove the existence and  asymptotic behaviors of maximizers of $ \mathcal{E}_\varepsilon $ in the set $ \mathcal{M}_{\varepsilon, \Lambda}$, from which  we obtain  rotating invariant solution pairs $ (v,w) $ to equations \eqref{1-23} and finish the proof of Theorem \ref{thm2}. 

\section{2D vorticity-stream formulation for helical vortices with swirl}
In this section, we deduce a 2D vorticity-stream evolved system for \eqref{1-2} under helical symmetry without the assumption \eqref{1-7} and thus give proof of Theorem \ref{thm1}. 

Let us start by giving some notations which are frequently used in the following sections. Let $ \Om $ be 
a helical domain defined by \eqref{hel dom} and $ D $ be the cross-section of $ \Om. $ For a general vector field $ \mathbf{v} $ in $ \Om $, we denote $ \mathbf{w}=\nabla\times \mathbf{v} $ the vorticity field of $  \mathbf{v} $. Let $ \vec{\xi} $ be the vector field defined by \eqref{def of xi} and $ |\vec{\xi}|=\sqrt{h^2+x_1^2+x_2^2} $ be the magnitude of $ \vec{\xi} $.  The directional derivative operator along $ \vec{\xi} $ is denoted by
\begin{equation*}
\nabla_{\vec{\xi}} =\vec{\xi}\cdot\nabla=x_2\partial_{x_1}  -x_1\partial_{x_2} +h\partial_{x_3}.
\end{equation*}

The following property can be proved directly by using the definitions of helical functions and vector fields, see also \cite{ET}.
\begin{lemma}[\cite{ET}]\label{lm2-1}
	The following claims hold true:
	\begin{itemize}
\item [(a).] Let  $f:  \Om \mapsto \mathbb R $   be a $ C^1 $ scalar function.  Then  $f$ is  helical symmetric if and only if $$\nabla_{\vec{\xi}} f\equiv0.$$
\item [(b).] Let $\mathbf{v}: \Om\mapsto \mathbb R^3 $  be a $C^1$ vector field .  Then  $\mathbf{v}=(v_1,v_2,v_3)^t$ is  helical symmetric if and only if $$\nabla_{\vec{\xi}} \mathbf{v}=(v_2,-v_1,0)^t.$$
	 
	\end{itemize}
 
\end{lemma}
\begin{proof}
Proofs of $(a)$ and  $(b)$ are given in Claims 2.3 and 2.5 in \cite{ET} respectively and   we omit them here. 
\end{proof}

It follows from \cite{ET} that if  $ \mathbf{v} $ satisfies  \eqref{1-7}, then the associated vorticity field $ \mathbf{w} $ satisfies \eqref{1-8} and \eqref{1-9}. However, in case that $ \mathbf{v} $ does not satisfy \eqref{1-7}, then \eqref{1-9} fails. To deduce the 2D evolved equations of \eqref{1-2} without the assumption \eqref{1-7}, 
we define the \emph{orthogonal vector field} of $ \mathbf{v} $
\begin{equation}\label{orth field} 
\mathbf{u}=\mathbf{v}-\frac{v_{\vec{\xi}}\vec{\xi}}{|\vec{\xi}|^2} 
\end{equation}
and the vorticity field of $ \mathbf{u} $
\begin{equation*} 
\vec{\zeta}=\nabla\times \mathbf{u},
\end{equation*}
where the helical swirl $ v_{\vec{\xi}} $ of $ \mathbf{v} $ is defined by \eqref{1-21}.

We  have  properties of $ \mathbf{u}$, $\mathbf{v} $  and the  associated vorticity field $ \vec{\zeta},\mathbf{w}$ as follows.
\begin{lemma}\label{lm2-2}
Let $\mathbf{v}$ be a $C^1$vector field which is helical symmetric  and divergence-free.  Then there hold
 \begin{itemize}
 	\item [(a).]  
 	$v_{\vec{\xi}}$ is   helical symmetric.
 	\item [(b).]  
 	 $\mathbf{u}$ is   helical symmetric  and divergence-free.		
 	\item [(c).]  Let $\mathbf{w}=(w_1,w_2,w_3)^t $ and  $\vec{\zeta}=(\zeta_1,\zeta_2,\zeta_3)^t $ be the vorticity field of $ \mathbf{v} $ and $ \mathbf{u} $. 
 	Then 
 	\begin{equation}\label{2-1}
 		\mathbf{w}=\frac{w_3}{h}\vec{\xi}+\frac{1}{h}(\partial_{x_2}v_{\vec{\xi}}, -\partial_{x_1}v_{\vec{\xi}},0)^t,
 	\end{equation}
 and
 	\begin{equation}\label{2-2}
 		w_3=\zeta_3-\frac{2h^2 v_{\vec{\xi}}}{|\vec{\xi}|^4}-\frac{x_1 \partial_{x_1}v_{\vec{\xi}}+x_2\partial_{x_2}v_{\vec{\xi}}}{|\vec{\xi}|^2}.
 	\end{equation}
 \end{itemize}

\end{lemma}
\begin{remark}
Note that if $ \mathbf{v} $ satisfies the orthogonality condition, i.e., $ v_{\vec{\xi}}\equiv 0, $ then \eqref{2-1} becomes $ \mathbf{w}=\frac{w_3}{h}\vec{\xi} $. This coincides with the property \eqref{1-8} in \cite{ET}.
\end{remark}
\begin{proof}
  
For $(a)$ and $(b)$, we refer to the proof of Lemma 2.5 in \cite{Jiu1}. Indeed, these can be directly verified by Lemma \ref{lm2-1} and the definition of $ \mathbf{u}$ and $\mathbf{v} $.
  
As for $(c)$, we first prove \eqref{2-1}.  By the definition of $\mathbf{u}$, we observe that   $\mathbf{u}\cdot \vec{\xi}=0$, which implies that
  \begin{equation}\label{u3}
   u_3=\frac{-x_2u_1+x_1u_2}{h}.
  \end{equation}
  We deduce from the above identity and the relation $\mathbf{u}=\mathbf{v}-\frac{v_{\vec{\xi}}\vec{\xi}}{|\vec{\xi}|^2}$ that 
  \begin{equation}\label{2-3}
  	\begin{split}
  		  	v_3=&u_3+\frac{v_{\vec{\xi}}h}{|\vec{\xi}|^2}\\
  		  	=&\frac{-x_2u_1+x_1u_2}{h}+\frac{v_{\vec{\xi}}h}{|\vec{\xi}|^2}\\
  		  	=&\frac{1}{h}\left[-x_2 \left( v_1-\frac{x_2v_{\vec{\xi}} }{|\vec{\xi}|^2}\right)+x_1 \left( v_2+\frac{x_1v_{\vec{\xi}} }{|\vec{\xi}|^2}\right)\right]+\frac{v_{\vec{\xi}}h}{|\vec{\xi}|^2}\\
  		  	=&\frac{-x_2v_1+x_1v_2}{h}+\frac{v_{\vec{\xi}}}{h}.
  	\end{split}
  \end{equation}

 Since  $\mathbf{v}$ is   helical symmetric,  using Lemma \ref{lm2-1}  we find 
   \begin{equation*}
   	\partial_{x_3}v_2=\frac{-x_2\partial_{x_1}v_2+x_1\partial_{x_2}v_2-v_1}{h}.
   \end{equation*}
   By direct calculations, we deduce from the above equation and \eqref{2-3} that 
      \begin{equation*} 
    	\begin{split}
    		w_1&=\partial_{x_2}v_3-\partial_{x_3}v_2\\
    		&=\partial_{x_2}\left(\frac{-x_2v_1+x_1v_2}{h}+\frac{v_{\vec{\xi}}}{h}\right)-\frac{-x_2\partial_{x_1}v_2+x_1\partial_{x_2}v_2-v_1}{h}\\
    		&=\frac{x_2}{h}(\partial_{x_1}v_2-\partial_{x_2}v_1)+\frac{\partial_{x_2} v_{\vec{\xi}}}{h}\\
            &=\frac{x_2}{h}w_3+\frac{\partial_{x_2} v_{\vec{\xi}}}{h}.
    	\end{split}
    \end{equation*}
    Similarly, one has
    \begin{equation*} 
    	\begin{split}
    		w_2 =\frac{-x_1}{h}w_3-\frac{\partial_{x_1} v_{\vec{\xi}}}{h}.
    	\end{split}
    \end{equation*}
\eqref{2-1} then follows from the above calculations immediately.  As for  \eqref{2-2},   using $\mathbf{u}=\mathbf{v}-\frac{v_{\vec{\xi}}\vec{\xi}}{|\vec{\xi}|^2}$ and the definition of vorticity field $  \vec{\zeta}$ and $\mathbf{w} $, one computes  directly that \eqref{2-2} holds. The proof of Lemma \ref{lm2-2} is complete.
\end{proof}

From \cite{ET}, we know that for a helical symmetric  and divergence-free vector field $ \mathbf{v} $ satisfying  \eqref{1-7},  the associated vorticity field $ \mathbf{w} $ is  a helical symmetric vector field. While the following lemma shows that, without   the assumption \eqref{1-7}, $ \mathbf{w} $ is still helical.
\begin{lemma}\label{lm2-3}
For a $C^2$ vector field $\mathbf{v} $, let $\mathbf{w} $ be the vorticity field of $ \mathbf{v} $.  If $\mathbf{v}$ is   helical symmetric  and divergence-free, then $\mathbf{w}$ is   helical symmetric.
	
\end{lemma}
\begin{proof}
	In view of Lemma \ref{lm2-1}, it suffices to show that 
	\begin{equation}\label{2-4}
		 \nabla_{\vec{\xi}} \mathbf{w}=(w_2,-w_1,0)^t.
	\end{equation}
	 
	 By \eqref{2-1} and straightforward computations, we get 
	 \begin{align*}
	 	0&=\nabla\cdot \mathbf{w}=\partial_{x_1} w_1+\partial_{x_2} w_2+\partial_{x_3} w_3\\
	 	&=\partial_{x_1}\left(\frac{x_2w_3+\partial_{x_2}v_{\vec{\xi}}}{h} \right)+\partial_{x_1}\left(\frac{-x_1w_3+\partial_{x_1}v_{\vec{\xi}}}{h} \right)+\partial_{x_3} w_3\\
	 	&=\frac{x_2\partial_{x_1} w_3-x_1\partial_{x_2} w_3+ h\partial_{x_3} w_3}{h}=\frac{\nabla_{\vec{\xi}} w_3}{h}.
	 \end{align*}
Thus  $\nabla_{\vec{\xi}}  w_3=0$, which implies that $w_3$ is a helical symmetric function.
    
    Using \eqref{2-1} and the fact   $\nabla_{\vec{\xi}}  w_3=\nabla_{\vec{\xi}}  v_{\vec{\xi}}=0$, one computes directly that
     \begin{align*}
    	\nabla_{\vec{\xi}} w_1= \frac{1}{h}\left( x_2 \nabla_{\vec{\xi}}  w_3+ w_3 \nabla_{\vec{\xi}}  x_2+\nabla_{\vec{\xi}}\partial_{x_2} v_{\vec{\xi}} \right)= \frac{-x_1 w_3-\partial_{x_1} v_{\vec{\xi}}}{h}=w_2.
    \end{align*}
  Here we have used the identity $\partial_{x_2}(\nabla_{\vec{\xi}} v_{\vec{\xi}})=\nabla_{\vec{\xi}}\partial_{x_2} v_{\vec{\xi}}+\partial_{x_1} v_{\vec{\xi}}$. Similarly, we find $$\nabla_{\vec{\xi}} w_2=-w_1.$$ Thus \eqref{2-4} holds and $\mathbf{w}$ is   helical symmetric.
\end{proof}

Based on the above three lemmas, we are able to  introduce the  ``stream function'' for a divergence-free  helical symmetric vector field $ \mathbf{v} $ without   the assumption \eqref{1-7}. 
Let $ D $ be the cross-section of $ \Om $, which is a simply-connected bounded  domain with smooth boundary $ \partial D. $ We have
\begin{lemma}\label{lm2-4}
Let $\mathbf{v}  $ be a helical symmetric  and divergence-free vector field and $v_{\vec{\xi}} $ and $\mathbf{u}$  be defined by \eqref{1-21} and \eqref{orth field}.	    If in addition 
	   $$	\mathbf{v}\cdot \mathbf{n}=0,\ \  \text{on}\ \ \partial \Om,$$  where $ \mathbf{n} $ is the outward unit normal to $ \partial \Om$.  
	   Then there exists the unique function  $\varphi: D\mapsto \mathbb R$ such that 
	   \begin{equation}\label{2-5}
	   	 \begin{pmatrix}
	   	 	u_1(x_1,x_2,0) \\
	   	 	u_2(x_1,x_2,0)
	   	 \end{pmatrix}=\frac{1}{|\vec{\xi}|^2}\begin{pmatrix}
	   	- x_1x_2 &   h^2+x_1^2  \\
	   	 -(h^2+x_2^2 ) &  x_1x_2
   	 \end{pmatrix}  \begin{pmatrix}
   	 \partial_{x_1}\varphi(x_1,x_2) \\
   	 \partial_{x_2}\varphi(x_1,x_2)
    \end{pmatrix}, \  \ \ \forall\, (x_1,x_2)\in D 
	   \end{equation}
   and 
      \begin{equation}\label{2-5'}
   	   \varphi=0, \ \  \text{on}\ \ \partial D.
   \end{equation}

Moreover, let $\vec{\zeta}=(\zeta_1,\zeta_2,\zeta_3)^t $ be the vorticity of $ \mathbf{u} $. Then the following relation holds
	 \begin{equation}\label{2-6}
	 	 	\mathcal{L}_{K}\varphi(x_1,x_2)=\zeta_3(x_1,x_2,0), \  \ \ \forall\, (x_1,x_2)\in D,
	 \end{equation}
where   $ \mathcal{L}_{K}\varphi:=-\text{div}(K(x_1,x_2)\nabla\varphi) $ is a second order elliptic operator in divergence form with   the coefficient matrix
	\begin{equation*} 
		K(x_1,x_2)=\frac{1}{h^2+x_1^2+x_2^2}
		\begin{pmatrix}
			h^2+x_2^2 & -x_1x_2 \\
			-x_1x_2 &  h^2+x_1^2
		\end{pmatrix}.
	\end{equation*}
\end{lemma}
\begin{proof}
	Since  $\mathbf{v} $ is   helical symmetric  and divergence-free, from Lemma \ref{lm2-2} we get that  $\mathbf{u} $ is also  helical symmetric  and divergence-free. Using identities  $\mathbf{u}\cdot \vec{\xi}=0$ and   $\nabla_{\vec{\xi}} u_3=0$, direct calculation gives 
	$$ 0=\partial_{x_1} u_1+\partial_{x_2} u_2+\partial_{x_3} u_3=\frac{1}{h^2}\partial_{x_1}  \left[(h^2+x_2^2)u_1-x_1x_2u_2\right]+\frac{1}{h^2}\partial_{x_2}  \left[(h^2+x_1^2)u_2-x_1x_2u_1\right].$$ Thus we can define a function  $\varphi: D\mapsto \mathbb R$ such that $$\partial_{x_2} \varphi(x_1,x_2)=\frac{1}{h^2}\left[(h^2+x_2^2)u_1(x_1,x_2,0)-x_1x_2u_2(x_1,x_2,0)\right]$$ and $$\partial_{x_1} \varphi(x_1,x_2)=\frac{1}{h^2}  \left[(h^2+x_1^2)u_2(x_1,x_2,0)-x_1x_2u_1(x_1,x_2,0)\right],$$ which implies \eqref{2-5}.   \eqref{2-6} follows immediately from  \eqref{2-5} and the definition of $ \zeta_3 $, that is,  $\zeta_3=\partial_{x_1} u_2-\partial_{x_2} u_1$.
	
	Next, we verify the boundary condition for $\varphi$. Let $ \mathbf{n}=(n_1,n_2,n_3)^t $ be the outward unit normal to $ \partial \Om$ at the point $(x_1,x_2,0)\in \partial \Om$. Then the helical symmetry of $\Om$ implies that $\vec{\xi}\cdot \mathbf{n}=0$ (See (2.62) in \cite{ET}). That is, $$n_3=\frac{-x_2 n_1+x_1n_2}{h}.$$
	Therefore, by using \eqref{2-5}, we obtain 
	\begin{align*}
		0&=	\mathbf{v}\cdot \mathbf{n}=\left(\mathbf{u} + \frac{v_{\vec{\xi}}\vec{\xi}}{|\vec{\xi}|^2}\right)\cdot \mathbf{n}=\mathbf{u}\cdot \mathbf{n}=u_1n_1+u_2n_2+\frac{-x_2 u_1+x_1u_2}{h}\cdot\frac{-x_2 n_1+x_1n_2}{h}\\
		&=\partial_{x_2}\varphi n_1-\partial_{x_2}\varphi n_2,
	\end{align*} 
	 which implies that for some constant $ C $, $$\varphi\equiv C \ \  \text{on}\ \ \partial D.$$  
	 Without loss of generality, we may take $\varphi\equiv 0 \ \  \text{on}\ \ \partial D.$ The uniqueness of $ \varphi $ follows from the maximum principle of elliptic equations. The proof   is thus finished.
\end{proof}

From the definition of helical functions, we get the following property of $ \varphi $.
\begin{lemma} \label{lm2-5}
For any helical symmetric function $g: \Om\mapsto \mathbb R$, the following identity holds
	\begin{equation*}
		\left (\mathbf{v} \cdot \nabla g\right )(x_1,x_2,0)  = \nabla^\perp\varphi(x_1,x_2)\cdot\nabla g(x_1,x_2,0),\ \ \ \ \ (x_1,x_2)\in D.
	\end{equation*}
\end{lemma}
\begin{proof}
	Since $g$ is  helical symmetric, using \eqref{orth field} and Lemma \ref{lm2-1} we obtain
	$$	\mathbf{v}  \cdot \nabla g=\left(\mathbf{u}+\frac{v_{\vec{\xi}}\vec{\xi}}{|\vec{\xi}|^2}\right)  \cdot \nabla g=	\mathbf{u}  \cdot \nabla g.$$ 
By $\mathbf{u}\cdot \vec{\xi}=0$ and $ \nabla_{\vec{\xi}} g=0 $, we have
\begin{equation*} 
u_3=\frac{ -x_2u_1+x_1u_2}{h},
\end{equation*}
and
\begin{equation*} 
\partial_{x_3} g=\frac{-x_2\partial_{x_1} g+x_1\partial_{x_2} g}{h}.
\end{equation*}
Combining the above equalities with \eqref{2-5}, one computes directly that 
\begin{equation*} 
\begin{split}
\mathbf{u}  \cdot \nabla g=&(u_1(x_1,x_2,0),u_2(x_1,x_2,0))\begin{pmatrix}
1+\frac{x_2^2}{h^2} &   -\frac{x_1x_2}{h^2}  \\
-\frac{x_1x_2}{h^2} & 1+\frac{x_1^2}{h^2}
\end{pmatrix}  \begin{pmatrix}
\partial_{x_1}g(x_1,x_2,0) \\
\partial_{x_2}g(x_1,x_2,0)
\end{pmatrix}\\
=&\frac{1}{|\vec{\xi}|^2}\left (\partial_{x_1}\varphi(x_1,x_2), \partial_{x_2}\varphi(x_1,x_2)\right )\begin{pmatrix}
- x_1x_2 &  -(h^2+x_2^2 ) \\ 
 h^2+x_1^2  &  x_1x_2
\end{pmatrix}  \begin{pmatrix}
1+\frac{x_2^2}{h^2} &   -\frac{x_1x_2}{h^2}  \\
-\frac{x_1x_2}{h^2} & 1+\frac{x_1^2}{h^2}
\end{pmatrix}  \begin{pmatrix}
\partial_{x_1}g(x_1,x_2,0) \\
\partial_{x_2}g(x_1,x_2,0)
\end{pmatrix}\\
=&\left (\partial_{x_1}\varphi(x_1,x_2), \partial_{x_2}\varphi(x_1,x_2)\right )\begin{pmatrix}
0 &  -1 \\ 
1  & 0
\end{pmatrix}   \begin{pmatrix}
\partial_{x_1}g(x_1,x_2,0) \\
\partial_{x_2}g(x_1,x_2,0)
\end{pmatrix}\\
=&\partial_{x_2} \varphi  (x_1,x_2)  \partial_{x_1} g(x_1,x_2,0)-\partial_{x_1} \varphi  (x_1,x_2)  \partial_{x_2} g(x_1,x_2,0)\\
=&\nabla^\perp\varphi(x_1,x_2)\cdot\nabla g(x_1,x_2,0).
\end{split}
\end{equation*}
\end{proof}

Based on Lemmas \ref{lm2-1}-\ref{lm2-5}, we now deduce the 2D evolved equations of helical solutions to \eqref{1-2} without the assumption \eqref{1-7}. Using \eqref{1-2} and Lemma \ref{2-2}, we first obtain the evolved equations of helical swirl $ v_{\vec{\xi}} $ and the third component $ w_3 $ of the vorticity field $ \mathbf{w} $ as follows. 
\begin{lemma} \label{lm2-6}
Let  $(\mathbf{v},p)$ be a  helical symmetric solution pair of the Euler equations \eqref{1-1} in  $\Om$, 
  $v_{\vec{\xi}} $ be defined by \eqref{1-21} and $\mathbf{w}=(w_1,w_2,w_3)^t $ be the vorticity field of $ \mathbf{v} $. Then $v_{\vec{\xi}}$ and $w_3$ satisfy the following transport system  
	\begin{equation}\label{2-7'}
		\begin{cases}
			\partial_t v_{\vec{\xi}}+\mathbf{v}\cdot \nabla  v_{\vec{\xi}}=0,&\text{in}\ \   \Om\times(0,T),\\
			\partial_t w_3+\mathbf{v}\cdot \nabla  w_3= \frac{1}{h}\left(\partial_{x_2} v_{\vec{\xi}}\partial_{x_1} v_3- \partial_{x_1} v_{\vec{\xi}}\partial_{x_2} v_3\right),\ &\text{in}\ \   \Om\times(0,T).
		\end{cases}
	\end{equation}
\end{lemma}
\begin{proof}
The equation for $v_{\vec{\xi}}$ was derived directly  by Lemma 2.7 in \cite{ET}. 
For $ w_3 $, by \eqref{1-2} and Lemma \ref{lm2-2}, we get
\begin{equation}\label{2-7} 
\begin{split}
\partial_t w_3+\mathbf{v}\cdot \nabla  w_3=   \mathbf{w}\cdot \nabla  v_3 
= \left( \frac{w_3}{h}\vec{\xi}+\frac{1}{h}(\partial_{x_2}v_{\vec{\xi}}, -\partial_{x_1}v_{\vec{\xi}},0)\right) \cdot \nabla  v_3.
\end{split}
\end{equation}
Since $ \mathbf{v} $ is helical, we have $\nabla_{\vec{\xi}} v_3=\vec{\xi}\cdot \nabla  v_3=0$. Taking this into \eqref{2-7}, we get the equation for $w_3$.
\end{proof}

We observe that both $v_{\vec{\xi}}$ and $w_3$ are helical functions in $ \Omega $, which are uniquely determined by their values on the cross-section $ D $.  If we denote $ v$ and $w $ the corresponding constraint functions of $v_{\vec{\xi}}$ and $w_3$ on   $ D $, i.e.,  $v(x_1,x_2,t):=v_{\vec{\xi}}(x_1,x_2,0,t)$ and $w(x_1,x_2,t):=w_3(x_1,x_2,0,t)$ for any $ (x_1,x_2)\in D $, then  according to Lemma \ref{lm2-6}, it is possible to obtain  equations satisfied by $ (v,w) $. Indeed, taking Lemmas \ref{lm2-1}-\ref{lm2-5} into Lemma \ref{lm2-6}, we get a 2D evolved model of $v $ and $w $, which is the  vorticity-stream formulation of helical solutions to \eqref{1-2} without the assumption \eqref{1-7}.
\begin{proposition}\label{pp2-7}
Let  $(\mathbf{v},p)$ be a  helical symmetric solution pair of incompressible Euler equations \eqref{1-1} in  $\Om$, $ D $ be the cross-section of $ \Om $,  $v_{\vec{\xi}} $ be defined by \eqref{1-21} and $\mathbf{w}=(w_1,w_2,w_3)^t $ be the vorticity field of $ \mathbf{v} $. Let $v(x_1,x_2,t) =v_{\vec{\xi}}(x_1,x_2,0,t)$ and $w(x_1,x_2,t) =w_3(x_1,x_2,0,t)$ for any $ (x_1,x_2)\in D $. Then $(v,w)$ satisfies the following 2D evolved system
	\begin{equation}\label{2-9}
		\begin{cases}
			\partial_t v+\nabla^\perp\varphi\cdot \nabla v=0,\ \ &\text{in}\ \   D\times(0,T),\\
			\partial_t w+\nabla^\perp\varphi\cdot \nabla w=\partial_{x_1} v \partial_{x_2}\left(\frac{x_1\partial_{x_1}\varphi+x_2\partial_{x_2}\varphi}{|{\vec{\xi}}|^2}\right)\\\qquad \qquad -\partial_{x_2} v \partial_{x_1}\left(\frac{x_1\partial_{x_1}\varphi+x_2\partial_{x_2}\varphi}{|{\vec{\xi}}|^2}\right)
			+\frac{2 v \left( x_2\partial_{x_1} v -x_1\partial_{x_2} v \right)}{|{\vec{\xi}}|^4},\ \ &\text{in}\ \   D\times(0,T),\\
			\mathcal{L}_{K}\varphi=w+\frac{2 h^2 v }{|{\vec{\xi}}|^4}+\frac{ x_1\partial_{x_1} v +x_2\partial_{x_2} v }{|{\vec{\xi}}|^2},\ \ &\text{in}\ \  D\times(0,T),\\
			\varphi=0,\ \ &\text{on}\ \   \partial D\times(0,T),\\
		w(\cdot,0)=w_{3}((\cdot,0),0),\ \  v(\cdot,0)=v_{\vec{\xi}}((\cdot,0),0), & \text{in}\ \   D.
		\end{cases}
	\end{equation}
\end{proposition}
\begin{proof}
	Let $\mathbf{u}$ 
	be the orthogonal vector field of $ \mathbf{v} $ defined by \eqref{orth field} and  $\varphi$ be the function given by Lemma \ref{lm2-4}.  Using  $\mathbf{u}\cdot \vec{\xi}=0$ and \eqref{2-5},   one computes directly that  
	\begin{equation}\label{2-8}
		u_3(x_1,x_2,0)=-\frac{h}{|{\vec{\xi}}|^2}\left(x_1\partial_{x_1}\varphi(x_1,x_2)+x_2\partial_{x_2}\varphi(x_1,x_2)\right).
	\end{equation}

For $ v $, it follows from  Lemma \ref{lm2-5} and the first equation of \eqref{2-7'} in Lemma \ref{lm2-6} that
\begin{equation*} 
0=\partial_t v +\mathbf{v}\cdot \nabla  v =\partial_t v + \nabla^\perp\varphi\cdot \nabla v.
\end{equation*}
So the first equation of \eqref{2-9} holds.

For $ w $, by Lemmas \ref{lm2-3} and \ref{lm2-5}, we obtain
\begin{equation*} 
\begin{split}
\partial_t w +\mathbf{v}\cdot \nabla  w =\partial_t w + \nabla^\perp\varphi\cdot \nabla w.
\end{split}
\end{equation*}
By \eqref{2-8}, we have
\begin{equation*} 
\begin{split}
&\frac{1}{h}\left(\partial_{x_2} v \partial_{x_1} v_3- \partial_{x_1} v \partial_{x_2} v_3\right)\\
=&\frac{1}{h}\left(\partial_{x_2} v\partial_{x_1} \left( -\frac{h}{|{\vec{\xi}}|^2}\left(x_1\partial_{x_1}\varphi +x_2\partial_{x_2}\varphi \right)+\frac{v h}{|\vec{\xi}|^2}\right) -\partial_{x_1} v\partial_{x_2} \left( -\frac{h}{|{\vec{\xi}}|^2}\left(x_1\partial_{x_1}\varphi +x_2\partial_{x_2}\varphi \right)+\frac{v h}{|\vec{\xi}|^2}\right) \right) \\
=&\partial_{x_1} v \partial_{x_2}\left(\frac{x_1\partial_{x_1}\varphi+x_2\partial_{x_2}\varphi}{|{\vec{\xi}}|^2}\right) -\partial_{x_2} v \partial_{x_1}\left(\frac{x_1\partial_{x_1}\varphi+x_2\partial_{x_2}\varphi}{|{\vec{\xi}}|^2}\right)
+\frac{2 v \left( x_2\partial_{x_1} v -x_1\partial_{x_2} v \right)}{|{\vec{\xi}}|^4}.
\end{split}
\end{equation*}
Combining the above two equalities with  the second equation of \eqref{2-7'} in Lemma \ref{lm2-6}, we get the second equation of \eqref{2-9}.

Finally, using \eqref{2-2} in Lemma \ref{lm2-2} and \eqref{2-6}, we can deduce the third equation of \eqref{2-9}. The proof is thus complete.
\end{proof}
\begin{remark}
We give a remark that if the orthogonality condition \eqref{1-7} holds, then $ v(x_1,x_2,t)\equiv0 $ in \eqref{2-9} and we immediately get the 2D model \eqref{1-9} from \eqref{2-9}. 
\end{remark}
Indeed, we can recover helical solution pairs $  (\mathbf{v},\mathbf{w}) $ of vorticity equations \eqref{1-2} without the orthogonality condition \eqref{1-7} by solution pairs $(v,w)$  of the system \eqref{2-9}. This implies that there is a one-to-one correspondence between the solution of the 2D  equations  \eqref{2-9} and the helical symmetric solution of \eqref{1-2} with swirl.   We have
\begin{proposition}\label{pp2-8}
Suppose that  $(v,w)$ is   a   solution pair  of the system \eqref{2-9} in $D\subset \mathbb R^2$. Then one can recover a helical vector field $ \mathbf{v} $ and its helical vorticity field $\mathbf{w} $  in $ \Om=\{x=(x_1,x_2,x_3)^t\in\mathbb R^3\ |\ R_{x_3/h}\left((x_1,x_2)^t\right)\in D\} $ 
	such that $(\mathbf{v}, \mathbf{w}) $ solves vorticity equations \eqref{1-2}.
\end{proposition} 
\begin{proof}
Let $(v,w)$ be  a   solution pair  of the system \eqref{2-9} in $ D $ and $\varphi$ be the associated stream function. We first recover $ \mathbf{v} $ and $ \mathbf{w} $. From Lemmas \ref{lm2-1}-\ref{lm2-4}, we define a vector field $ \tilde{\mathbf{u}}=(\tilde{u}_1,\tilde{u}_2,\tilde{u}_3)^t $ in $ D $ such that  $\tilde{u}_1$ and $\tilde{u}_2$ satisfy  \eqref{2-5} and $\tilde{u}_3$  satisfies \eqref{u3}, i.e.,
\begin{equation*} 
   	 \begin{pmatrix}
\tilde{u}_1(x_1,x_2) \\
\tilde{u}_2(x_1,x_2)
\end{pmatrix}=\frac{1}{|\vec{\xi}|^2}\begin{pmatrix}
- x_1x_2 &   h^2+x_1^2  \\
-(h^2+x_2^2 ) &  x_1x_2
\end{pmatrix}  \begin{pmatrix}
\partial_{x_1}\varphi(x_1,x_2) \\
\partial_{x_2}\varphi(x_1,x_2)
\end{pmatrix}, 
\end{equation*}
\begin{equation*} 
\tilde{u}_3(x_1,x_2)=\frac{-x_2\tilde{u}_1(x_1,x_2)+x_1\tilde{u}_2(x_1,x_2)}{h}.
\end{equation*}
We then extend $\tilde{\mathbf{u}}$  to a helical symmetric vector field $ \mathbf{u} $ in $ \Om $  naturally by letting 
$$\mathbf{u}(x',x_3)=Q_{x_3/h}\tilde{\mathbf{u}}\left( R_{-x_3/h} x'\right),\ \ \  \ \  x'=(x_1,x_2)^t, \ \ \  (x',x_3)\in\Om.$$
So $ \mathbf{u} $ is a helical vector field which satisfies $ \mathbf{u}\cdot \vec{\xi}=0 $. By \eqref{2-5} and \eqref{u3}, one can check that $ \mathbf{u} $ is divergence-free, i.e., $ \nabla\cdot \mathbf{u}=0 $. 

Now, we  extend $v$  to a helical symmetric function $ v_{\vec{\xi}} $ in $ \Om $  by letting 
$$v_{\vec{\xi}}(x',x_3)=v\left( R_{-x_3/h} x'\right),\ \ \  \ \  x'=(x_1,x_2)^t, \ \ \  (x',x_3)\in\Om $$ and 
set 
\begin{equation}\label{2-10} 
\mathbf{v}=\mathbf{u}+\frac{v_{\vec{\xi}}\vec{\xi}}{|\vec{\xi}|^2},\ \ \ \ \text{in}\ \Om.
\end{equation}
Then $ \mathbf{v} $ is a helical vector field  and  $ \nabla\cdot \mathbf{v}=0 $. Moreover, the helical swirl of $  \mathbf{v} $ is $ v_{\vec{\xi}}. $ Define $ \mathbf{w}=(w_1,w_2,w_3)^t=\nabla\times \mathbf{v} $ the vorticity field of $  \mathbf{v}  $ and $ \vec{\zeta}=(\zeta_1,\zeta_2,\zeta_3)^t=\nabla\times \mathbf{u} $ the vorticity field of $  \mathbf{u}  $. From Lemma \ref{lm2-3}, $ \mathbf{w} $ is helical.

We need to figure out the relation between $ \mathbf{w} $ that $ w $. Indeed, it follows from \eqref{2-1}, \eqref{2-2}, \eqref{2-6} and the third equation of \eqref{2-9} that  for $ (x_1,x_2)\in D $
\begin{equation*}
\begin{split}
w_3(x_1,x_2,0)=\zeta_3(x_1,x_2,0)-\frac{2h^2 v(x_1,x_2)}{|\vec{\xi}|^4}-\frac{x_1 \partial_{x_1}v(x_1,x_2)+x_2\partial_{x_2}v(x_1,x_2)}{|\vec{\xi}|^2}=w(x_1,x_2).
\end{split}
\end{equation*}
This implies that $ w $ is the constraint function of the third component $ w_3 $ of $ \mathbf{w} $ on the cross-section $ D $.

Finally, we prove that $ \mathbf{v} $ defined by \eqref{2-10} and its vorticity field $ \mathbf{w} $ satisfy vorticity equations \eqref{1-2}, i.e., $ \partial_t \mathbf{w}+(\mathbf{v}\cdot\nabla)\mathbf{w}=(\mathbf{w}\cdot\nabla)\mathbf{v} $.
For $ w_3 $, it follows from the second equation of \eqref{2-9} and former deductions that the second equation of \eqref{2-7'} holds. So by \eqref{2-7}, there holds
\begin{equation}\label{2-10'} 
\partial_t w_3+(\mathbf{v}\cdot\nabla)w_3=(\mathbf{w}\cdot\nabla)v_3.
\end{equation}

For $ w_2 $, taking the derivative of the first equation of \eqref{2-7'} with respect to $ x_i $, we get
\begin{equation}\label{2-11}
\partial_t \left (\partial_{x_i}v_{\vec{\xi}}\right )+\left( \mathbf{v}\cdot \nabla \right)  \left (\partial_{x_i}v_{\vec{\xi}}\right )+\left (\partial_{x_i}\mathbf{v}\cdot \nabla \right ) v_{\vec{\xi}}=0.
\end{equation}
Here $ \left (\partial_{x_i}\mathbf{v}\cdot \nabla \right ) v_{\vec{\xi}}:=\sum_{j=1}^{3}\partial_{x_i}v_j\partial_{x_j}v_{\vec{\xi}} $. So by \eqref{2-1}, \eqref{2-10'} and \eqref{2-11}, we have
\begin{equation}\label{2-12}
\begin{split}
&h\left( \partial_t w_1+\left (\mathbf{v}\cdot\nabla\right )w_1-\left (\mathbf{w}\cdot\nabla\right )v_1\right)\\
=&\partial_t\left (x_2 w_3+\partial_{x_2}v_{\vec{\xi}}\right )+\left (\mathbf{v}\cdot\nabla\right )\left (x_2 w_3+\partial_{x_2}v_{\vec{\xi}}\right )-h\left (\mathbf{w}\cdot\nabla\right )v_1\\
=&x_2\partial_t w_3+\partial_t\left (\partial_{x_2}v_{\vec{\xi}}\right )+x_2\left (\mathbf{v}\cdot\nabla\right )w_3 + w_3v_2+\left (\mathbf{v}\cdot\nabla\right )\left (\partial_{x_2}v_{\vec{\xi}}\right )-h\left (\mathbf{w}\cdot\nabla\right )v_1\\
=&x_2\left (\mathbf{w}\cdot\nabla\right )v_3-\left (\partial_{x_2}\mathbf{v}\cdot \nabla \right ) v_{\vec{\xi}}+ w_3v_2-h\left (\mathbf{w}\cdot\nabla\right )v_1.
\end{split}
\end{equation} 
Since  $ \mathbf{v} $ and $ v_{\vec{\xi}} $ are helical symmetric, we have
$ \nabla_{\vec{\xi}}v_3=0, \nabla_{\vec{\xi}}v_1=v_2, \nabla_{\vec{\xi}}v_{\vec{\xi}}=0. $ Combining these with \eqref{2-1}, we obtain
\begin{equation}\label{2-13}
x_2\left (\mathbf{w}\cdot\nabla\right )v_3=\frac{x_2}{h}\left(\partial_{x_2}v_{\vec{\xi}}\partial_{x_1}v_3-\partial_{x_1}v_{\vec{\xi}}\partial_{x_2}v_3 \right) 
\end{equation}
and
\begin{equation}\label{2-14}
h\left (\mathbf{w}\cdot\nabla\right )v_1=w_3v_2+ \partial_{x_2}v_{\vec{\xi}}\partial_{x_1}v_1-\partial_{x_1}v_{\vec{\xi}}\partial_{x_2}v_1.
\end{equation}
Taking \eqref{2-13}, \eqref{2-14} into \eqref{2-12} and using the fact that $ \nabla\cdot \mathbf{v}=0 $, $ \partial_{x_3}v_3=\frac{1}{h}\left(-x_2\partial_{x_1}v_3+x_1\partial_{x_2}v_3 \right) $ and  $ \partial_{x_3}v_{\vec{\xi}}=\frac{1}{h}\left(-x_2\partial_{x_1}v_{\vec{\xi}}+x_1\partial_{x_2}v_{\vec{\xi}} \right)  $, we get
\begin{equation}\label{2-15}
\begin{split}
&h\left( \partial_t w_1+\left (\mathbf{v}\cdot\nabla\right )w_1-\left (\mathbf{w}\cdot\nabla\right )v_1\right)\\
=&\frac{x_2}{h}\left(\partial_{x_2}v_{\vec{\xi}}\partial_{x_1}v_3-\partial_{x_1}v_{\vec{\xi}}\partial_{x_2}v_3 \right)-\left( \partial_{x_2}v_1\partial_{x_1}v_{\vec{\xi}}+\partial_{x_2}v_2\partial_{x_2}v_{\vec{\xi}}+\partial_{x_2}v_3\partial_{x_3}v_{\vec{\xi}} \right)\\
& +w_3v_2-\left(w_3v_2+ \partial_{x_2}v_{\vec{\xi}}\partial_{x_1}v_1-\partial_{x_1}v_{\vec{\xi}}\partial_{x_2}v_1 \right) \\
=&\frac{x_2}{h}\left(\partial_{x_2}v_{\vec{\xi}}\partial_{x_1}v_3-\partial_{x_1}v_{\vec{\xi}}\partial_{x_2}v_3 \right)+\partial_{x_3}v_3\partial_{x_2}v_{\vec{\xi}}-\partial_{x_2}v_3\partial_{x_3}v_{\vec{\xi}}\\
=&\frac{x_2}{h} \partial_{x_2}v_{\vec{\xi}}\partial_{x_1}v_3+\frac{1}{h}\left(-x_2\partial_{x_1}v_3+x_1\partial_{x_2}v_3 \right)\partial_{x_2}v_{\vec{\xi}}-\frac{x_1}{h}\partial_{x_2}v_3\partial_{x_2}v_{\vec{\xi}}=0.
\end{split}
\end{equation}
Thus, we have $ \partial_t w_1+\left (\mathbf{v}\cdot\nabla\right )w_1-\left (\mathbf{w}\cdot\nabla\right )v_1=0. $ Similarly, we can get $ \partial_t w_2+\left (\mathbf{v}\cdot\nabla\right )w_2-\left (\mathbf{w}\cdot\nabla\right )v_2=0. $

To conclude,  $\mathbf{v}$ defined by \eqref{2-10} and its vorticity field $\mathbf{w}$ are helical vector  fields satisfying vorticity equations \eqref{1-2}. The proof is thus complete. 
\end{proof}

\textbf{Proof of Theorem \ref{thm1}:}
Theorem \ref{thm1} follows immediately from Propositions \ref{pp2-7} and \ref{pp2-8}.

\section{Elliptic equation for traveling-rotating helical vortices}\label{sec3}

Let $\bar \alpha$ be a fixed constant. For traveling-rotating  helical vortices of \eqref{1-2}, it suffices to find rotating invariant  solutions to \eqref{2-9} being of the following form:
\begin{equation}\label{3-1}
	w(x,t)=W(R_{\bar \alpha t}x),\ \ \  v(x,t)=V(R_{\bar \alpha t}x),\ \ \ \varphi(x,t)=\Phi(R_{\bar \alpha t}x). 
\end{equation}
Note that for a general domain $ D $ that may not be rotationally invariant, $\bar \alpha $ has to be 0. While when $D$ is rotationally invariant such as  a disk, an annulus or the whole space $\mathbb R^2$,  we can take   $\bar \alpha $  to be any real number. 

For a function $	f(x,t)=F(y)=F(R_{\bar \alpha t}x)$, straightforward computation gives the following identities:
\begin{equation*}
	\begin{split}
		\partial_t f(x,t)=-\bar \alpha (R_{\bar \alpha t}x)^\perp\cdot (\nabla_y F)(R_{\bar \alpha t}x)&= 	\nabla_y^\perp\left( -\frac{\bar \alpha}{2}|y|^2\right)\cdot \nabla_y F(y),\ \ \ y= R_{\bar \alpha t}x,\\
	\nabla_x 	f(x,t)=	&R_{\bar \alpha t} (\nabla_y F)(R_{\bar \alpha t}x).
	\end{split}
\end{equation*}
In what follows, we will use change of variable $x\mapsto y=R_{\bar \alpha t}x$ and still denote $y$ as $x$ for simplicity of notations. Substituting \eqref{3-1} into the first equation of \eqref{2-9}, we have
\begin{equation}\label{3-2}
	\nabla^\perp\left(\Phi(x)-\frac{\bar \alpha}{2}|x|^2\right)\cdot \nabla V(x)= 0. 
\end{equation}
 So if there exists a function $F_1$ such that 
\begin{equation}\label{3-3}
	V(x)=F_1\left(\Phi(x)-\frac{\bar \alpha}{2}|x|^2\right), 
\end{equation}  
then \eqref{3-2} holds.

Substituting \eqref{3-1} into the second equation in \eqref{2-9}, we obtain 
\begin{equation}\label{3-4}
	\nabla^\perp\left(\Phi -\frac{\bar \alpha}{2}|x|^2\right)\cdot \nabla W =\partial_{x_2} V\partial_{x_1} U_3-\partial_{x_1} V\partial_{x_2} U_3+\frac{2 V \left( x_2\partial_{x_1} V -x_1\partial_{x_2} V\right)}{|{\vec{\xi}}|^4},
\end{equation} 
where $U_3(x)=-\frac{x_1\partial_{x_1}\Phi(x)+x_2\partial_{x_2}\Phi(x)}{|{\vec{\xi}}|^2}$. From \eqref{3-3}, we observe that 
\begin{equation*}\begin{split}
		&\partial_{x_2} V\partial_{x_1} U_3-\partial_{x_1} V\partial_{x_2} U_3\\
		=& F_1'\left(\Phi(x)-\frac{\bar \alpha}{2}|x|^2\right)\left(\nabla^\perp\left(\Phi -\frac{\bar \alpha}{2}|x|^2\right)\cdot \nabla U_3 \right)\\
		=&\nabla^\perp\left(\Phi -\frac{\bar \alpha}{2}|x|^2\right)\cdot \nabla \left(F_1'\left(\Phi(x)-\frac{\bar \alpha}{2}|x|^2\right) U_3\right),
	\end{split}
\end{equation*}
and 
\begin{equation*} 
	\frac{2 V \left( x_2\partial_{x_1} V -x_1\partial_{x_2} V\right)}{|{\vec{\xi}}|^4} =\nabla^\perp\left(\Phi -\frac{\bar \alpha}{2}|x|^2\right)\cdot \nabla\left( \frac{F_1\left(\Phi(x)-\frac{\bar \alpha}{2}|x|^2\right) F_1'\left(\Phi(x)-\frac{\bar \alpha}{2}|x|^2\right) }{|{\vec{\xi}}|^2}\right). 
\end{equation*}
Taking the above two identities into \eqref{3-4}, we get 
\begin{equation}\label{3-5}
	\nabla^\perp\left(\Phi -\frac{\bar \alpha}{2}|x|^2\right)\cdot \nabla \left(W- F_1'\left(\Phi(x)-\frac{\bar \alpha}{2}|x|^2\right) U_3- \frac{F_1\left(\Phi(x)-\frac{\bar \alpha}{2}|x|^2\right) F_1'\left(\Phi(x)-\frac{\bar \alpha}{2}|x|^2\right) }{|{\vec{\xi}}|^2}\right)=0.
\end{equation} 
So if there exists a function $F_2$  such that
\begin{equation}\label{3-6}
	W(x)= F_1'\left(\Phi(x)-\frac{\bar \alpha}{2}|x|^2\right) U_3(x)+ \frac{F_1\left(\Phi(x)-\frac{\bar \alpha}{2}|x|^2\right) F_1'\left(\Phi(x)-\frac{\bar \alpha}{2}|x|^2\right) }{|{\vec{\xi}}|^2}+F_2\left(\Phi(x)-\frac{\bar \alpha}{2}|x|^2\right),
\end{equation} 
then \eqref{3-5} holds.

Now,  substituting \eqref{3-1} into the third equation of \eqref{2-9}, we obtain
\begin{equation}\label{3-7}
	\mathcal{L}_{K}\Phi=W+\frac{2 h^2 V }{|{\vec{\xi}}|^4}+\frac{ x_1\partial_{x_1} V +x_2\partial_{x_2} V}{|{\vec{\xi}}|^2}.
\end{equation} 
Taking $U_3(x)=-\frac{x_1\partial_{x_1}\Phi(x)+x_2\partial_{x_2}\Phi(x)}{|{\vec{\xi}}|^2}$, \eqref{def of xi}, \eqref{3-3} and \eqref{3-6} into  \eqref{3-7}, we deduce that 
\begin{equation}\label{3-8}
	\begin{split}
		\mathcal{L}_{K}\Phi(x)
		&=\frac{2 h^2    F_1\left(\Phi(x)-\frac{\bar \alpha}{2}|x|^2\right)}{(h^2+x_1^2+x_2^2)^2}+\frac{F_1\left(\Phi(x)-\frac{\bar \alpha}{2}|x|^2\right) F_1'\left(\Phi(x)-\frac{\bar \alpha}{2}|x|^2\right) }{h^2+x_1^2+x_2^2}\\
		&\quad -\frac{\bar \alpha |x|^2 F_1'\left(\Phi(x)-\frac{\bar \alpha}{2}|x|^2\right) }{h^2+x_1^2+x_2^2}+F_2\left(\Phi(x)-\frac{\bar \alpha}{2}|x|^2\right), \ \  \ x\in D.
	\end{split}
\end{equation} 

Therefore, we conclude that for rotating invariant solutions of \eqref{2-9} with angular velocity $ \bar{\alpha} $, it suffices to solve a non-autonomous semilinear elliptic equation 
for $ \Phi $:
\begin{equation}\label{3-9} 
\begin{cases}
\mathcal{L}_{K}\Phi=\frac{2h^2F_1\left(\Phi-\frac{\bar\alpha}{2}|x|^2 \right)}{(h^2+|x|^2)^2}+\frac{F_1\left(\Phi-\frac{\bar\alpha}{2}|x|^2 \right)F_1'\left(\Phi-\frac{\bar\alpha}{2}|x|^2 \right)}{h^2+|x|^2}-\frac{\bar\alpha |x|^2 F_1'\left(\Phi-\frac{\bar\alpha}{2}|x|^2 \right)}{h^2+|x|^2}+F_2\left(\Phi-\frac{\bar\alpha}{2}|x|^2 \right),\ \ &\text{in}\ \  D,\\
\Phi=0,\ \ &\text{on}\ \   \partial D,\\
\end{cases}
\end{equation}
where profile functions $F_1$ and $F_2$ can be chosen arbitrary. Let $ \zeta =\mathcal{L}_{K}\Phi $ and $\mathcal{G}_K $ be the Green's operator of the elliptic operator $ \mathcal{L}_{K} $ in $ D $ with Dirichlet condition (see Lemma \ref{lm4-1} in section 4).   Then $ \Phi=\mathcal{G}_K\zeta $ and \eqref{3-9} is equivalent to the following equations for $ \zeta $
\begin{equation}\label{3-10} 
\begin{split}
\zeta=&\frac{2h^2F_1\left(\mathcal{G}_K\zeta-\frac{\bar\alpha}{2}|x|^2 \right)}{(h^2+|x|^2)^2}+\frac{F_1\left(\mathcal{G}_K\zeta-\frac{\bar\alpha}{2}|x|^2 \right)F_1'\left(\mathcal{G}_K\zeta-\frac{\bar\alpha}{2}|x|^2 \right)}{h^2+|x|^2}\\
&-\frac{\bar\alpha |x|^2 F_1'\left(\mathcal{G}_K\zeta-\frac{\bar\alpha}{2}|x|^2 \right)}{h^2+|x|^2}+F_2\left(\mathcal{G}_K\zeta-\frac{\bar\alpha}{2}|x|^2 \right),\ \  \text{in}\ \  D.
\end{split}
\end{equation}
Solutions of  \eqref{3-10} correspond to traveling-rotating helical vortices of \eqref{1-2}. By  constructing   families of concentrated solutions to \eqref{3-10} with the proper choice of $ F_1 $ and $ F_2 $, 
 we will finally construct  families  of concentrated helical vortices with non-zero helical swirl in infinite cylinders, whose vorticity field $ \textbf{w} $ tends to a helical vortex filament  \eqref{1007} evolved by the binormal curvature flow as $ \varepsilon\to0 $.
\section{Construction of traveling-rotating helical vortices in an infinite pipe}\label{sec4}
As stated above,  we will prove the existence of concentrated solutions $ \zeta_\varepsilon $ of \eqref{3-10} for suitably selected $ F_1 $ and $ F_2 $, such that the corresponding helical vorticity field $ \mathbf{w}_\varepsilon $ satisfies \eqref{1-2} in an infinite cylinder without the assumption \eqref{1-7} and  concentrates near the helix \eqref{1007} in topological and distributional sense. In what follows, we always assume $h>0$ and $D=B_R(0)\in\mathbb R^2$ for some  $R>0$.

\subsection{Green's function of $\mathcal{L}_{K}$} 
To construct concentrated solutions of \eqref{3-10}, the first key point is the asymptotic expansion of Green's function of $\mathcal{L}_{K}$. Since the Green's operator $ \mathcal{G}_K $ in \eqref{3-10} is the same as that in \eqref{1-9}, we give some properties of the existence and decomposition of Green's function $ G_K(x,y) $ for the operator $\mathcal{L}_{K}$ in bounded domains, which has been already found in the previous literature \cite{CW2}. 
 
 \begin{lemma}[Theorem 1.3, \cite{CW2}]\label{lm4-1}
 	 Let $U\subset \mathbb R^2$ be a bounded domain with smooth boundary. Let $ \Gamma(x)=-\frac{1}{2\pi}\ln|x|$ be the fundamental solution of the Laplacian $ -\Delta $ in $ \mathbb{R}^2 $ and $ T_x $ be a $ C^\infty $ positive-definite matrix-valued function determined by $ K $ satisfying
 	 \begin{equation}\label{matrix T}
 	 	T_x^{-1}(T_x^{-1})^{t}=K(x)\ \ \ \ \ \ \forall x\in U.
 	 \end{equation}
Then there exists a function $G_K(x,y)$ defined on $U\times U$ such that for any $f\in L^q(U)$ with $q>2$, the solution of the following linear problem
 	  \begin{equation}\label{4-1}
 	 	\begin{cases}
 	 		\mathcal{L}_Ku=-\mathrm{div}(K(x)\nabla u)= f,\ \ &x\in U,\\
 	 		u=0,\ \ &x\in\partial U,
 	 	\end{cases}
 	 \end{equation}
  can be represented by 
     	\begin{equation}\label{4-2}
     	\begin{split}
     		u(x)=\mathcal{G}_Kf(x)=\int_{U}G_K(x,y)f(y)\mathrm d y,\ \ \ \  \ \ x\in U.
     	\end{split}
     \end{equation}
      Moreover, $G_K(x,y)$ has the following decomposition
      \begin{equation}\label{4-3}
      	G_K(x, y)=\frac{\sqrt{\det K(x)}^{-1}+\sqrt{\det K(y)}^{-1}}{2}\Gamma\left (\frac{T_x+T_y}{2}(x-y)\right )+H_0(x,y),
      \end{equation}
      where $H_0(x,y)\in C_{loc}^{0, \gamma}(  U\times U)$ for some $\gamma\in (0,1)$ and $$H_0(x,y)\leq C, \ \ \ \forall\, x,y\in U $$ for some constant $C$.
 \end{lemma}

\subsection{Variational setting and existence of maximizers}
We observe that profile functions $F_1$ and $F_2$ in \eqref{3-10} can be arbitrary.  
In the following,   we will choose
\begin{equation*} 
F_1(s)=\ep^{-1}a (s-\mu_\ep)_+,\ \ \   F_2(s)=\ep^{-2}(b+\bar\alpha a \ep) \textbf{1}_{\{s-\mu_\ep>0\}},\ \ \  s\in \mathbb R,
\end{equation*}
in \eqref{3-10}, where $0<\ep<1$, $a>0, b\geq0$ are arbitrary fixed constants, $\bar \alpha=\alpha|\ln\ep|$  with $\alpha=\frac{\kappa}{4\pi h\sqrt{h^2+r_*^2}}$ and $ \mu_\ep $ is a constant to be determined later, see Lemma \ref{lm4-2}. From \eqref{3-10}, we get  equations for $ \zeta $
	\begin{equation}\label{4-5}  
	\begin{split}
\zeta =&\frac{1}{\ep^2}\frac{2h^2a\ep +a^2(h^2+|x|^2)}{(h^2+|x|^2)^2 }\left (\mathcal{G}_K\zeta-\frac{\bar\alpha}{2}|x|^2-\mu_\ep\right )_+\\ &+\frac{1}{\ep^2}\left(\frac{h^2\bar \alpha a\ep}{h^2+|x|^2}+b\right)\textbf{1}_{\{ \mathcal{G}_K\zeta-\frac{\bar\alpha}{2}|x|^2-\mu_\ep>0\} }	\ \  \ \ a.e.\  \text{in}\  D.
\end{split}
\end{equation}
Note that the case $a=0$ corresponds to helical vortices with  the assumption \eqref{1-7}, which has been considered in \cite{CL23, CW, CW2, DDMW2}. To construct helical vortices of \eqref{1-2} without the assumption \eqref{1-7}, we   focus our attention on the case $a>0$.  

Set $|x|=(x_1^2+x_2^2)^\frac{1}{2}$. Let $ 	i_\ep\left(\cdot, \cdot\right) $ be a function defined by
\begin{equation}\label{4-6}
i_\ep\left(r, s\right)=\frac{2h^2 a \ep s_+}{(h^2+r^2)^2}+\frac{ a^2 s_+}{h^2+r^2 }+\frac{h^2 \bar\alpha a \ep \textbf{1}_{\{s>0\}}}{h^2+r^2 }+b\textbf{1}_{\{s>0\}},\ \ \ \ \ r,s\geq 0.
\end{equation} 
Then the right hand side of \eqref{4-5} equals $\ep^{-2} i_\ep\left(|x|, \mathcal{G}_K\zeta-\frac{\bar\alpha}{2}|x|^2-\mu_\ep\right)$.

Let $I_\ep(r,s)=\int_0^s i_\ep(r,\tau) \mathrm d \tau$  for $ s\geq 0 $ and $J_\ep(r, \cdot)$ be the  modified conjugate function of $I_\ep(r,\cdot)$ defined by 
\begin{equation*}
	J_\ep(r,s)=\begin{cases}
		\sup\limits_{t\in\mathbb{R}}[st-I_\ep(r,t)], \ \ \ &\text{if}\ \ s>0,\\
		0, \ \ \ &\text{if}\ \ s=0.
	\end{cases}
\end{equation*}  
Then direct calculation  yields
 \begin{equation}\label{4-8}
 	J_\ep\left(r, s\right)=\frac{(h^2+r^2)^2\left(s-\frac{h^2\bar\alpha a\ep}{h^2+r^2}-b\right)_+^2}{4h^2a\ep+2a^2(h^2+r^2)}.
 \end{equation}

Let $ \Lambda>1 $. We   consider maximization problem of the following  functional for $ \zeta $
   \begin{equation}\label{4-10}
   E_\ep(\zeta)=\frac{1}{2}\int_{D} \zeta\mathcal{G}_{K} \zeta \mathrm{d} x-\frac{\bar{\alpha}}{2}\int_{D} |x|^2 \zeta \mathrm{d} x-\ep^{-2}\int_{D}  J_\ep(|x|, \ep^2\zeta(x)) \mathrm d x, 
   \end{equation}
over the constraint set
\begin{equation}\label{4-9}
	\mathcal{M}_{\kappa, \Lambda} =\left\{ \zeta\in L^1\cap L^\infty(D) \ \Big| \ 0\leq \zeta\leq \frac{\Lambda}{\ep^2},\,\,\, \int_{D} \zeta dx=\kappa\right\}. 
\end{equation}
The choice of $ \Lambda $ will be given in Lemma \ref{lm4-11}.

Our goal is to find maximizers of $E_\ep$ relative to $\mathcal{M}_{\kappa, \Lambda}$ and give their asymptotic behavior as $ \varepsilon\to 0 $. To begin with, we  prove the existence and profiles of maximizers of $E_\ep$.

\begin{lemma}\label{lm4-2}
	For every $\ep\in (0,1)$ and $\Lambda>\max\{\alpha a+b,1+\pi R^2/\kappa\}$, 
	there exists $ \zeta_{\ep, \Lambda} \in \mathcal{M}_{\ep,\Lambda}$ being a maximizer  of $E_\ep$ over $\mathcal{M}_{\kappa, \Lambda}$, i.e., 
	\begin{equation}\label{4-12}
		E_\ep(\zeta_{\ep, \Lambda})= \max_{\zeta\in \mathcal{M}_{\ep,\Lambda}}E_\ep(\zeta)<+\infty.
	\end{equation}
	Moreover, there exists a Lagrange multiplier $ \mu_{\ep, \Lambda} $ with $$\mu_{\ep, \Lambda} \ge- \frac{\alpha R^2}{2}\ln\frac{1}{\ep}-\frac{(h^2+R^2)^2\left(  \Lambda -\frac{h^2\bar\alpha a\ep}{h^2+R^2}-b\right)_+}{2h^2a \ep+ a^2(h^2+R^2)},$$ such that
	\begin{equation}\label{4-13}
	\begin{split}
	\zeta_{\ep, \Lambda}=&\frac{1}{\ep^2}\left(\frac{2h^2 a \ep +a^2(h^2+r^2)}{(h^2+r^2)^2 }\psi_{\ep, \Lambda}+\frac{h^2\bar\alpha a\ep}{h^2+r^2}+b\right)\textbf{1}_{\{0<\psi_{\ep, \Lambda}<(h^2+r^2)^2(\Lambda-\frac{h^2\bar\alpha a\ep}{h^2+r^2}-b)/(2h^2a\ep +a^2(h^2+r^2))\} }\\
	&+\frac{\Lambda}{\ep^2}\textbf{1}_{\{\psi_{\ep, \Lambda}\geq(h^2+r^2)^2(\Lambda-\frac{h^2\bar\alpha a\ep}{h^2+r^2}-b)/(2h^2a\ep +a^2(h^2+r^2))\}} \ \ a.e.\  \text{in}\  D,
		\end{split}
	\end{equation}
	where $ r=|x| $ and 
	\begin{equation}\label{4-14}
		\psi_{\ep, \Lambda}=\mathcal{G}_K\zeta_{\ep, \Lambda}-\frac{\bar{\alpha} r^2}{2}-\mu_{\ep, \Lambda}.
	\end{equation}
    
\end{lemma}

\begin{proof}
	We may take a sequence $\zeta_{j}\in \mathcal{M}_{\ep,\Lambda}$ such that as $j\to +\infty$
	\begin{equation*}
		\begin{split}
			E_\ep(\zeta_{j}) & \to \sup\{E_\ep(\zeta)~|~\zeta\in \mathcal{M}_{\ep,\Lambda}\}.
		\end{split}
	\end{equation*}
	Since $ \mathcal{M}_{\ep,\Lambda} $ is closed in $ L^\infty(D) $ weak star topology and $ L^2(D) $ weak  topology, there exists $ \zeta_{\ep, \Lambda}\in L^{\infty}(D) $ such that as $j\to +\infty$
	\begin{equation*}
		\zeta_{j}\to \zeta_{\ep, \Lambda} \ \ \ \ \ \ ~~\text{weakly~in}\ L^{2}(D)\ \text{and~weakly~star~in}\ L^{\infty}(D).
	\end{equation*}
	It is easily checked that $\zeta_{\ep, \Lambda}\in \mathcal{M}_{\ep,\Lambda}$. Since the Green's operator $ \mathcal{G}_{K} $ is  compact  from $ L^2(D) $ to $L^2(D)$, we have
	\begin{equation*}
		\lim_{j\to +\infty}\int_D{\zeta_{j} \mathcal{G}_{K}\zeta_{j}}   \mathrm d x = \int_D{\zeta_{\ep, \Lambda} \mathcal{G}_{K}\zeta_{\ep, \Lambda}}    \mathrm d x.
	\end{equation*}
	On the other hand, it follows from $|x|^2\in L^2(D)$ that
	\begin{equation*}
		\begin{split}
			\lim_{j\to +\infty} \int_{D}|x|^2\zeta_j   \mathrm d x    = \int_{D}|x|^2\zeta_{\ep, \Lambda}     \mathrm d x.
		\end{split}
	\end{equation*}
Moreover, due to the convexity of $J_\ep(r,s)$ relative to $s$, we infer from the lower semicontinuity of $J_\ep(r,s)$  that
	\begin{equation*}
			\liminf_{j\to +\infty}\int_D J_\ep(r, \ep^2\zeta_j(x))    \mathrm d x \ge \int_D J_\ep(r, \ep^2\zeta_{\ep, \Lambda}(x))   \mathrm d x.
	\end{equation*}
To conclude, we have $E_\ep(\zeta_{\ep, \Lambda})=\lim_{j\to +\infty}E_\ep(\zeta_j)=\sup_{\zeta\in \mathcal{M}_{\ep,\Lambda}} E_\ep(\zeta)$. The existence of maximizers of $E_\ep$ over $\mathcal{M}_{\kappa, \Lambda}$ is proved.
	
	We now turn to prove the profile $\eqref{4-13}$ of maximizers $ \zeta_{\ep, \Lambda} $. We may assume $\zeta_{\ep, \Lambda} \not\equiv 0$, otherwise the conclusion is obtained by letting $\mu_{\ep, \Lambda}=0$.
	We consider a family of variations of $\zeta_{\ep, \Lambda}$
	\begin{equation*}
		\zeta_{(s)}=\zeta_{\ep, \Lambda}+s(\zeta-\zeta_{\ep, \Lambda}),\ \ \ s\in[0,1],
	\end{equation*}
	defined for arbitrary $\zeta\in \mathcal{M}_{\ep,\Lambda}$. Since $\zeta_{\ep, \Lambda}$ is a maximizer, one computes directly that
	\begin{equation*}
		\begin{split}
			0 & \ge \frac{d}{ds}E_\ep(\zeta_{(s)})|_{s=0^+} \\
			& =\int_{D}(\zeta-\zeta_{\ep, \Lambda})\left(\mathcal{G}_K\zeta_{\ep, \Lambda}-\frac{\bar{\alpha} r^2}{2}  - \frac{(h^2+r^2)^2\left(\ep^2\zeta_{\ep, \Lambda}-\frac{h^2\bar \alpha a\ep}{h^2+r^2}-b\right)_+}{2h^2a\ep+ a^2(h^2+r^2)}\right)  \mathrm d x.
		\end{split}
	\end{equation*}
	This implies that for any  $\zeta\in \mathcal{M}_{\ep,\Lambda}$, there holds
	\begin{equation*}
		\begin{split}
				\int_{D}\zeta_{\ep, \Lambda} \left(\mathcal{G}_K\zeta_{\ep, \Lambda}-\frac{\bar{\alpha} r^2}{2} - \frac{(h^2+r^2)^2\left(\ep^2\zeta_{\ep, \Lambda}-\frac{h^2\bar \alpha a\ep}{h^2+r^2}-b\right)_+}{2h^2a\ep+ a^2(h^2+r^2)}\right)   \mathrm d x \\
			\ge \int_{D}\zeta   \left(\mathcal{G}_K\zeta_{\ep, \Lambda}-\frac{\bar{\alpha} r^2}{2} - \frac{(h^2+r^2)^2\left(\ep^2\zeta_{\ep, \Lambda}-\frac{h^2\bar \alpha a\ep}{h^2+r^2}-b\right)_+}{2h^2a\ep+ a^2(h^2+r^2)}\right)\ \mathrm d x.
		\end{split}
	\end{equation*}
	By  the classical bathtub principle (see \cite{Lieb}, \S1.14x), we obtain that there exists a  Lagrange multiplier $ \mu_{\ep, \Lambda} $ such taht 
	\begin{equation}\label{4-15}
		\begin{split}
			\mathcal{G}_K\zeta_{\ep, \Lambda} -\frac{\bar{\alpha} r^2}{2}-\frac{(h^2+r^2)^2\left(\ep^2\zeta_{\ep, \Lambda}-\frac{h^2\bar \alpha a\ep}{h^2+r^2}-b\right)_+}{2h^2a\ep+ a^2(h^2+r^2)} &\ge \mu_{\ep, \Lambda}  \ \ \  \text{whenever}\  \zeta_{\ep, \Lambda} =\frac{\Lambda}{\ep^2}, \\
			\mathcal{G}_K\zeta_{\ep, \Lambda} -\frac{\bar{\alpha} r^2}{2}-\frac{(h^2+r^2)^2\left(\ep^2\zeta_{\ep, \Lambda}-\frac{h^2\bar \alpha a\ep}{h^2+r^2}-b\right)_+}{2h^2a\ep+ a^2(h^2+r^2)} &=\mu_{\ep, \Lambda}  \ \ \ \text{whenever}\  0<\zeta_{\ep, \Lambda} <\frac{\Lambda}{\ep^2}, \\
			\mathcal{G}_K\zeta_{\ep, \Lambda} -\frac{\bar{\alpha} r^2}{2}- \frac{(h^2+r^2)^2\left(\ep^2\zeta_{\ep, \Lambda}-\frac{h^2\bar \alpha a\ep}{h^2+r^2}-b\right)_+}{2h^2a\ep+ a^2(h^2+r^2)}& \le\mu_{\ep, \Lambda} \ \ \  \text{whenever}\  \zeta_{\ep, \Lambda} =0.
		\end{split}
	\end{equation}
Moreover, $ \mu_{\ep, \Lambda} $ satisfies
	\begin{equation}\label{4-16}
		\mu_{\ep, \Lambda}=\inf\left\{t:\left|\left\{\mathcal{G}_K\zeta_{\ep, \Lambda}-\frac{\bar{\alpha} r^2}{2}- \frac{(h^2+r^2)^2\left(\ep^2\zeta_{\ep, \Lambda}-\frac{h^2\bar\alpha a\ep}{h^2+r^2}-b\right)_+}{2h^2a\ep+ a^2(h^2+r^2)} >t\right\}\right|\le \frac{\kappa\ep^2}{\Lambda}\right\}\in \mathbb{R}.
	\end{equation} 
	Set $$	\psi_{\ep, \Lambda}=\mathcal{G}_K\zeta_{\ep, \Lambda}-\frac{\bar{\alpha} r^2}{2}-\mu_{\ep, \Lambda}.$$ Note that   from the second equation in \eqref{4-15}, we have $\{0<\zeta_{\ep, \Lambda}\le \ep^{-2}(h^2\bar \alpha a\ep/(h^2+r^2)+b)  \}\subseteq\{\psi_{\ep, \Lambda}=0\}$. While on the set $\{\psi_{\ep, \Lambda}=0\}$, by   properties of Sobolev spaces, one has 
	\begin{equation*}
		0=\mathcal{L}_{K}(\mathcal{G}_K\zeta_{\ep, \Lambda})-\mathcal{L}_{K}\left (\frac{\bar{\alpha} |x|^2}{2}\right )=\zeta_{\ep, \Lambda}+\frac{2 h^4\bar{\alpha}}{(h^2+r^2)^2}>0\ \ \ \text{a.e.}
	\end{equation*}
	 Thus, we conclude that the Lesbague measure of $ \{\psi_{\ep, \Lambda}=0\} $ is 0, which implies that $|\{0<\zeta_{\ep, \Lambda}\le  \ep^{-2}(h^2\bar \alpha a\ep/(h^2+r^2)+b)\}|=|\{\psi_{\ep, \Lambda}=0\}|=0$. 
 $\eqref{4-13}$ follows immediately from the above conclusion and \eqref{4-15}.

	It remains to show that  $  \mu_{\ep, \Lambda} \ge- \frac{\alpha R^2}{2}\ln\frac{1}{\ep}-\frac{(h^2+R^2)^2\left(  \Lambda -\frac{h^2\bar \alpha a\ep}{h^2+R^2}-b\right)_+}{2h^2a \ep+ a^2(h^2+R^2)}$. We prove it  by contradiction. Suppose on the contrary that $\mu_{\ep, \Lambda} < - \frac{\alpha R^2}{2}\ln\frac{1}{\ep}-\frac{(h^2+R^2)^2\left(  \Lambda -\frac{h^2\bar \alpha a\ep}{h^2+R^2}-b\right)_+}{2h^2a \ep+ a^2(h^2+R^2)}$. On the one hand, from \eqref{4-16}, we find that the Lesbague measure of the set  \begin{equation*}
		 \left\{\mathcal{G}_K\zeta_{\ep, \Lambda}-\frac{\bar{\alpha} r^2}{2} - \frac{(h^2+r^2)^2\left(\ep^2\zeta_{\ep, \Lambda}-\frac{h^2\bar \alpha a\ep}{h^2+r^2}-b\right)_+}{2h^2a\ep+ a^2(h^2+r^2)} >- \frac{\alpha R^2}{2}\ln\frac{1}{\ep}-\frac{(h^2+R^2)^2\left(  \Lambda -\frac{h^2\bar \alpha a\ep}{h^2+R^2}-b\right)_+}{2h^2a \ep+ a^2(h^2+R^2)}\right\}  
	\end{equation*}  
is not larger than $ \frac{\kappa\ep^2}{\Lambda}$. 

On the other hand, it can be seen that for any $x\in D$, $\mathcal{G}_K\zeta_{\ep, \Lambda}(x)>0$ and 
$$\frac{\bar{\alpha} r^2}{2} + \frac{(h^2+r^2)^2\left(\ep^2\zeta_{\ep, \Lambda}-\frac{h^2\bar \alpha a\ep}{h^2+r^2}-b\right)_+}{2h^2a\ep+ a^2(h^2+r^2)} \leq \frac{\alpha R^2}{2}\ln\frac{1}{\ep}+\frac{(h^2+R^2)^2\left(  \Lambda -\frac{h^2\bar \alpha a\ep}{h^2+R^2}-b\right)_+}{2h^2a \ep+ a^2(h^2+R^2)}.$$ This implies that the set 
 \begin{equation*}
 \left\{\mathcal{G}_K\zeta_{\ep, \Lambda}-\frac{\bar{\alpha} r^2}{2}  - \frac{(h^2+r^2)^2\left(\ep^2\zeta_{\ep, \Lambda}-\frac{h^2\bar \alpha a\ep}{h^2+r^2}-b\right)_+}{2h^2a\ep+ a^2(h^2+r^2)} >- \frac{\alpha R^2}{2}\ln\frac{1}{\ep}-\frac{(h^2+R^2)^2\left(  \Lambda -\frac{h^2\bar \alpha a\ep}{h^2+R^2}-b\right)_+}{2h^2a \ep+ a^2(h^2+R^2)}\right\} 
\end{equation*}  equals exactly to $D$, whose measure is $\pi R^2$. This immediately leads to a contradiction, since  $\frac{\kappa\ep^2}{\Lambda}<\pi R^2$ due to   assumptions $\ep<1$ and $\Lambda>\pi R^2/\kappa$. The proof is thus complete.
\end{proof}

We observe that the appearance of last ``vortex patch'' term of $ \zeta_{\ep, \Lambda} $ in profile \eqref{4-13} makes  $ \zeta_{\ep, \Lambda} $ to not satisfy \eqref{4-5}, which implies that $  \mathcal{G}_K\zeta_{\ep, \Lambda} $ does not satisfy the semilinear equation \eqref{3-9}. 
In  the following, by analysing the asymptotic behavior of  maximizers
$ \zeta_{\ep, \Lambda} $, we will show that the last term of $ \zeta_{\ep, \Lambda} $ in \eqref{4-13} indeed vanishes when choosing the parameter $ \Lambda $ sufficiently large.
 \subsection{lower bounds of   $ E_\ep(\zeta_{\ep,\Lambda}) $ and $ \mu_{\ep,\Lambda} $}
For convenience, we will use $ C_1, C_2,\cdots $ to denote generic positive constants independent of $ \ep $ and $ \Lambda $ that may change from line to line. Let us define a function $ Y(x):D\to\mathbb{R} $ which will be frequently used in this subsection. For $ x\in D$, we define
\begin{equation}\label{4-17}
	Y(x):=\frac{\kappa}{2\pi\sqrt{\det K(x)}}-\alpha|x|^2=\frac{\kappa\sqrt{h^2+|x|^2}}{2\pi h}-\alpha|x|^2.
\end{equation}
Clearly, $ Y $ is radially symmetric in $ D=B_{R}(0) $. Moreover, one computes directly that
\begin{lemma}\label{lm4-3}
The maximizers set of $ Y $ in $ D $ is $ \{x\mid |x|=r_*\} $. That is, $ Y|_{\partial B_{r_*}(0)}=\max_{D}Y $. Moreover, up to a rotation the maximizer is unique.
\end{lemma}

We first give lower bounds of   $ E_\ep(\zeta_{\ep,\Lambda}) $ by choosing proper test functions.

\begin{lemma}\label{lm4-4}
	For any $z\in D$, there exists a constant $C>0$ independent of $\ep $ and $ \Lambda$, such that for all $\varepsilon$ sufficiently small 
	\begin{equation*}
		E_\ep(\zeta_{\ep,\Lambda})\ge \frac{\kappa}{2}Y(z)\ln\frac{1}{\varepsilon}-C.
	\end{equation*}
	As a consequence, there holds
	\begin{equation}\label{4-18}
		E_\ep(\zeta_{\ep,\Lambda})\ge \frac{\kappa}{2}Y((r_*,0))\ln\frac{1}{\varepsilon}-C.
	\end{equation}
    
\end{lemma}

\begin{proof}
	For any $z\in D$, we choose a test function $\hat{\zeta}_{\varepsilon} $ defined by
	\begin{equation*}
		\hat{\zeta}_{\varepsilon}=\frac{1}{\varepsilon^2}\textbf{1}_{T_z^{-1}B_{r_\varepsilon}(0)+z},
	\end{equation*}
	where the matrix $ T_z $ satisfies \eqref{matrix T}, and $ r_\varepsilon>0 $ satisfies
	\begin{equation}\label{4-19}
		\pi r_\varepsilon^2\sqrt{\det K(z)}=\varepsilon^2\kappa.
	\end{equation}
	Thus $ \hat{\zeta}_{\varepsilon} \in \mathcal{M}_{\varepsilon,\Lambda} $ for $ \varepsilon $ sufficiently small. Since $\zeta_{\ep,\Lambda}$ is a maximizer, we have $E_\ep(\zeta_{\ep,\Lambda})\ge E_\ep(\hat{\zeta}_{\varepsilon})$. By Lemma \ref{lm4-1}, we get
	\begin{equation*}
		\begin{split}
			E_\ep(\hat{\zeta}_{\varepsilon})=& \frac{1}{2}\int_D\int_D\hat{\zeta}_{\varepsilon}(x)G_{K}(x,y)\hat{\zeta}_{\varepsilon}(y)\mathrm d x\mathrm d y-{\frac{\alpha}{2} \ln{\frac{1}{\varepsilon}}}\int_D|x|^2\hat{\zeta}_{\varepsilon}  \mathrm d x-\ep^{-2}\int_D J_\ep(r, \ep^2 \hat{\zeta}_{\varepsilon}(x))  \mathrm d x \\
			\geq&\frac{1}{2}\int_{T_z^{-1}B_{r_\varepsilon}(0)+z}\int_{T_z^{-1}B_{r_\varepsilon}(0)+z}\frac{\sqrt{\det K(x)}^{-1}+\sqrt{\det K(y)}^{-1}}{2\varepsilon^4}\Gamma\left (\frac{T_x+T_y}{2}(x-y)\right)   \mathrm d x\mathrm d y\\
			& -{\frac{\alpha}{2} \ln{\frac{1}{\varepsilon}}}\int_D|x|^2\hat{\zeta}_{\varepsilon}  \mathrm d x-C_1.
		\end{split}
	\end{equation*}
	Since $ r_\varepsilon=O(\varepsilon),  $ by the positive-definiteness and regularity of $ K $, we have
	\begin{equation*}
		\sqrt{\det K(x)}^{-1}=\sqrt{\det K(z)}^{-1}+O(\varepsilon),\ \ \ \ \forall x\in T_z^{-1}B_{r_\varepsilon}(0)+z,
	\end{equation*}
	\begin{equation*}
		\Gamma\left (\frac{T_x+T_y}{2}(x-y)\right)=\Gamma\left (T_z(x-y)\right)+O(\varepsilon),\ \ \ \ \forall x,y\in T_z^{-1}B_{r_\varepsilon}(0)+z,\ x\neq y.
	\end{equation*}
	Thus one computes directly that
	\begin{equation*}
		\begin{split}
			E_\ep(\hat{\zeta}_{\varepsilon})
			\geq&\frac{1}{2\varepsilon^4}\int_{T_z^{-1}B_{r_\varepsilon}(0)+z}\int_{T_z^{-1}B_{r_\varepsilon}(0)+z}\left (\sqrt{\det K(z)}^{-1}+O(\varepsilon)\right ) \left( \Gamma\left (T_z(x-y)\right)+O(\varepsilon)\right )  \mathrm d x\mathrm d y\\
			& -{\frac{\kappa\alpha(|z|+O(\varepsilon))^2}{2} \ln{\frac{1}{\varepsilon}}}-C_1\\
			\ge& \frac{\sqrt{\det K(z)}^{-1}}{4\pi\varepsilon^4}\int_{T_z^{-1}B_{r_\varepsilon}(0)+z}\int_{T_z^{-1}B_{r_\varepsilon}(0)+z}\ln{\frac{1}{|T_z(x-y)|}}  \mathrm d x\mathrm d y -{\frac{\kappa\alpha|z|^2}{2} \ln{\frac{1}{\varepsilon}}}-C_2\\
			=& \frac{\sqrt{\det K(z)}^{-1}}{4\pi\varepsilon^4}\int_{B_{r_\varepsilon}(0)}\int_{B_{r_\varepsilon}(0)}\ln{\frac{1}{|x'-y'|}}\cdot \det K(z) \mathrm d x'\mathrm d y' -{\frac{\kappa\alpha|z|^2}{2} \ln{\frac{1}{\varepsilon}}}-C_2\\
			\geq&\frac{\sqrt{\det K(z)}}{4\pi\varepsilon^4}(\pi r_\varepsilon^2)^2\ln\frac{1}{\varepsilon}-{\frac{\kappa\alpha|z|^2}{2} \ln{\frac{1}{\varepsilon}}}-C_3\\
			=& \frac{\kappa}{2}Y(z)\ln\frac{1}{\varepsilon}-C_4,
		\end{split}
	\end{equation*}
	where we have used \eqref{matrix T} and \eqref{4-19}. By choosing $ z=(r_*,0) $ and using Lemma \ref{lm4-4}, we get \eqref{4-18}. The proof is thus finished.
\end{proof}

Based on the lower bound of $ E_\ep(\zeta_{\ep,\Lambda}) $ in Lemma \ref{lm4-4}, we are   able to refine the lower bound of Lagrange multiplier $\mu_{\ep,\Lambda}$ in Lemma \ref{lm4-2}.
\begin{lemma}\label{lm4-5}
	There holds for $ \varepsilon $ sufficiently small 
	$$\mu_{\ep,\Lambda}\ge \frac{2}{\kappa}E_\varepsilon(\zeta_{\ep,\Lambda})+ \frac{\alpha}{2\kappa}\ln\frac{1}{\varepsilon}\int_D |x|^2 \zeta_{\ep,\Lambda} \mathrm d x-C,$$
where $C$ is a constant independent of $\ep $ and $ \Lambda$.	As a consequence, it holds 
	\begin{equation*}
		\mu_{\ep,\Lambda}\ge Y((r_*,0))\ln\frac{1}{\varepsilon}+ \frac{\alpha}{2\kappa}\ln\frac{1}{\varepsilon}\int_D |x|^2 \zeta_{\ep,\Lambda} \mathrm d x-C.
	\end{equation*}

\end{lemma}

\begin{proof}
	It follows from the definition of $ E_\varepsilon $ and \eqref{4-15} that
	\begin{align*} 
		&\quad 2E_\ep(\zeta_{\ep,\Lambda})= \int_{D}\zeta_{\ep,\Lambda}\mathcal{G}_{K} \zeta_{\ep,\Lambda} \mathrm{d} x- \alpha \ln \frac{1}{\ep}\int_{D} |x|^2 \zeta_{\ep,\Lambda} \mathrm{d} x-2\ep^{-2}\int_{D}  J_\ep(r, \ep^2\zeta_{\ep,\Lambda}(x))    \mathrm d x\nonumber\\
		&= \int_{D}\zeta_{\ep,\Lambda} \psi_{\ep,\Lambda}   \mathrm d x- \frac{\alpha}{2} \ln \frac{1}{\ep}\int_{D} |x|^2 \zeta_{\ep,\Lambda} \mathrm{d} x-2\ep^{-2}\int_{D}  J_\ep(r, \ep^2\zeta_{\ep,\Lambda}(x))+\mu_{\ep,\Lambda} \kappa\nonumber\\
		&=\int_{\{ \ep^{-2}(h^2\bar \alpha a\ep/(h^2+r^2)+b)<\zeta_{\ep, \Lambda}< \Lambda \ep^{-2}\}}\zeta_{\ep,\Lambda}  \frac{(h^2+r^2)^2\left(\ep^2\zeta_{\ep, \Lambda}-\frac{h^2\bar \alpha a\ep}{h^2+r^2}-b\right)_+}{2h^2a\ep+ a^2(h^2+r^2)}\mathrm{d} x\nonumber\\
		 &\ \ \ +\int_{\{  \zeta_{\ep, \Lambda}= \Lambda \ep^{-2}\}}\zeta_{\ep,\Lambda}  \psi_{\ep,\Lambda} \mathrm d x- \ep^{-2}\int_{D} \frac{(h^2+r^2)^2\left(\ep^2\zeta_{\ep,\Lambda}(x)-\frac{h^2\bar \alpha a\ep}{h^2+r^2}-b\right)_+^2}{2h^2a\ep+ a^2(h^2+r^2)} \mathrm d x\nonumber\\
		 &\ \  \  - \frac{\alpha}{2} \ln \frac{1}{\ep}\int_{D} |x|^2 \zeta_{\ep,\Lambda} \mathrm{d} x +\mu_{\ep,\Lambda}\kappa \nonumber\\ 
		&=\ep^{-2}\int_{D}  \left(\frac{h^2\bar \alpha a\ep}{h^2+r^2}+b\right)\frac{(h^2+r^2)^2}{2h^2a\ep+ a^2(h^2+r^2)}\left(\ep^2\zeta_{\ep, \Lambda}-\frac{h^2\bar \alpha a\ep}{h^2+r^2}-b\right)_+\mathrm{d} x\nonumber \\
		&\  \ \ +\int_{\{  \zeta_{\ep, \Lambda}= \Lambda \ep^{-2}\}}\zeta_{\ep,\Lambda}  \left[\psi_{\ep,\Lambda}-\frac{(h^2+r^2)^2\left(  \Lambda -\frac{h^2\bar \alpha a\ep}{h^2+r^2}-b\right)_+}{2h^2a\ep+ a^2(h^2+r^2)}\right]\nonumber\\
		&\ \ \ - \frac{\alpha}{2} \ln \frac{1}{\ep}\int_{D} |x|^2 \zeta_{\ep,\Lambda} \mathrm{d} x +\mu_{\ep,\Lambda}\kappa,\nonumber
	\end{align*}	
	from which we obtain 
	\begin{align}\label{4-20}
				&\mu_{\ep,\Lambda} \kappa =2E_\ep(\zeta_{\ep,\Lambda})+\frac{\alpha}{2} \ln \frac{1}{\ep}\int_{D} |x|^2 \zeta_{\ep,\Lambda} \mathrm{d} x\\
				&\ \   \ -\ep^{-2}\int_{D} \left(\frac{h^2\bar \alpha a\ep}{h^2+r^2}+b\right)\frac{(h^2+r^2)^2}{2h^2a\ep+ a^2(h^2+r^2)}\left(\ep^2\zeta_{\ep, \Lambda}-\frac{h^2\bar \alpha a\ep}{h^2+r^2}-b\right)_+\mathrm{d} x -\frac{\Lambda}{\ep^2}\int_{D}  U_{\ep,\Lambda}(x) \mathrm{d} x,\nonumber
	\end{align}
	where $ U_{\ep,\Lambda}(x) =\left[\psi_{\ep,\Lambda}(x)-\frac{(h^2+r^2)^2\left(  \Lambda -\frac{h^2\bar \alpha a\ep}{h^2+r^2}-b\right)_+}{2h^2a\ep+ a^2(h^2+r^2)}\right]_+$. Note that due to \eqref{4-15}, $U_{\ep,\Lambda}(x)=0$ for any $x\in\{ \zeta_{\ep, \Lambda}<\Lambda \ep^{-2}\}$. 
	
	In what follows, we will prove that the last two terms in the right hand side of \eqref{4-20} are uniformly bounded independent of $\ep$ and $\Lambda$.  
	
	By direct calculations, we get 
	\begin{equation*}
	\begin{split}
				\ep^{-2}\int_{D}  & \left(\frac{h^2\bar \alpha a\ep}{h^2+r^2}+b\right)\frac{(h^2+r^2)^2}{2h^2a\ep+ a^2(h^2+r^2)}\left(\ep^2\zeta_{\ep, \Lambda}-\frac{h^2\bar \alpha a\ep}{h^2+r^2}-b\right)_+\mathrm{d} x\\
			\leq& \int_{D} \left(\frac{h^2\bar \alpha a }{h^2+r^2}+b\right)\frac{(h^2+r^2)^2}{  a^2(h^2+r^2)} \zeta_{\ep, \Lambda} \mathrm{d} x\leq C\int_{D}  \zeta_{\ep, \Lambda} \mathrm{d} x=C\kappa.
		\end{split}
	\end{equation*}
    This demonstrates the uniform boundedness of the penultimate term on the right side of \eqref{4-20}.

As for the term $\frac{\Lambda}{\ep^2}\int_{D}  U_{\ep,\Lambda}(x) \mathrm{d} x$,  
by \eqref{4-14} and $\mu_{\ep,\Lambda} \ge - \frac{\alpha R^2}{2}\ln\frac{1}{\ep}-\frac{(h^2+R^2)^2\left(  \Lambda -\frac{h^2\bar \alpha a\ep}{h^2+R^2}-b\right)_+}{2h^2a \ep+ a^2(h^2+R^2)}$ in Lemma \ref{lm4-2}, we have $$\psi_{\ep,\Lambda}(x)-\frac{(h^2+R^2)^2\left(  \Lambda -\frac{h^2\bar \alpha a\ep}{h^2+R^2}-b\right)_+}{2h^2a \ep+ a^2(h^2+R^2)}\leq0 $$ for $ x\in \partial D. $ So $ U_{\ep,\Lambda}\in H^1_0(D). $
 Denote $ A_\varepsilon=\{x\in \Omega\mid U_{\ep,\Lambda}(x)>0\} =\{ \zeta_{\ep, \Lambda}=\Lambda \ep^{-2}\}$. Since $|A_\ep|\frac{\Lambda}{ \ep^2}\le \int_{D}  \zeta_{\ep,\Lambda} \mathrm{d} x=\kappa$, we obtain
  \begin{equation}\label{4-21}
		|A_\ep|\leq \frac{\kappa \ep^2}{\Lambda}. 
	\end{equation}
    
	In view of \eqref{4-13}, we infer from the definition of $\psi_{\ep,\Lambda}$ that 
	\begin{equation}\label{4-22}
		\mathcal{L}_{K}\left[\psi_{\ep,\Lambda}(x)-\frac{(h^2+r^2)^2\left(  \Lambda -\frac{h^2\bar \alpha a\ep}{h^2+r^2}-b\right)_+}{2h^2a\ep+ a^2(h^2+r^2)}\right]= \zeta_{\ep,\Lambda}-\frac{\bar{\alpha}\mathcal{L}_{K}|x|^2}{2}-\mathcal{L}_{K} \frac{(h^2+r^2)^2\left(  \Lambda -\frac{h^2\bar \alpha a\ep}{h^2+r^2}-b\right)_+}{2h^2a\ep+ a^2(h^2+r^2)}. 
	\end{equation}
	Multiplying  $U_{\ep,\Lambda} $ to both sides of \eqref{4-22} and integrating by parts, we get
	\begin{equation}\label{4-23}\begin{split}
					\int_D\nabla U_{\ep,\Lambda} \cdot ( K(x)\nabla U_{\ep,\Lambda})\mathrm d x= \frac{\Lambda}{ \ep^2}\int_D  U_{\ep,\Lambda}\mathrm d x-\frac{\bar{\alpha}}{2}\int_D(\mathcal{L}_{K}|x|^2) U_{\ep,\Lambda} \mathrm d x\\
					-\int_D\left(\mathcal{L}_{K}\frac{(h^2+r^2)^2\left(  \Lambda -\frac{h^2\bar \alpha a\ep}{h^2+r^2}-b\right)_+}{2h^2a\ep+ a^2(h^2+r^2)}\right) U_{\ep,\Lambda} \mathrm d x.
		\end{split}
	\end{equation}
	
Note that $ \sup_{x\in D} |\mathcal{L}_{K}|x|^2|\leq c$ and $ \sup_{x\in D} \left|\mathcal{L}_{K}\frac{(h^2+r^2)^2\left(  \Lambda -\frac{h^2\bar \alpha a\ep}{h^2+r^2}-b\right)_+}{2h^2a\ep+ a^2(h^2+r^2)}\right|\leq c(\Lambda+1)$ for some constant $c$ independent of $\ep $ and $ \Lambda$. Thus we have
	\begin{align}\label{4-24}
		 &\left|  \frac{\Lambda}{ \ep^2}\int_D  U_{\ep,\Lambda}\mathrm d x-\frac{\bar{\alpha}}{2}\int_D(\mathcal{L}_{K}|x|^2) U_{\ep,\Lambda} \mathrm d x 
		-\int_D\left(\mathcal{L}_{K}\frac{(h^2+r^2)^2\left(  \Lambda -\frac{h^2\bar \alpha a\ep}{h^2+r^2}-b\right)_+}{2h^2a\ep+ a^2(h^2+r^2)}\right) U_{\ep,\Lambda} \mathrm d x\right|\nonumber\\
		\leq& \frac{C\Lambda}{ \ep^2}\int_D  U_{\ep,\Lambda}\mathrm d x\nonumber\\
		\leq &\frac{C\Lambda}{ \ep^2} |A_\ep|^{\frac{1}{2}} \left(\int_D  U_{\ep,\Lambda}^2 \mathrm d x\right)^{\frac{1}{2}}\nonumber\\
		\leq & \frac{C\Lambda}{ \ep^2} |A_\ep|^{\frac{1}{2}}  \int_{A_\ep}  |\nabla U_{\ep,\Lambda}| \mathrm d x\nonumber\\
		\leq &  \frac{C\Lambda}{ \ep^2} |A_\ep|  \left(\int_{A_\ep}     |\nabla U_{\ep,\Lambda}|^2 \mathrm d x\right)^{\frac{1}{2}}\nonumber\\
		\leq & C\kappa\left( \int_D\nabla U_{\ep,\Lambda} \cdot ( K(x)\nabla U_{\ep,\Lambda})\right)^{\frac{1}{2}} \mathrm d x,
	\end{align}	
	where we have used \eqref{4-21}, the Sobolev embedding $ W^{1,1}_0(D)\subset L^2(D) $ and  the positive-definiteness of $ K$. Note that the constant $C$ in \eqref{4-24} is independent of $\ep$ and $\Lambda$. 
	
	Combining \eqref{4-23} and \eqref{4-24}, we conclude that $\frac{\Lambda}{\ep^2}\int_{D}  U_{\ep,\Lambda}(x) \mathrm{d} x$ is uniformly bounded with respect to $\varepsilon$ and $\Lambda$. 	The desired result clearly follows from \eqref{4-20} and Lemma \ref{lm4-4}.
\end{proof}

\subsection{Location and diameter of $ \text{supp}(\zeta_{\ep,\Lambda}) $}
Let $\Lambda>\max\{\alpha a+b,1+\pi R^2/\kappa\}$ be a fixed number. Now, we  estimate the diameter and location of the support set of $\zeta_{\ep,\Lambda}$. 
To this end, we denote $ r_{\ep, \Lambda} $ and $ R_{\ep, \Lambda} $ as 
\begin{equation}\label{4-25}
	r_{\ep, \Lambda}=\inf\{|x|\mid x\in \text{supp}(\zeta_{\ep,\Lambda})\},\ \ \text{and}\ \ R_{\ep, \Lambda}=\sup\{|x|\mid x\in \text{supp}(\zeta_{\ep,\Lambda})\}.
\end{equation}
So $ r_{\ep, \Lambda} $ and  $ R_{\ep, \Lambda} $ describe the lower bound and upper bound of the distance between the origin and $ \text{supp}(\zeta_{\ep,\Lambda}) $, respectively. 
We will prove that, the support set of the maximizers $ \zeta_{\ep,\Lambda} $ constructed in Lemma \ref{lm4-2}
must concentrate near some points in $ \partial B_{r_*}(0) $ as $\varepsilon$ tends to 0. 
The steps of proof are as follows: estimates of $ \lim\limits_{\varepsilon\to 0^+}r_{\ep, \Lambda} $, estimates of $ \lim\limits_{\varepsilon\to 0^+}R_{\ep, \Lambda} $ and estimates of diameter and limiting location of $  \text{supp}(\zeta_{\ep,\Lambda}) $.

Using the expansion of Green's function $ G_K(x,y) $ in Lemma \ref{lm4-1} and the theory of rearrangement function, we have
\begin{lemma}\label{lm4-6}
It holds that	$\lim\limits_{\varepsilon\to 0^+}r_{\ep, \Lambda}=r_*$.
\end{lemma}

\begin{proof}
	Let  $\gamma\in(0,1)$ be a fixed number. By Lemma \ref{lm4-2}, for any  $x\in {supp}(\zeta_{\ep,\Lambda})$, we have $$\mathcal{G}_{K}\zeta_{\ep,\Lambda}(x)-\frac{\alpha|x|^2}{2}\ln{\frac{1}{\varepsilon}}\ge \mu_{\ep,\Lambda}.$$
	Note that
	\begin{equation*}
		\begin{split}
			\mathcal{G}_{K}\zeta_{\ep,\Lambda}(x)&= \int_{D}G_{K}(x,y)\zeta_{\ep,\Lambda}(y)\mathrm d y\\
			&=\left (\int_{D\cap\{|T_{x}(y-x)|>\varepsilon^\gamma\}}+\int_{D\cap\{|T_{x}(y-x)|\le\varepsilon^\gamma\}}\right )G_{K}(x,y)\zeta_{\ep,\Lambda}(y)\mathrm d y.
		\end{split}
	\end{equation*}
	
	On the one hand, it follows from Lemma \ref{lm4-1} that $  \max_{x,y\in D} H_0(x,y)\le C $, which implies that $ \int_D H_0(x,y)\zeta_{\ep,\Lambda}(y)\mathrm d y\leq C $.  Direct computation yields that
	\begin{equation*}
		\sqrt{\det K(y)}^{-1}\leq \sqrt{\det K((R,0))}^{-1} \ \ \ \ \ \ \ \ \forall\ y\in D,
	\end{equation*}
	\begin{equation*}
		\bigg|\frac{T_{x }+T_y}{2}(y-x )\bigg|\geq C\varepsilon^\gamma\ \ \  \ \ \ \ \ \forall\ |T_{x }(y-x )|>\varepsilon^\gamma.
	\end{equation*}
	Thus by the decomposition  of $ G_{K} $ in  Lemma \ref{lm4-1} , we have
	\begin{equation}\label{4-26}
		\begin{split}
			&\int_{D\cap\{|T_{x}(y-x)|>\varepsilon^\gamma\}}G_{K}(x,y)\zeta_{\ep,\Lambda}(y)\mathrm d y\\
			&\ \ \leq\int\limits_{D\cap\{|T_{x}(y-x)|>\varepsilon^\gamma\}}\frac{\sqrt{\det K(x)}^{-1}+\sqrt{\det K(y)}^{-1}}{2}\Gamma\left (\frac{T_{x}+T_y}{2}\left (x-y\right )\right )\zeta_{\ep,\Lambda}(y)\mathrm d y+C\\
			&\ \ \le \frac{\sqrt{\det K((R,0))}^{-1}}{2\pi} \int\limits_{D\cap\{|T_{x}(y-x)|>\varepsilon^\gamma\}}\ln\frac{1}{\varepsilon^\gamma}\zeta_{\ep,\Lambda}(y)\mathrm d y+C\\
			&\ \ \le \frac{\sqrt{\det K((R,0))}^{-1}\kappa\gamma}{2\pi}\ln\frac{1}{\varepsilon}+C,
		\end{split}
	\end{equation}
where $C$ is a constant independent of $\ep$ and $\Lambda$.

	On the other hand, by the smoothness and positive-definiteness of $ K $, it is not hard to get that
	\begin{equation*}
		\sqrt{\det K(y)}^{-1}\leq \sqrt{\det K(x)}^{-1}+C\varepsilon^\gamma,\ \ \ \ \forall\ y\in D\cap\{|T_{x}(y-x)|\leq\varepsilon^\gamma\},
	\end{equation*}
\begin{equation*}
	\Gamma\left (\frac{T_x+T_y}{2}(x-y)\right)=\Gamma\left (T_x(x-y)\right)+O(\varepsilon^\gamma),\ \ \ \ \forall\ y\in D\cap\{|T_{x}(y-x)|\leq\varepsilon^\gamma\}, y\neq x		.
\end{equation*}
	Hence we get
	\begin{equation*}
		\begin{split}
			&\int_{D\cap\{|T_{x}(y-x)|\leq\varepsilon^\gamma\}}G_{K}(x,y)\zeta_{\ep,\Lambda}(y)\mathrm d y\\
			&\ \ \leq \left (\sqrt{\det K(x)}^{-1}+C\varepsilon^\gamma\right )\int\limits_{D\cap\{|T_{x}(y-x)|\leq\varepsilon^\gamma\}}\Gamma\left (\frac{T_{x}+T_y}{2}\left (x-y\right )\right )\zeta_{\ep,\Lambda}(y)\mathrm d y+C\\
			&\ \ \leq \left (\sqrt{\det K(x)}^{-1}+C\varepsilon^\gamma\right )\int\limits_{D\cap\{|T_{x}(y-x)|\leq\varepsilon^\gamma\}}\Gamma\left (T_{x}\left (x-y\right )\right )\zeta_{\ep,\Lambda}(y)\mathrm d y+C\\
			&=\ \ \left (\sqrt{\det K(x)}^{-1}+C\varepsilon^\gamma\right )\int_{|z|\leq\varepsilon^\gamma}\Gamma(z)\zeta_{\ep,\Lambda}(T_{x}^{-1}z+x)\sqrt{\det K(x)}dz+C.
		\end{split}
	\end{equation*}
	
	Denote $ \zeta_{\ep,\Lambda}^*(z)=\frac{\Lambda}{\varepsilon^2}\textbf{1}_{B_{t_{\ep,\Lambda}}(0)} $ the non-negative decreasing rearrangement function of $ \zeta_{\ep,\Lambda}(T_{x}^{-1}z+x)\textbf{1}_{\{|z|\leq \varepsilon^\gamma\}} $. Here we choose $ t_{\ep,\Lambda}>0 $ such that $$ \int\limits_{|z|\leq\varepsilon^\gamma} \zeta_{\ep,\Lambda}(T_{x}^{-1}z+x) dz=\int\limits_{\mathbb{R}^2} \zeta_{\ep,\Lambda}^* dz =\frac{\pi\Lambda t_{\ep,\Lambda}^2}{\ep^2}.$$ By the rearrangement inequality and the fact $\int\limits_{|z|\leq\varepsilon^\gamma} \zeta_{\ep,\Lambda}(T_{x}^{-1}z+x) dz\leq \sqrt{\det K(x)}^{-1}\kappa$, we have
	\begin{equation*}
		\begin{split}
			\int_{|z|\leq\varepsilon^\gamma}\Gamma(z)\zeta_{\ep,\Lambda}(T_{x}^{-1}z+x)dz\le&\int_{\mathbb{R}^2} \Gamma(z)\zeta_{\ep,\Lambda}^*(z) dz\\
			\leq&\frac{1}{2\pi}\ln\frac{1}{\varepsilon}\int_{|z|\leq\varepsilon^\gamma}\zeta_{\ep,\Lambda}(T_{x}^{-1}z+x) dz+C(1+\ln\Lambda),
		\end{split}
	\end{equation*}
	from which we deduce that
	\begin{align}\label{4-27}
			&\int_{D\cap\{|T_{x}(y-x)|\leq\varepsilon^\gamma\}}G_{K}(x,y)\zeta_{\ep,\Lambda}(y)\mathrm d y\nonumber\\
			&\ \ \le\left (\sqrt{\det K(x)}^{-1}+C\varepsilon^\gamma\right )\frac{1}{2\pi}\ln\frac{1}{\varepsilon}\int\limits_{|z|\leq\varepsilon^\gamma}\zeta_{\ep,\Lambda}(T_{x}^{-1}z+x)\sqrt{\det K(x)}dz+C(1+\ln\Lambda)\nonumber\\
			&\ \ \le	 \left (\sqrt{\det K(x)}^{-1}+C\varepsilon^\gamma\right )\frac{1}{2\pi}\ln\frac{1}{\varepsilon}\int\limits_{D\cap\{|T_{x}(y-x)|\leq\varepsilon^\gamma\}}\zeta_{\ep,\Lambda}\mathrm d y+C(1+\ln\Lambda)\nonumber\\
			&\ \ \le	  \frac{\sqrt{\det K(x)}^{-1}}{2\pi}\ln\frac{1}{\varepsilon}\int\limits_{D\cap\{|T_{x}(y-x)|\leq\varepsilon^\gamma\}}\zeta_{\ep,\Lambda}\mathrm d y+C(1+\ln\Lambda),
	\end{align}
where $C$ is a constant independent of $\ep$, $\gamma$ and $\Lambda$.
	Therefore by Lemma \ref{lm4-5}, \eqref{4-26} and \eqref{4-27}, we conclude that for any $x\in supp(\zeta_{\ep,\Lambda})$ and $ \gamma\in(0,1) $, there holds
	\begin{equation}\label{4-28}
		\begin{split}
			&\frac{\sqrt{\det K(x)}^{-1}}{2\pi}\ln\frac{1}{\varepsilon} \int\limits_{D\cap\{|T_{x}(y-x)|\leq\varepsilon^\gamma\}}\zeta_{\ep,\Lambda}\mathrm d y+\frac{\sqrt{\det K((R,0))}^{-1}\kappa\gamma}{2\pi}\ln\frac{1}{\varepsilon}-\frac{\alpha|x|^2}{2}\ln{\frac{1}{\varepsilon}} \\
			&\ \ \ge  Y((r_*,0))\ln\frac{1}{\varepsilon}+ \frac{\alpha}{2\kappa}\ln\frac{1}{\varepsilon}\int_D |x|^2 \zeta_{\ep,\Lambda} \mathrm d x-C(1+\ln\Lambda).
		\end{split}
	\end{equation}
	Dividing both sides of the above inequality by $\ln\frac{1}{\varepsilon}$, we obtain
	\begin{align}\label{4-29}
			Y(x)\ge& \frac{\sqrt{\det K(x)}^{-1}}{2\pi}\int\limits_{D\cap\{|T_{x}(y-x)|\leq\varepsilon^\gamma\}}\zeta_{\ep,\Lambda}\mathrm d y-\alpha|x|^2\\
			\ge &  Y((r_*,0))+\frac{\alpha}{2\kappa}\left \{\int_D |x|^2 \zeta_{\ep,\Lambda} \mathrm d x-\kappa|x|^2\right \}-\frac{\sqrt{\det K((R,0))}^{-1}\kappa\gamma}{2\pi}-\frac{C(1+\ln\Lambda)}{\ln\frac{1}{\varepsilon}}.\nonumber
	\end{align}

		Note that the point  $x\in {supp}(\zeta_{\ep,\Lambda})$ is arbitrary. So we may take $ x =x_\varepsilon$ be a point such that $$	|x_\varepsilon|\leq r_{\ep, \Lambda}+\ep.$$
		The definition of $r_{\ep, \Lambda}$ yields
		\begin{equation*}
			\int_D |x_\ep|^2 \zeta_{\ep,\Lambda} \mathrm d x \ge \kappa (r_{\ep, \Lambda}) ^2.
		\end{equation*}
		Then letting $\varepsilon$ tend to $0^+$, we deduce from $\eqref{4-29}$ that
	\begin{equation}\label{4-30}
		\liminf_{\varepsilon\to 0^+}Y(x_\varepsilon)\ge Y((r_*,0))-\sqrt{\det K((R,0))}^{-1}d\gamma/(2\pi).
	\end{equation}
	Hence we get the desired result by letting $\gamma \to 0$ and using Lemma \ref{lm4-3}.
\end{proof}

Next, we estimate the moment of inertia $\mathcal{I}(\zeta_{\ep,\Lambda})=\frac{1}{2} \int_D|x|^2\zeta_{\ep,\Lambda} \mathrm d x $. We have
\begin{lemma} \label{lm4-7} 
There holds
	\begin{equation}\label{4-31}
		\lim_{\varepsilon\to 0^+}\mathcal{I}(\zeta_{\ep,\Lambda})=\frac{1}{2}\lim_{\varepsilon\to 0^+}\int_D|x|^2\zeta_{\ep,\Lambda} \mathrm d x= \frac{\kappa}{2} r_*^2.
	\end{equation}
	As a consequence, for any $\eta>0$, it holds
	\begin{equation}\label{4-32}
		\lim_{\varepsilon\to 0^+}\int_{ \{x\in D\mid |x|\ge r_*+\eta\}}\zeta_{\ep,\Lambda} \mathrm d x=0.
	\end{equation}
\end{lemma}

\begin{proof}
	Let $ x=x_\varepsilon\in D$   in $\eqref{4-29}$, we know that for any $\gamma\in(0,1)$ 
	\begin{equation*}
		0\le  \liminf_{\varepsilon\to 0^+}\left [\int_D |x|^2 \zeta_{\ep,\Lambda} \mathrm d x - \kappa  (x_\varepsilon) ^2\right ]\le  \limsup_{\varepsilon\to 0^+}\left [\int_D |x|^2 \zeta_{\ep,\Lambda} \mathrm d x - \kappa (x_\varepsilon) ^2\right ]\le \frac{\sqrt{\det K((R,0))}^{-1}\kappa^2\gamma}{\alpha\pi}.
	\end{equation*}
	Combining this with Lemma \ref{lm4-6}, we find
	\begin{equation*}
		\lim_{\varepsilon\to 0^+}\int_D |x|^2 \zeta_{\ep,\Lambda} \mathrm d x   =\lim_{\varepsilon\to 0^+}\kappa (x_\varepsilon) ^2  =\kappa r_*^2.
	\end{equation*}
	
	Then \eqref{4-32} follows immediately from $\eqref{4-31}$ and Lemma \ref{lm4-6}. The proof is hence complete.
\end{proof}

Lemma \ref{lm4-7} implies that  the   moment of inertia of $ \zeta_{\varepsilon, \Lambda} $ will tend to $ \kappa r_*^2/2  $ and for any $ \eta>0 $ the mass of $ \zeta_{\varepsilon, \Lambda} $  on the set $  \{x\in D \mid |x|\ge r_*+\eta\} $   tends to 0 as $ \varepsilon\to 0^+ $. Based on this, we can then get limits of $ R_{\ep, \Lambda}. $
\begin{lemma}\label{lm4-8}
There holds 
	$\lim\limits_{\varepsilon\to 0^+}R_{\ep, \Lambda}=r_*$.
\end{lemma}

\begin{proof}
	Clearly, $ \liminf\limits_{\varepsilon\to 0^+}R_{\ep, \Lambda}\geq\lim\limits_{\varepsilon\to 0^+}r_{\ep, \Lambda} =r_*  $. Note that  for any $x\in D$  $$ 0<\sqrt{\det K(x)}^{-1}\leq \sqrt{\det K((R,0))}^{-1}<+\infty. $$ 
Let $ \hat{X}_\varepsilon $ be a point in $ supp(\zeta_{\ep,\Lambda}) $ such that $$|\hat{X}_\varepsilon|\geq R_{\ep, \Lambda}-\ep.$$	Taking $x =\hat{X}_\varepsilon $ in \eqref{4-29},  we obtain
	\begin{equation*}
		\begin{split}
			&\frac{\sqrt{\det K((R,0))}^{-1}}{2\pi}\liminf_{\varepsilon\to 0^+}\int\limits_{D\cap\{|T_{\hat{X}_\ep}(y-\hat{X}_\ep)|\leq\varepsilon^\gamma\}}\zeta_{\ep,\Lambda}\mathrm d y\\
			&\ \ \ge  Y((r_*,0))+ \frac{\alpha}{2}\liminf_{\varepsilon\to 0^+}(R_{\ep, \Lambda})^2+\frac{ \alpha r_*^2}{2}-\frac{\sqrt{\det K((R,0))}^{-1}\kappa\gamma}{2\pi}  \\
			&\ \ \ge	 \frac{\kappa \sqrt{\det K((r_*,0))}^{-1}}{2\pi}-\frac{\sqrt{\det K((R,0))}^{-1}\kappa\gamma}{2\pi}.
		\end{split}
	\end{equation*}
	Hence, we get
	\begin{equation*}
		\liminf_{\varepsilon\to 0^+}\int\limits_{D\cap\{|T_{\hat{X}_\ep}(y-\hat{X}_\ep)|\leq\varepsilon^\gamma\}}\zeta_{\ep,\Lambda}\mathrm d y \ge \frac{\kappa \sqrt{\det K((r_*,0))}^{-1}}{\sqrt{\det K((R,0))}^{-1}}-\kappa\gamma.
	\end{equation*}
	So by choose $ \gamma $ small such that $ 0<\gamma<\frac{ \sqrt{\det K((r_*,0))}^{-1}}{\sqrt{\det K((R,0))}^{-1}} $, we arrive at	 $$ \liminf\limits_{\varepsilon\to 0^+}\int\limits_{D\cap\{|T_{\hat{X}_\ep}(y-\hat{X}_\ep)|\leq\varepsilon^\gamma\}}\zeta_{\ep,\Lambda}\mathrm d y \ge C_0>0, $$ 
	which combined with \eqref{4-32} immediately yields the conclusion of this lemma.

\end{proof}

From Lemmas \ref{lm4-6} and  \ref{lm4-8}, we know that $ supp(\zeta_{\ep,\Lambda}) $ locates in a narrow annular domain whose centerline is $ \partial B_{r_*}(0) $. That is, for any $ x_\ep\in supp(\zeta_{\ep,\Lambda}) $
\begin{equation*}
	\lim_{\varepsilon\to 0^+}|x_\ep|=r_*.
\end{equation*}
Based on this, we can further refine the estimates of $ supp(\zeta_{\ep,\Lambda}) $. We can prove that, for any $ \ep>0 $ small the diameter of the support set of $ \zeta_{\ep,\Lambda} $  is the order of $ \varepsilon^{\gamma} $ for any $ \gamma\in(0, 1) $.
\begin{lemma}\label{lm4-9}
	Let $\gamma\in(0, 1)$ be a fixed number. Then 
	there exists an $\ep_\Lambda\in(0,1)$ depending on $\Lambda$ such that 
	$$ diam\big({supp}(\zeta_{\ep,\Lambda})\big) \le 2 \varepsilon^{\gamma},\  \  \ \forall\, \ep\in(0, \ep_\Lambda). $$
\end{lemma}

\begin{proof}
	Since $\int_D\zeta_{\ep,\Lambda} dx=\kappa, $ it suffices to prove that
	\begin{equation}\label{4-33}
		\int\limits_{B_{\varepsilon^\gamma}(x )}\zeta_{\ep,\Lambda} dx>\kappa/2, \ \ \forall\, x  \in {supp}(\zeta_{\ep,\Lambda}).
	\end{equation}
	Indeed, for any $ x \in {supp}(\zeta_{\ep,\Lambda}) $,   we infer from $\eqref{4-29}$  that
	\begin{equation}\label{4-34}
		\begin{split}
			\int\limits_{D\cap\{|T_{x}(y-x)|\leq\varepsilon^\gamma\}}\zeta_{\ep,\Lambda} \mathrm d y
			\ge  &2\pi\sqrt{\det K(x)}\bigg[Y((r_*,0))+\frac{\alpha}{2\kappa}\left \{\int_D |x|^2 \zeta_{\ep,\Lambda} dx+\kappa|x|^2\right \}\\
			&-\frac{\sqrt{\det K((R,0))}^{-1}\kappa\gamma}{2\pi}-\frac{C(1+\ln\Lambda)}{\ln\frac{1}{\varepsilon}}\bigg].
		\end{split}
	\end{equation}	
	From Lemmas  \ref{lm4-6} and  \ref{lm4-8}, we know that $|x|\to r_*$  as $\varepsilon \to 0^+$. Taking this into \eqref{4-34}, we obtain
	\begin{equation}\label{216}
		\liminf_{\varepsilon\to 0^+}\int\limits_{\{|T_{x}(y-x)|\leq\varepsilon^\gamma\}}\zeta_{\ep,\Lambda} \mathrm d y \ge \kappa-\frac{\sqrt{\det K((r_*,0))} \kappa\gamma}{\sqrt{\det K((R,0))}}.
	\end{equation}
	From the definition of $ T_{x} $, one computes directly that
	\begin{equation*}
		\{y  \mid T_{x}(y-x)|\leq\varepsilon^\gamma\}\subset B_{\varepsilon^\gamma}(x).
	\end{equation*}
	Thus we get
	\begin{equation*}
		\liminf_{\varepsilon\to 0^+}\int\limits_{B_{\varepsilon^\gamma}(x)}\zeta_{\ep,\Lambda} \mathrm d y \ge \kappa-\frac{\sqrt{\det K((r_*,0))} \kappa\gamma}{\sqrt{\det K((R,0))}},
	\end{equation*}
	which implies \eqref{4-33} by taking $\gamma=\gamma_0:=\min\left\{1,  \frac{\sqrt{\det K((R,0))} }{4\sqrt{\det K((r_*,0))}  }\right\}>0 $. Thus we conclude that $ diam\big({supp}(\zeta_{\ep,\Lambda})\big) \le 2 \varepsilon^{\gamma_0} $, from which we deduce that
	\begin{equation}\label{4-36}
		diam\big({supp}(\zeta_{\ep,\Lambda})\big)\le \frac{C}{\ln\frac{1}{\varepsilon}}
	\end{equation}
	provided $\varepsilon$ is small enough.
	
	Now we go back to $\eqref{4-26}$ and improve estimates of $ \int\limits_{D\cap\{|T_{x_\varepsilon}(y-x)|>\varepsilon^\gamma\}}G_{K}(x,y)\zeta_{\ep,\Lambda}(y)\mathrm d y $. From \eqref{4-36}, we get for any $ \gamma\in(0,1) $
	\begin{equation*}
		\begin{split}
			&\int\limits_{D\cap\{|T_{x}(y-x)|>\varepsilon^\gamma\}}G_{K}(x,y)\zeta_{\ep,\Lambda}(y)\mathrm d y\\
			&\ \ \leq\int\limits_{D\cap\{|T_{x}(y-x)|>\varepsilon^\gamma\}}\frac{\sqrt{\det K(x)}^{-1}+\sqrt{\det K(y)}^{-1}}{2}\Gamma\left (\frac{T_{x}+T_y}{2}\left (x-y\right )\right )\zeta_{\ep,\Lambda}(y)\mathrm d y+C\\
			&\ \ \le	 \frac{\sqrt{\det K(x)}^{-1}\gamma}{2\pi}\ln\frac{1}{\varepsilon} \int\limits_{D\cap\{|T_{x}(y-x)|>\varepsilon^\gamma\}}\zeta_{\ep,\Lambda} \mathrm d y+C.
		\end{split}
	\end{equation*}
	Repeating the proof of \eqref{4-28}, we obtain a refined estimate of $\eqref{4-28}$ as follows
	\begin{equation*}
		\begin{split}
			& \frac{\sqrt{\det K(x)}^{-1}}{2\pi}\int\limits_{D\cap\{|T_{x}(y-x)|\leq\varepsilon^\gamma\}}\zeta_{\ep,\Lambda} \mathrm d y+\frac{\sqrt{\det K(x)}^{-1}\gamma}{2\pi} \int\limits_{D\cap\{|T_{x}(y-x)|>\varepsilon^\gamma\}}\zeta_{\ep,\Lambda} \mathrm d y\\
			&\ \ \ge   Y((r_*,0))+\frac{\alpha}{2\kappa}\left \{\int_D |x|^2 \zeta_{\ep,\Lambda} dx+\kappa|x|^2\right \}-\frac{C(1+\ln\Lambda)}{\ln\frac{1}{\varepsilon}},
		\end{split}
	\end{equation*}
	which implies that
	\begin{equation}\label{4-37}
		\begin{split}
			&(1-\gamma)\int\limits_{D\cap\{|T_{x}(y-x)|\leq\varepsilon^\gamma\}}\zeta_{\ep,\Lambda} \mathrm d y\\
			&\ \ \geq 2\pi\sqrt{\det K(x)}\left [ Y((r_*,0))+\frac{\alpha}{2\kappa}\left \{\int_D |x|^2 \zeta_{\ep,\Lambda} dx+\kappa|x|^2\right \}-\frac{C(1+\ln\Lambda)}{ \ln\frac{1}{\varepsilon}}\right ]-\kappa\gamma.
		\end{split}
	\end{equation}
	Taking the limit inferior to \eqref{4-37}, we get
	\begin{equation*}
		\liminf_{\varepsilon\to 0^+}\int\limits_{D\cap\{|T_{x}(y-x)|\leq\varepsilon^\gamma\}}\zeta_{\ep,\Lambda} \mathrm d y \ge \frac{1}{1-\gamma}\left [2\pi\sqrt{\det K((r_*,0))}(Y((r_*,0))+\alpha r_*^2)-\kappa\gamma\right ]=\kappa,
	\end{equation*}
	which implies \eqref{4-33} for all $ \gamma\in (0,1) $. The proof is thus complete.
\end{proof}

\subsection{Vanishment of the vortex patch term in profile \eqref{4-13}}
Having made the above preparations, we are able to show that the ``bad'' term of $ \zeta_{\ep, \Lambda} $, i.e., the last patch term in \eqref{4-13}, indeed vanishes when choosing the parameter  $ \Lambda $ properly large. This, together with \eqref{4-13}, will guarantee that the maximizers $ \zeta_{\ep,\Lambda} $ in Lemma \ref{lm4-2} is a solution of \eqref{3-10} with proper profile functions $ F_1 $ and $ F_2 $. The key observation is the following estimates of upper bounds of  $\psi_{\ep,\Lambda}$. 
\begin{lemma}\label{lm4-10}
For any $\Lambda>\max\{\alpha a+b,1+\pi R^2/\kappa\}$, let $\ep_\Lambda$ be the small positive constant in Lemma \ref{lm4-9}. Then there exists  an $\tilde \ep_\Lambda\in (0, \ep_\Lambda)$ such that  for any $ 0<\ep< \tilde\ep_\Lambda $,  
	\begin{equation*}
		||(\psi_{\ep,\Lambda})_+||_{L^\infty(D)} \le C (\ln\Lambda +1),
	\end{equation*}
	where $C$   is a positive constant  independent of $\ep$ and $\Lambda$.
\end{lemma}

\begin{proof}
If $ x\notin supp\,(\zeta_{\ep,\Lambda})$, then $ (\psi_{\ep,\Lambda})_+=0. $ For any $x\in supp\,(\zeta_{\ep,\Lambda})$, by Lemma \ref{lm4-9}  we find 
	$$ supp\,(\zeta_{\ep,\Lambda})\subset  B_{2\ep^\gamma}(x).$$
	Using similar calculations as the proof of \eqref{4-27},   we  are able to show
	\begin{equation*}
	\begin{split}
		&\ \ \ \int_{D}G_{K}(x,y)\zeta_{\ep,\Lambda}(y)\mathrm d y\\
		&\leq \left (\sqrt{\det K(x)}^{-1}+C\varepsilon^\gamma\right )\int\limits_{ B_{2\ep^\gamma}(x)}\Gamma\left (\frac{T_{x}+T_y}{2}\left (x-y\right )\right )\zeta_{\ep,\Lambda}(y)\mathrm d y+C\\
		& \leq \left (\sqrt{\det K(x)}^{-1}+C\varepsilon^\gamma\right )\int\limits_{B_{2\ep^\gamma}(x)}\Gamma\left (T_{x}\left (x-y\right )\right )\zeta_{\ep,\Lambda}(y)\mathrm d y+C\\
		&= \left (\sqrt{\det K(x)}^{-1}+C\varepsilon^\gamma\right )\int_{|T_{x}^{-1}z|\leq2\varepsilon^\gamma}\Gamma(z)\zeta_{\ep,\Lambda}(T_{x}^{-1}z+x)\sqrt{\det K(x)}dz+C\\
		&\leq \frac{\sqrt{\det K(x)}^{-1} \kappa}{2\pi}\ln\frac{1}{\ep}+C(1+\ln\Lambda),
	\end{split}
\end{equation*}
where $C$ is a constant independent of $\ep$ and $\Lambda$. Thus using the above estimate and  Lemmas \ref{lm4-3} and \ref{lm4-5}, we obtain
	\begin{equation*}
	\begin{split}
		&\ \ \  \psi_{\ep,\Lambda}(x) =\mathcal{G}_K\zeta_{\ep,\Lambda}(x)-\frac{\alpha |x|^2}{2}\ln{\frac{1}{\ep}}-\mu_{\ep,\Lambda} \\
		& \le \frac{\sqrt{\det K(x)}^{-1} \kappa}{2\pi}\ln\frac{1}{\ep}-\frac{\alpha |x|^2}{2}\log{\frac{1}{\ep}}-Y((r_*,0))\ln\frac{1}{\varepsilon}- \frac{\alpha}{2\kappa}\ln\frac{1}{\varepsilon}\int_D |y|^2 \zeta_{\ep,\Lambda} \mathrm d y+C(1+\ln\Lambda)\\
		& \le \left(Y(x)-Y((r_*,0))\right)\ln\frac{1}{\varepsilon}+ C(1+\ln\Lambda)\\
		&\le  C(1+\ln\Lambda),
	\end{split}
\end{equation*}
	where the positive number $C$ does not depend on $\ep$ and $\Lambda$. The proof is thus completed.
\end{proof}

Based on Lemma \ref{lm4-10}, we can finally get that the last term in \eqref{4-13}  indeed vanishes when choosing  $\Lambda$ large. 
\begin{lemma}\label{lm4-11}
	Let $\tilde \ep_\Lambda$ be the small positive constant in Lemma \ref{lm4-10}. Let $ r=|x| $. Then there exists a constant $\Lambda_0$ sufficiently large independent of $ \varepsilon $ such that  for any $ \Lambda\geq \Lambda_0 $ and $ 0<\ep<\tilde \ep_\Lambda, $
	\begin{equation*}
		\left|\left\{\zeta_{\ep,\Lambda}=\frac{\Lambda}{\ep^2}\right\}\right|=0.
	\end{equation*}
As a result, if we choose $\Lambda\geq \Lambda_0$  and $ 0<\ep<\tilde \ep_\Lambda$, then one has
	\begin{equation}\label{4-38} \begin{split}
		\zeta_{\ep, \Lambda}=\frac{1}{\ep^2}\frac{2h^2a\ep +a^2(h^2+r^2)}{(h^2+r^2)^2 }(\psi_{\ep, \Lambda})_+ +\frac{1}{\ep^2}\left(\frac{h^2\bar \alpha a\ep}{h^2+r^2}+b\right)\textbf{1}_{\{ \psi_{\ep, \Lambda}>0\} }	 \ \ a.e.\  \text{in}\  D.
	\end{split}
\end{equation}
\end{lemma}

\begin{proof}
	In view of \eqref{4-15}, we find 
	\begin{equation*}
		(\psi_{\ep,\Lambda})_+\ge    \frac{(h^2+r^2)^2\left(  \Lambda -\frac{h^2\bar \alpha a\ep}{h^2+r^2}-b\right)_+}{2h^2a\ep+ a^2(h^2+r^2)} \ \ \ \ \text{on}\ \  \left \{\zeta_{\ep, \Lambda}=\frac{\Lambda}{\ep^2}\right \},
	\end{equation*}
	from which we conclude that
	\begin{equation}\label{4-39}
		(\psi_{_{\beta,\Lambda}})_+\ge C_1\Lambda-C_2 \ \ \ \ \text{on}\ \  \left \{\zeta_{\ep, \Lambda}=\frac{\Lambda}{\ep^2}\right \},
	\end{equation}
	for some positive numbers $C_1$, $C_2$ independent of $\ep$ and $\Lambda$.
	Combining \eqref{4-39} with Lemma \ref{lm4-10}, we deduce that on the set $\left \{\zeta_{\ep, \Lambda}=\frac{\Lambda}{\ep^2}\right \}$ 
	\begin{equation*}
		C\ln\Lambda\ge C_1\Lambda-C_3,
	\end{equation*}
	which is impossible when $\Lambda$ is chosen to be sufficiently large, i.e., $\Lambda \ge \Lambda_0$ for a constant $\Lambda_0$ independent of $\ep$. Hence, we have  
	\begin{equation}\label{4-39'}
		\left|\left\{\zeta_{\ep,\Lambda}=\frac{\Lambda}{\ep^2}\right\}\right|=0, \ \  \  \forall\ \ 0<\ep<\tilde \ep_\Lambda,\ \ \Lambda\geq \Lambda_0.
	\end{equation} 
\eqref{4-38} follows immediately from \eqref{4-39'} and  Lemma \ref{lm4-2}.
 The proof is therefore completed.
\end{proof}

For the rest of this section, we will choose $ \Lambda = \Lambda_0 $ in Lemma \ref{lm4-11} and abbreviate $(\zeta_{\ep,\Lambda_0},\psi_{\ep,\Lambda_0}, \mu_{\ep,\Lambda_0})$ to $(\zeta_\ep, \psi_\ep, \mu_\ep)$. So maximizers $ \zeta_\ep $ in Lemma \ref{lm4-2} have the profile \eqref{4-5} and thus satisfy \eqref{3-10} with the profile functions 
\begin{equation*} 
F_1(s)=\ep^{-1}a (s-\mu_\ep)_+,\ \ \   F_2(s)=\ep^{-2}(b+\bar\alpha a \ep) \textbf{1}_{\{s-\mu_\ep>0\}},\ \ \  s\in \mathbb R.
\end{equation*}
According to the previous derivation in sections 2 and 3, $ \zeta_\ep $ corresponds to a family of  concentrated helical symmetric vorticity fields of \eqref{1-2} without the assumption \eqref{1-7}.
\subsection{Refined asymptotic estimates of $ \zeta_\varepsilon $}
We can further improve results of Lemma \ref{lm4-9} that the diameter of the support set of $ \zeta_\varepsilon $ is the order  of $ \varepsilon $ and analyze the properties of limiting function  of  $ \zeta_\varepsilon $. As a result, we can deduce that the support set $ \text{supp}(\zeta_\varepsilon) $ tends asymptotically to an ellipse.
\begin{lemma}\label{lm4-12}
	There exist  $r_1, r_2>0$ independent of $\varepsilon$ such that for $ \varepsilon $ sufficiently  small,
	\begin{equation*}
		r_1 \varepsilon\leq diam(\text{supp}(\zeta_\varepsilon)) \le r_2 \varepsilon.
	\end{equation*}
\end{lemma}
\begin{proof}
	Since $0\leq \zeta_\varepsilon\leq \frac{\Lambda_0}{\ep^2}$, we get 
	$$\kappa=\int_D\zeta_\varepsilon \mathrm d x\leq \frac{\pi\Lambda_0}{\ep^2}  (diam(\text{supp}(\zeta_\varepsilon)))^2,$$
	which shows   $ diam(  {supp}(\zeta_\varepsilon))\geq r_1 \varepsilon$ for some $ r_1>0 $. It suffices to prove $ diam({supp}(\zeta_\varepsilon)) \le r_2 \varepsilon $ for some $r_2>0$.
	
We observe that for any $  x  \in  {supp}(\zeta_\varepsilon)  $,
	\begin{equation}\label{4-40}
		\mathcal{G}_{K}\zeta_\varepsilon(x)-\frac{\alpha|x|^2}{2}\ln{\frac{1}{\varepsilon}}\ge \mu_\ep.
	\end{equation}
	According to Lemma $\ref{lm4-9}$, there holds
	\begin{equation}\label{4-41}
		supp(\zeta_\varepsilon)\subseteq B_{2\varepsilon^{\frac{1}{2}}} ( x ).
	\end{equation}
	Let $R_*>1$ to be determined later. On the one hand, we have
	\begin{equation*}
		\begin{split}
			\mathcal{G}_{K}\zeta_\varepsilon(x)
			&=\left (\int_{D\cap\{|T_{x}(y-x)|> R_*\varepsilon\}}+\int_{D\cap\{|T_{x}(y-x)|\leq R_*\varepsilon\}}\right )G_{K}(x,y)\zeta_\varepsilon(y)\mathrm d y\\
			&=:B_1+B_2.
		\end{split}
	\end{equation*}
	By \eqref{4-41}, we get
	\begin{equation}\label{4-42}
		\begin{split}
			B_1 &=\int_{D\cap\{|T_{x}(y-x)|> R_*\varepsilon\}}G_{K}(x,y)\zeta_\varepsilon(y)\mathrm d y\\
			&\le \frac{\sqrt{\det K(x)}^{-1}+O\left (\varepsilon^{\frac{1}{2}}\right )}{2\pi}\ln{\frac{1}{R_*\varepsilon}}\int\limits_{D\cap\{|T_{x}(y-x)|> R_*\varepsilon\}}\zeta_\varepsilon \mathrm d y+C\\
			&\le \frac{\sqrt{\det K(x)}^{-1}}{2\pi}\ln{\frac{1}{R_*\varepsilon}}\int\limits_{D\cap\{|T_{x}(y-x)|> R_*\varepsilon\}}\zeta_\varepsilon \mathrm d y+C,
		\end{split}
	\end{equation}
	and
	\begin{equation}\label{4-43}
		\begin{split}
			B_2 &=\int_{D\cap\{|T_{x}(y-x)|\leq R_*\varepsilon\}}G_{K}(x,y)\zeta_\varepsilon(y)\mathrm d y\\
			&\le \frac{\sqrt{\det K(x)}^{-1}+O\left (\varepsilon^\frac{1}{2}\right )}{2\pi}\int\limits_{D\cap\{|T_{x}(y-x)|\leq R_*\varepsilon\}}\left (\ln\frac{1}{|T_{x}(y-x)|}+O\left (\varepsilon^{\frac{1}{2}}\right )\right )\zeta_\varepsilon(y)\mathrm d y+C\\
			&\le \frac{\sqrt{\det K(x)}^{-1}+O\left (\varepsilon^\frac{1}{2}\right )}{2\pi}\int\limits_{ |z|\leq R_*\varepsilon }\left (\ln\frac{1}{|z|}+O\left (\varepsilon^{\frac{1}{2}}\right )\right )\zeta_\varepsilon\left (T_{x}^{-1}z+x\right )\sqrt{\det K(x)}dz +C\\
			&\le \frac{\sqrt{\det K(x)}^{-1}+O\left (\varepsilon^\frac{1}{2}\right )}{2\pi}\ln\frac{1}{\varepsilon}\int\limits_{ |z|\leq R_*\varepsilon }\zeta_\varepsilon\left (T_{x}^{-1}z+x\right )\sqrt{\det K(x)}dz +C\\
			& \le \frac{\sqrt{\det K(x)}^{-1}}{2\pi}\ln\frac{1}{\varepsilon}\int\limits_{D\cap\{|T_{x}(y-x)|\leq R_*\varepsilon\}}\zeta_\varepsilon \mathrm d y+C.
		\end{split}
	\end{equation}
	Taking \eqref{4-42} and \eqref{4-43} into \eqref{4-40}, we get
	\begin{equation}\label{4-44}
		\begin{split}
			&\frac{\sqrt{\det K(x)}^{-1}}{2\pi}\ln{\frac{1}{R_*\varepsilon}}\int\limits_{\{|T_{x}(y-x)|> R_*\varepsilon\}}\zeta_\varepsilon \mathrm d y+\frac{\sqrt{\det K(x)}^{-1}}{2\pi}\ln\frac{1}{\varepsilon}\int\limits_{\{|T_{x}(y-x)|\leq R_*\varepsilon\}}\zeta_\varepsilon \mathrm d y\\
			&\ \ \ge \frac{\alpha|x|^2}{2}\ln{\frac{1}{\varepsilon}}+\mu_\ep-C .
		\end{split}
	\end{equation}
	On the other hand, by Lemma $\ref{lm4-5}$, one has
	\begin{equation}\label{4-45}
		\mu_\varepsilon\ge\left(\frac{\kappa}{2\pi\sqrt{\det K(x)}}-\frac{  \alpha |x|^2}{2}\right)\ln{\frac{1}{\varepsilon}}-C.
	\end{equation}
	Combining $\eqref{4-44}$ and $\eqref{4-45}$, we obtain
	\begin{equation*}
		\frac{\kappa}{2\pi}\ln\frac{1}{\varepsilon} \le \frac{1}{2\pi}\ln{\frac{1}{R_*\varepsilon}}\int\limits_{\{|T_{x}(y-x)|> R_*\varepsilon\}}\zeta_\varepsilon \mathrm d y+\frac{1}{2\pi}\ln\frac{1}{\varepsilon}\int\limits_{\{|T_{x}(y-x)|\leq R_*\varepsilon\}}\zeta_\varepsilon \mathrm d y+C,
	\end{equation*}
	which implies that
	\begin{equation*}
		\int\limits_{\{|T_{x}(y-x)|\leq R_*\varepsilon\}}\zeta_\varepsilon \mathrm d y\ge \kappa\left (1-\frac{C}{\ln R_*}\right ).
	\end{equation*}
	Taking $R_*>1$ such that $C(\ln R_*)^{-1}<1/2$, we obtain
	\begin{equation*}
		\int\limits_{D\cap{B_{R_*\varepsilon}(x)}}\zeta_\varepsilon \mathrm d y\ge\int\limits_{\{|T_{x}(y-x)|\leq R_*\varepsilon\}}\zeta_\varepsilon \mathrm d y> \frac{\kappa}{2}.
	\end{equation*}
	Taking $r_2=2R_*$, we complete the proof of Lemma \ref{lm4-12}.	
\end{proof}
Lemma \ref{4-12} implies that diameter of the support set of $ \zeta_\varepsilon $ is the order  of $ \varepsilon $. By using the scaling method and theory of rearrangement function, we can further analyse the asymptotic properties of the scaled function  of  $ \zeta_\varepsilon $.

We define the center of $  \zeta_\varepsilon $
\begin{equation}\label{224}
X_\varepsilon=\frac{1}{\kappa}\int_D x\zeta_\varepsilon(x) dx.
\end{equation}
It follows from Lemmas \ref{lm4-6}, \ref{lm4-8} and \ref{lm4-9} that $X_\varepsilon \to (r_*,0)$ as $\varepsilon \to 0^+$. Let $g_{\varepsilon}$ 
be the scaled function of $ \zeta_{\varepsilon} $ defined  by
\begin{equation}\label{225}
g_{\varepsilon}(x)={\varepsilon^2}\zeta_{\varepsilon}(X_\varepsilon+T^{-1}_{X_\varepsilon}\varepsilon x),\ \ \ \ \ x\in D_\ep,
\end{equation}
where $ T_{X_\varepsilon} $ is the positive-definite matrix such that \eqref{matrix T} holds, i.e.,
\begin{equation*} 
T_{X_\varepsilon}^{-1}\left (T_{X_\varepsilon}^{-1}\right )^{t}=K\left (X_\varepsilon\right ),
\end{equation*} 
and $ D_\ep=\left \{x\in\mathbb{R}^2\mid X_\varepsilon+T^{-1}_{X_\varepsilon}\varepsilon x\in D\right \} $. By Lemmas \ref{lm4-2}, \ref{lm4-11} and \ref{lm4-12}, $ \int_{D}g_{\varepsilon}(x)dx =\sqrt{\det K(X_\varepsilon)}^{-1}\kappa$,  $ supp(g_{\varepsilon})\subseteq B_{R_1}(0) $ and $ ||g_{\varepsilon}||_{L^\infty}\leq \Lambda_0 $ for some constants $R_1, \Lambda_0>0$.  
Then up to a subsequence, there must be a function  $ g  \in L^\infty(B_{R_1}(0))   $ such that as $\varepsilon \to 0^+$ 
$$ g_{\varepsilon} \to g \ \ \text{ weakly ~in}\ \ L^2(B_{R_1}(0)).$$ 
 
The following lemma gives the asymptotic behavior of  $g_{\varepsilon}$ as $\varepsilon\to0$.

\begin{lemma}\label{le12}
$ g $ is a  radially symmetric nonincreasing function, i.e., $ g(x)=g(|x|) $ for any $ x\in\mathbb{R}^2 $.
\end{lemma}

\begin{proof}
We denote $g^\bigtriangleup_{\varepsilon}$ the  radially symmetric nonincreasing Lebesgue-rearrangement of $g_\varepsilon$ centered on 0, i.e., $g^\bigtriangleup_{\varepsilon}$ is  radially symmetric nonincreasing and for any $ \mu>0 $
\begin{equation}\label{rearr ineq} 
\left |\left\lbrace g^\bigtriangleup_{\varepsilon}(x)>\mu\right\rbrace \right |=\left |\left\lbrace g_{\varepsilon}(x)>\mu\right\rbrace \right |.
\end{equation}
Then $ \int_{D}g^\bigtriangleup_{\varepsilon}(x)dx=\int_{D}g_{\varepsilon}(x)dx =\sqrt{\det K(X_\varepsilon)}^{-1}\kappa$, $ ||g^\bigtriangleup_{\varepsilon}||_{L^\infty}\leq \Lambda_0 $ and there must be a function  $ h  \in L^\infty(B_{R_1}(0))   $ such that as $\varepsilon \to 0^+$ 
$$ g^\bigtriangleup_\varepsilon \to h \ \ \text{ weakly ~in}\ \ L^2(B_{R_1}(0)).$$ 
Moreover, $ h $ is a  radially symmetric nonincreasing function.

One the one hand, by virtue of Riesz' rearrangement inequality, we have
	\begin{equation*}
	\iint\limits_{B_{R_1}(0)\times B_{R_1}(0)}\ln\frac{1}{|x-x'|}g_{\varepsilon}(x)g_{\varepsilon}(x')dxdx' 
	\le \iint\limits_{B_{R_1}(0)\times B_{R_1}(0)}\ln\frac{1}{|x-x'|}g^\bigtriangleup_{\varepsilon}(x)g^\bigtriangleup_{\varepsilon}(x')dxdx'.
	\end{equation*}
	Hence
	\begin{equation}\label{226}
	\iint\limits_{B_{R_1}(0)\times B_{R_1}(0)}\ln\frac{1}{|x-x'|}g(x)g(x')dxdx'\le \iint\limits_{B_{R_1}(0)\times B_{R_1}(0)}\ln\frac{1}{|x-x'|}h(x)h(x')dxdx'.
	\end{equation}
On the other hand, we define test functions $\tilde{\zeta}_\varepsilon$  
	\[
	\tilde{\zeta}_\varepsilon(x)=\left\{
	\begin{array}{lll}
	\varepsilon^{-2}g^\bigtriangleup_{\varepsilon}\left (\varepsilon^{-1}T_{X_\varepsilon}(x-X_\varepsilon)\right ) &    \text{if} & x\in B_{R_1\varepsilon}(X_\varepsilon), \\
	0                  &    \text{if} & x\in D\backslash B_{R_1\varepsilon}(X_\varepsilon).
	\end{array}
	\right.
	\]
Then $ \int_{D}\tilde{\zeta}_\varepsilon(x)dx=\kappa $ and thus $ \tilde{\zeta}_\varepsilon\in \mathcal{M}_{\ep} $. This implies $ 	E_\varepsilon(\zeta_{\varepsilon})\geq 	E_\varepsilon(\tilde{\zeta}_{\varepsilon}) $.

From \eqref{4-8}, \eqref{4-10} and Lemma \ref{lm4-1}, we have
	\begin{equation}\label{227}
\begin{split}
E_\varepsilon(\zeta_{\varepsilon}) =&\frac{1}{4\pi}\iint\limits_{B_{R_1\varepsilon} \times B_{R_1\varepsilon} }\left( \frac{\sqrt{\det K(x)}^{-1}+\sqrt{\det K(y)}^{-1}}{2}\ln\frac{1}{\left |\frac{T_x+T_y}{2}\left( x-y\right) \right |}+H_0(x,y)\right)\zeta_{\varepsilon}(x)\zeta_{\varepsilon}(y)dxdy \\
&-\frac{\bar{\alpha}}{2}\int_D|x|^2\zeta_{\varepsilon}(x)dx-\frac{1}{\varepsilon^2}\int_{D}J(|x|, \varepsilon^2\zeta_{\varepsilon}(x))dx.
\end{split}
\end{equation}
Using transformation of coordinates $ x=X_\varepsilon+T^{-1}_{X_\varepsilon}\varepsilon x', y=X_\varepsilon+T^{-1}_{X_\varepsilon}\varepsilon y'  $ and Lemma \ref{lm4-12}, one computes directly that
\begin{equation*}
\begin{split}
\frac{1}{4\pi}&\iint\limits_{B_{R_1\varepsilon} \times B_{R_1\varepsilon} }\left( \frac{\sqrt{\det K(x)}^{-1}+\sqrt{\det K(y)}^{-1}}{2}\ln\frac{1}{\left |\frac{T_x+T_y}{2}\left( x-y\right) \right |}+H_0(x,y)\right)\zeta_{\varepsilon}(x)\zeta_{\varepsilon}(y)dxdy\\
=&\frac{\left(\det T_{X_\varepsilon}^{-1} \right)^2}{4\pi}\iint\limits_{B_{R_1}(0)\times B_{R_1}(0)}\left(\frac{\sqrt{\det K(x)}^{-1}+\sqrt{\det K(y)}^{-1}}{2}\ln\frac{1}{\varepsilon\left |\frac{T_x+T_y}{2}T_{X_\varepsilon}^{-1}\left( x'-y'\right) \right |}+H_0(x,y)\right)\\
&\cdot g_{\varepsilon}(x')g_{\varepsilon}(y') dx'dy' \\
=&\frac{\sqrt{\det K(X_\varepsilon)}}{4\pi}\iint\limits_{B_{R_1}(0)\times B_{R_1}(0)}\left( \ln\frac{1}{\varepsilon\left |  \frac{T_x+T_y}{2}T_{X_\varepsilon}^{-1}\left( x'-y'\right)  \right |}+H_0(X_\varepsilon,X_\varepsilon)\right)g_{\varepsilon}(x')g_{\varepsilon}(y') dx'dy'+o(1)\\
=&\frac{\sqrt{\det K(X_\varepsilon)}}{4\pi}\iint\limits_{B_{R_1}(0)\times B_{R_1}(0)}  \ln\frac{1}{ \left |  \frac{T_x+T_y}{2}T_{X_\varepsilon}^{-1}\left( x'-y'\right)  \right |} g_{\varepsilon}(x')g_{\varepsilon}(y') dx'dy'\\
&+\frac{\sqrt{\det K(X_\varepsilon)}}{4\pi}\left(\int_{B_{R_1}(0)} g_{\varepsilon}(x') dx'\right)^2 \ln\frac{1}{\varepsilon}+\frac{\sqrt{\det K(X_\varepsilon)}H_0(X_\varepsilon,X_\varepsilon)}{4\pi}\left( \int_{B_{R_1}(0)}g_{\varepsilon}(x')dx'\right)^2 +o(1), 
\end{split}
\end{equation*}
\begin{equation*}
\begin{split}
-\frac{\bar{\alpha}}{2} \int_D|x|^2\zeta_{\varepsilon}(x)dx=-\frac{\bar{\alpha}\kappa}{2}|X_\varepsilon|^2+O\left (\varepsilon\ln\frac{1}{\varepsilon}\right ),
\end{split}
\end{equation*}
and
\begin{equation*}
	\begin{split}
\frac{1}{\varepsilon^2}\int_{D}J(|x|, \varepsilon^2\zeta_{\varepsilon}(x))dx =&\frac{1}{\varepsilon^2}\int_{D}\frac{(h^2+|x|^2)^2\left(\varepsilon^2\zeta_{\varepsilon}(x)-\frac{h^2\bar\alpha a\ep}{h^2+|x|^2}-b\right)_+^2}{4h^2a\ep+2a^2(h^2+|x|^2)}dx \\
=&  \det T_{X_\varepsilon}^{-1}\cdot\int_{B_{R_1}(0)}\frac{(h^2+|X_\varepsilon|^2)^2\left(g_{\varepsilon}(x')-\frac{h^2\bar\alpha a\ep}{h^2+|X_\varepsilon|^2}-b\right)_+^2}{4h^2a\ep+2a^2(h^2+|X_\varepsilon|^2)}dx'+O(\varepsilon) \\ 
=& \frac{\det T_{X_\varepsilon}^{-1}\cdot(h^2+|X_\varepsilon|^2)^2}{4h^2a\ep+2a^2(h^2+|X_\varepsilon|^2)}\int_{B_{R_1}(0)}\left(g_{\varepsilon}(x')-b\right)_+^2dx'+O\left (\varepsilon\ln\frac{1}{\varepsilon}\right ).
	\end{split}
	\end{equation*}
Taking these into \eqref{227}, we get
	\begin{equation*}
	\begin{split}
	E_\varepsilon( \zeta_{\varepsilon}) =&\frac{\sqrt{\det K(X_\varepsilon)}}{4\pi}\iint\limits_{B_{R_1}(0)\times B_{R_1}(0)}  \ln\frac{1}{ \left |  \frac{T_x+T_y}{2}T_{X_\varepsilon}^{-1}\left( x'-y'\right)  \right |} g_{\varepsilon}(x')g_{\varepsilon}(y') dx'dy'\\
	&+\frac{\sqrt{\det K(X_\varepsilon)}}{4\pi}\left(\int_{B_{R_1}(0)} g_{\varepsilon}(x') dx'\right)^2 \ln\frac{1}{\varepsilon}+\frac{\sqrt{\det K(X_\varepsilon)}H_0(X_\varepsilon,X_\varepsilon)}{4\pi}\left( \int_{B_{R_1}(0)}g_{\varepsilon}(x')dx'\right)^2\\
	&-\frac{\bar{\alpha}\kappa}{2}|X_\varepsilon|^2 -\frac{\det T_{X_\varepsilon}^{-1}\cdot(h^2+|X_\varepsilon|^2)^2}{4h^2a\ep+2a^2(h^2+|X_\varepsilon|^2)}\int_{B_{R_1}(0)}\left(g_{\varepsilon}(x')-b\right)_+^2dx'+o(1).
	\end{split}
	\end{equation*}
Similarly, we have
\begin{equation*} 
\begin{split}
	E_\varepsilon( \tilde{\zeta}_{\varepsilon}) =&\frac{\sqrt{\det K(X_\varepsilon)}}{4\pi}\iint\limits_{B_{R_1}(0)\times B_{R_1}(0)}  \ln\frac{1}{ \left |  \frac{T_x+T_y}{2}T_{X_\varepsilon}^{-1}\left( x'-y'\right)  \right |} g^\bigtriangleup_{\varepsilon}(x')g^\bigtriangleup_{\varepsilon}(y') dx'dy'\\
&+\frac{\sqrt{\det K(X_\varepsilon)}}{4\pi}\left(\int_{B_{R_1}(0)} g^\bigtriangleup_{\varepsilon}(x') dx'\right)^2 \ln\frac{1}{\varepsilon}+\frac{\sqrt{\det K(X_\varepsilon)}H_0(X_\varepsilon,X_\varepsilon)}{4\pi}\left( \int_{B_{R_1}(0)}g^\bigtriangleup_{\varepsilon}(x')dx'\right)^2\\
&-\frac{\bar{\alpha}\kappa}{2}|X_\varepsilon|^2-\frac{\det T_{X_\varepsilon}^{-1}\cdot(h^2+|X_\varepsilon|^2)^2}{4h^2a\ep+2a^2(h^2+|X_\varepsilon|^2)}\int_{B_{R_1}(0)}\left(g^\bigtriangleup_{\varepsilon}(x')-b\right)_+^2dx'+o(1).
\end{split}
\end{equation*}

Using \eqref{rearr ineq} and $E_\varepsilon(\tilde{\zeta}_{\varepsilon})\le E_\varepsilon(\zeta_{\varepsilon})$, we obtain
\begin{equation*} 
\begin{split}
&\iint\limits_{B_{R_1}(0)\times B_{R_1}(0)}  \ln\frac{1}{ \left |  \frac{T_x+T_y}{2}T_{X_\varepsilon}^{-1}\left( x'-y'\right)  \right |} g^\bigtriangleup_{\varepsilon}(x')g^\bigtriangleup_{\varepsilon}(y') dx'dy'\\
\leq& \iint\limits_{B_{R_1}(0)\times B_{R_1}(0)}  \ln\frac{1}{ \left |  \frac{T_x+T_y}{2}T_{X_\varepsilon}^{-1}\left( x'-y'\right)  \right |} g_{\varepsilon}(x')g_{\varepsilon}(y') dx'dy'+o(1), 
\end{split}
\end{equation*}
where $ x=X_\varepsilon+T^{-1}_{X_\varepsilon}\varepsilon x', y=X_\varepsilon+T^{-1}_{X_\varepsilon}\varepsilon y'  $. Passing $ \varepsilon $   to 0, we conclude that
	\begin{equation*}
	\iint\limits_{B_{R_1}(0)\times B_{R_1}(0)}\ln\frac{1}{|x-x'|}h(x)h(x')dxdx'\le \iint\limits_{B_{R_1}(0)\times B_{R_1}(0)}\ln\frac{1}{|x-x'|}g(x)g(x')dxdx'.
	\end{equation*}
	which together with $\eqref{226}$ yield to
	\begin{equation*}
	\iint\limits_{B_{R_1}(0)\times B_{R_1}(0)}\ln\frac{1}{|x-x'|}g(x)g(x')dxdx'= \iint\limits_{B_{R_1}(0)\times B_{R_1}(0)}\ln\frac{1}{|x-x'|}h(x)h(x')dxdx'.
	\end{equation*}
	By Lemma 3.2 of \cite{BG}, we know that there exists a translation $\mathcal{T}$ in $\mathbb{R}^2$ such that $\mathcal{T}g=h$. Note also that
	\begin{equation*}
	\int_{B_{R_1}(0)}xg(x)dx=\int_{B_{R_1}(0)}xh(x)dx=0.
	\end{equation*}
	Thus $ \mathcal{T} $ is an identity mapping and $g=h$. The proof is thus complete.
\end{proof}

Note that $ g_\ep\to g $ weakly in $ L^2(B_{R_1}(0)) $, where $ g $ is radially symmetric. We define $ \tilde{\psi_{\ep}} $ the scaled function of the stream function $ \psi_{\ep} $ by 
\begin{equation}\label{228}
\tilde{\psi_{\ep}}(y)=\psi_{\ep}\left (X_\ep+T_{X_\ep}^{-1}\ep y\right ), \ \ \ y\in D_\ep.
\end{equation}
From Lemmas \ref{lm4-2}, \ref{lm4-12} and the definition of $ g_\ep $ in \eqref{225}, we know that $ supp((\tilde{\psi_{\ep}})_+)=supp(g_\ep)\subseteq B_{R_1}(0). $
Moreover, one has
\begin{equation*} 
\begin{split}
-\nabla_y&\cdot\left(T_{X_\ep}K\left (X_\ep+T_{X_\ep}^{-1}\ep y\right )T_{X_\ep}^t\nabla_y \tilde{\psi_{\ep}}(y)\right)\\ =&-\varepsilon^2\nabla_x\cdot\left(K\left (x\right )\nabla_x \psi_{\ep}(x)\right)\left (X_\ep+T_{X_\ep}^{-1}\ep y\right )\\
=&\ep^2\zeta_{\ep}\left (X_\ep+T_{X_\ep}^{-1}\ep y\right )+\frac{\bar{\alpha}\ep^2}{2}\mathcal{L}_K|x|^2\left (X_\ep+T_{X_\ep}^{-1}\ep y\right )\\
=&g_\ep(y)+\frac{\bar{\alpha}\ep^2}{2}\mathcal{L}_K|x|^2\left (X_\ep+T_{X_\ep}^{-1}\ep y\right )\to g\ \ \text{ weakly ~in}\ \ L^2(B_{R_1}(0))\ \ \text{as}\ \ep\to0.
\end{split}
\end{equation*}
By classical regularity estimates for elliptic equations (see \cite{GT}), there exists a function $ \psi^*\in C^{1,\gamma}_{loc}(\mathbb{R}^2) $ for $ \gamma\in(0,1) $ such that as $ \ep\to0 $
\begin{equation}\label{229}
\tilde{\psi_{\ep}}\to \psi^*\ \ \text{in}\ C^1_{loc}(\mathbb{R}^2).
\end{equation}
Moreover, $ supp\left( \left (\psi^*\right )_+\right) \subseteq B_{R_1}(0)  $ and 
\begin{equation*} 
-\Delta \psi^*=g,\  \ \ \text{in}\ \mathbb{R}^2.
\end{equation*}
Using the moving plane method, we know that $ \left (\psi^*\right )_+ $ is radially symmetric in $ B_{R_1}(0) $. 

As a result, we have
\begin{lemma}\label{lm4-13}
There holds for $ i=1,2 $
\begin{equation*} 
\frac{1}{\varepsilon}\int_{D}\partial_{x_i}(\psi_{\ep}(x))_+dx\to 0\ \ \ \text{as}\ \ \ep\to0.
\end{equation*}
\end{lemma}
\begin{proof}
Using variable substitution and \eqref{228}, we have
\begin{equation*} 
\begin{split}
\frac{1}{\varepsilon}\int_{D}\partial_{x_i}(\psi_{\ep}(x))_+dx=&\sum_{j=1}^2\left( T_{X_\ep}\right)_{i,j}\int_{B_{R_1}(0)}\partial_{y_j}(\tilde{\psi}_{\ep}(y))_+\cdot \det T_{X_\ep}^{-1}dy.
\end{split}
\end{equation*}
Here $ \left( T_{X_\ep}\right)_{i,j} $ stands for row $ i $ and column $ j $ of the matrix $ T_{X_\ep} $. Since \eqref{229} holds and $ \left (\psi^*\right )_+ $ is radially symmetric, we get
\begin{equation*} 
\int_{B_{R_1}(0)}\partial_{y_j}(\tilde{\psi}_{\ep}(y))_+dy\to \int_{B_{R_1}(0)}\partial_{y_j}(\psi^*(y))_+dy=0\ \ \ \text{as}\ \ \ep\to0.
\end{equation*}
Combining the above two formulas, we complete the proof.

\end{proof}
\subsection{Proof of Theorem \ref{thm2}}
Having made all the preparation, we are now able to give proof of Theorem \ref{thm2}. Let $h>0, \kappa>0, R>0$  be fixed numbers. For any $ r_*\in (0,R) $,  let
\begin{equation*}
\alpha=\frac{\kappa}{4\pi h\sqrt{h^2+r_*^2}}.
\end{equation*}
We first construct a family of  vorticity field $ \mathbf{w}_\varepsilon $ in terms of $ \zeta_{\varepsilon} $.  Let $ a>0, b\geq 0 $ be two prescribed constants and $ \bar{\alpha}=\alpha\ln\frac{1}{\ep} $.
From Lemmas \ref{lm4-2} and \ref{lm4-11}, we know that for any $ \varepsilon \in(0,\varepsilon_0)  $ all maximizers of $E_\ep$ over $\mathcal{M}_{\ep,\kappa}$ are of the form (see \eqref{4-38})
	\begin{equation*}
	\begin{split}
\zeta_{\ep}=\frac{1}{\ep^2}\frac{2h^2a\ep +a^2(h^2+|x|^2)}{(h^2+|x|^2)^2 }(\psi_{\ep})_+ +\frac{1}{\ep^2}\left(\frac{h^2\bar \alpha a\ep}{h^2+|x|^2}+b\right)\textbf{1}_{\{ \psi_{\ep}>0\} }	 \ \ a.e.\  \text{in}\  D,
\end{split}
\end{equation*}
where $ 	\psi_{\ep}=\mathcal{G}_K\zeta_{\ep}-\frac{\alpha |x|^2}{2}\ln{\frac{1}{\ep}}-\mu_{\ep}  $ and $\mu_{\varepsilon} $ is a  Lagrange multiplier determined by $\zeta_{\varepsilon}$. If we denote two functions 
\begin{equation*} 
F_1(s)=\ep^{-1}a (s-\mu_\ep)_+,\ \ \   F_2(s)=\ep^{-2}(b+\bar\alpha a \ep) \textbf{1}_{\{s-\mu_\ep>0\}},\ \ \  s\in \mathbb R,
\end{equation*}
then $ \zeta_{\ep} $ has the profile  
\begin{equation*} 
\begin{split}
\zeta_{\ep}=&\frac{2h^2F_1\left(\mathcal{G}_K\zeta_{\ep}-\frac{\bar\alpha}{2}|x|^2 \right)}{(h^2+|x|^2)^2}+\frac{F_1\left(\mathcal{G}_K\zeta_{\ep}-\frac{\bar\alpha}{2}|x|^2 \right)F_1'\left(\mathcal{G}_K\zeta_{\ep}-\frac{\bar\alpha}{2}|x|^2 \right)}{h^2+|x|^2}\\
&-\frac{\bar\alpha |x|^2 F_1'\left(\mathcal{G}_K\zeta_{\ep}-\frac{\bar\alpha}{2}|x|^2 \right)}{h^2+|x|^2}+F_2\left(\mathcal{G}_K\zeta_{\ep}-\frac{\bar\alpha}{2}|x|^2 \right),\ \  \text{in}\ \  D,
\end{split}
\end{equation*}
which coincides with \eqref{3-10}. We define function pairs $ (\varPhi_\ep, V_\ep, W_\ep, w_\ep, v_\ep, \varphi_\ep) $ in $ D $ such that $ \varPhi_\ep=\mathcal{G}_K\zeta_{\ep} $, $ V_\ep $ satisfies \eqref{3-3}, $ W_\ep $ satisfies \eqref{3-6} and $ (w_\ep, v_\ep, \varphi_\ep) $ satisfies \eqref{3-1}. From the deduction of \eqref{3-10} in section 3, we know that $ (w_\ep, v_\ep) $ is a rotating-invariant solution pair of the system \eqref{1-23} with the angular velocity $ \bar{\alpha} $ and the  stream function $ \varphi_\ep $. 

By Lemmas \ref{lm4-6}, \ref{lm4-8}, \ref{lm4-9} and \ref{lm4-12}, the support set $ \bar{A}_\varepsilon $ of the maximizer $ \zeta_\varepsilon $ concentrates near $ (r_*,0) $ up to a rotation as $ \varepsilon\to0 $ with  $\int_{D}\zeta_\varepsilon dx=\kappa $. Define for any $ (x_1,x_2,x_3)\in B_{R}(0)\times\mathbb{R}, t\in\mathbb{R}  $
\begin{equation}\label{503}
\mathbf{w}_\ep(x_1,x_2,x_3,t)=\frac{w_3^\ep(x_1,x_2,x_3,t)}{h}\vec{\xi}+\frac{1}{h}\left (\partial_{x_2}v_{\vec{\xi}}^\ep(x_1,x_2,x_3,t), -\partial_{x_1}v_{\vec{\xi}}^\ep(x_1,x_2,x_3,t),0\right )^t,
\end{equation}
where $ w_3^\varepsilon(x_1,x_2,x_3,t) $ is a helical function satisfying
\begin{equation*}
w_3^\varepsilon(x_1,x_2,0,t)=w_\varepsilon(x_1,x_2,t)
\end{equation*}
and $ v_{\vec{\xi}}^\ep(x_1,x_2,x_3,t) $ is a helical function satisfying
\begin{equation*} 
v_{\vec{\xi}}^\ep(x_1,x_2,0,t)=v_\varepsilon(x_1,x_2,t).
\end{equation*}
From Theorem \ref{thm1}, we know that $ \mathbf{w}_\ep $ is a helical vorticity field of \eqref{1-2} with non-zero helical swirl $ v_{\vec{\xi}}^\ep $. Note that the support set of $ \mathbf{w}_\ep $ is a helical domain whose cross-section is $ \bar{A}_\varepsilon $. Inside this helical region, the helical swirl of $ \mathbf{w}_\ep  $ is non-zero, while outside the helical swirl is identical to zero. Moreover,  the cross-section of  $  supp(\mathbf{w}_\ep) $ tends asymptotically to a rotating point vortex $ \bar{R}_{\alpha|\ln\ep|}((r_*,0)) $ up to a rotation as $ \varepsilon\to0 $.

Define $ P(\tau) $ the intersection point of the curve  parameterized by \eqref{1007} and the $ x_1Ox_2 $ plane. 
Taking $ s=\frac{b_1\tau}{h} $ into \eqref{1007}, one computes directly that  $ P(\tau) $ satisfies a 2D point vortex model
\begin{equation*}
\begin{split}
P(\tau)=\bar{R}_{\alpha'\tau}((r_*,0)),
\end{split}
\end{equation*}
where $ \alpha'=\frac{1}{\sqrt{h^2+r_*^2}}\left( a_1+\frac{b_1}{h}\right)=\frac{\kappa}{4\pi h\sqrt{h^2+r_*^2}}=\alpha $. Thus by our construction, we  readily check that the support set of $ \mathbf{w}_\varepsilon(x_1,x_2,0,|\ln\varepsilon|^{-1}\tau) $ defined by  \eqref{503} concentrates near $ P(\tau) $ as $ \varepsilon\to 0 $, which implies that, the support set of  $ \mathbf{w}_\varepsilon $ is a topological traveling-rotating helical tube that does not change form and concentrates near $ \gamma(\tau) $ \eqref{1007} in sense   that for all $ \tau $,
\begin{equation*}
\text{dist}\left (\partial(supp(\mathbf{w}_\varepsilon(\cdot, |\ln\varepsilon|^{-1}\tau))), \gamma(\tau)\right ) \to 0,\ \ \text{as}\ \varepsilon\to0.
\end{equation*}


It remains to prove that the vorticity field $ \mathbf{w}_\varepsilon $ tends asymptotically to \eqref{1007} in distributional sense, i.e., $ 	\mathbf{w}_\varepsilon(\cdot,|\ln\varepsilon|^{-1}\tau)\to \kappa \delta_{\gamma(\tau)}\mathbf{t}_{\gamma(\tau)}  $ as $ \varepsilon\to0 $. Note that \eqref{503} holds. We will give the proof in two steps. 

\textbf{Step 1.} We calculate $ \frac{w_3^\ep}{h}\vec{\xi} $. By \eqref{2-1}, we get
\begin{equation}\label{505} 
	w_\varepsilon=\zeta_\varepsilon-\frac{2h^2 v_{\ep}}{|\vec{\xi}|^4}-\frac{x_1 \partial_{x_1}v_{\ep}+x_2\partial_{x_2}v_{\ep}}{|\vec{\xi}|^2}.
\end{equation}
By the definition of the constraint set \eqref{4-9} and \eqref{4-38}, we have
\begin{equation}\label{504} 
\kappa=\int_{D}\zeta_\varepsilon dx=\int_{D}\frac{1}{\ep^2}\frac{2h^2a\ep +a^2(h^2+|x|^2)}{(h^2+|x|^2)^2 }(\psi_{\ep})_+ +\frac{1}{\ep^2}\left(\frac{h^2\bar \alpha a\ep}{h^2+|x|^2}+b\right)\textbf{1}_{\{ \psi_{\ep}>0\} } dx.
\end{equation}
So by \eqref{504}, the choice of $ F_1(\cdot) $ and \eqref{3-3}, we obtain
\begin{equation*} 
\int_{D}\frac{2h^2 v_{\ep}}{|\vec{\xi}|^4}dx=\int_{D}\frac{2h^2 a(\psi_{\ep})_+}{\varepsilon(h^2+|x|^2)^2}dx=O(\varepsilon).
\end{equation*}
From Lemma \ref{lm4-13}, we get 
\begin{equation*} 
\int_{D}\frac{x_1 \partial_{x_1}v_{\ep}}{|\vec{\xi}|^2}dx=\int_{D}\frac{ax_1 \partial_{x_1}(\psi_{\ep})_+}{\varepsilon(h^2+|x|^2)}dx=\frac{aX_{\ep,1} }{\varepsilon(h^2+|X_\ep|^2)}\int_{D}\partial_{x_1}(\psi_{\ep})_+dx+o(1)=o(1).
\end{equation*}
Similarly, we have $ \int_{D}\frac{x_2 \partial_{x_2}v_{\ep}}{|\vec{\xi}|^2}dx=o(1). $ Taking these into \eqref{505}, we have
\begin{equation}\label{506}
\int_{D}w_\varepsilon dx=\kappa+o(1),
\end{equation}
which implies that in distributional sense as $ \ep\to0 $,
\begin{equation}\label{507}
\frac{w_3^\ep(\cdot,|\ln\varepsilon|^{-1}\tau)}{h}\vec{\xi}\to \kappa \delta_{\gamma(\tau)}\mathbf{t}_{\gamma(\tau)}.
\end{equation}

\textbf{Step 2.} We calculate  $ \frac{1}{h}\left (\partial_{x_2}v_{\vec{\xi}}^\ep(\cdot,|\ln\varepsilon|^{-1}\tau), -\partial_{x_1}v_{\vec{\xi}}^\ep(\cdot,|\ln\varepsilon|^{-1}\tau),0\right )^t $.
Note that from Lemma \ref{lm4-13}, one has
\begin{equation*}
\int_{D}\partial_{x_2}v_\ep dx=\frac{a}{\ep}\int_{D}\partial_{x_2}(\psi_{\ep})_+ dx=o(1),
\end{equation*}
which implies that $ \frac{1}{h}\left (\partial_{x_2}v_{\vec{\xi}}^\ep(\cdot,|\ln\varepsilon|^{-1}\tau), 0,0\right )^t\to 0 $ in distributional sense as $ \ep\to0 $. Similarly, we get $ \frac{1}{h}\left (0, -\partial_{x_1}v_{\vec{\xi}}^\ep(\cdot,|\ln\varepsilon|^{-1}\tau),0\right )^t\to 0 $.
So we have  as $ \ep\to0 $,
\begin{equation} \label{508}
\frac{1}{h}\left (\partial_{x_2}v_{\vec{\xi}}^\ep(\cdot,|\ln\varepsilon|^{-1}\tau), -\partial_{x_1}v_{\vec{\xi}}^\ep(\cdot,|\ln\varepsilon|^{-1}\tau),0\right )^t\to 0.
\end{equation}

Combining \eqref{507} and \eqref{508}, we conclude that  in distributional sense
\begin{equation*} 
\mathbf{w}_\varepsilon(\cdot,|\ln\varepsilon|^{-1}\tau)\to \kappa \delta_{\gamma(\tau)}\mathbf{t}_{\gamma(\tau)}\ \ \ \text{as}\ \ \ep\to0.
\end{equation*}
The proof of Theorem \ref{thm2} is thus complete.

\subsection*{Acknowledgments:}

\par
J. Wan was supported by NNSF of China (grants No. 12101045, 12471190). G. Qin was supported by China National Postdoctoral Program for Innovative Talents (No. BX20230009). 

\subsection*{Conflict of interest statement} On behalf of all authors, the corresponding author states that there is no conflict of interest.

\subsection*{Data availability statement} All data generated or analysed during this study are included in this published article  and its supplementary information files.

     \phantom{s}
     \thispagestyle{empty}

\end{document}